\documentclass[12pt]{article}

\usepackage{amsmath,euscript}
\usepackage{amssymb}
\usepackage{amsthm}
\usepackage{enumerate}
\usepackage{amsfonts}
\usepackage{amssymb,amsmath,array}

\usepackage{latexsym}

\usepackage{amscd}

\author{}

\title{}
\date{}

\input xy \xyoption{matrix} \xyoption{arrow}
\def\edge{\ar@{-}}
\def\plc{*{}+0}

\def\bigblacksquare{\vrule height 8pt depth 2pt width 10pt}
\def\mdblk{\save+<0ex,0ex>\drop{\bigblacksquare}\restore}

\newcommand{\gc}{ [\![}
\newcommand{\dc}{]\!]}
\newcommand{\ov}{\overline}

\newcommand{\rank}{\mathrm{rank}}

\newcommand{\mmpc}{\oc_{q}\left( \mathcal{M}_{m,p}(\bbC) \right)}
\newcommand{\pmmpc}{\oc\left( \mathcal{M}_{m,p}(\bbC) \right)}

\newcommand{\ia}{i,\alpha}

\newcommand{\oc}{\mathcal{O}}

\newcommand{\mc}{\mathcal{M}}
\newcommand{\mcw}{\mathcal{M}(w)}
\newcommand{\hc}{\mathcal{H}}

\newtheorem{theo}{Theorem}[section]
\newtheorem{prop}[theo]{Proposition}
\newtheorem{lem}[theo]{Lemma}
\newtheorem{cor}[theo]{Corollary}

\theoremstyle{definition}
\newtheorem{defi}[theo]{Definition}
\newtheorem{nota}[theo]{Notation}
\newtheorem{conv}[theo]{Convention}
\newtheorem{alg}[theo]{Algorithm}
\newtheorem{obs}[theo]{Observations}
\newtheorem{remark}[theo]{Remark}
\newtheorem{remarks}[theo]{Remarks}
\newtheorem{example}[theo]{Example}
\newtheorem{examples}[theo]{Examples}
\newtheorem{subsec}[theo]{}

\numberwithin{equation}{section}

\newcommand{\bbN}{\mathbb{N}} 
\newcommand{\bbZ}{\mathbb{Z}} 
\newcommand{\bbQ}{\mathbb{Q}} 
\newcommand{\bbR}{\mathbb{R}} 
\newcommand{\bbC}{\mathbb{C}} 
\newcommand{\Mmpc}{\mc_{m,p}(\bbC)}
\newcommand{\diag}{\mbox{diag}}
\newcommand{\lc}{\mathcal{L}}
\newcommand{\wop}{w_\circ^p}
\newcommand{\wom}{w_\circ^m}
\newcommand{\woN}{w_\circ^N}
\newcommand{\Ecirc}{E^\circ}
\newcommand{\MmpK}{\mc_{m,p}(K)}
\newcommand{\onem}{\gc 1,m \dc}
\newcommand{\onep}{\gc 1,p \dc}
\newcommand{\onemxp}{\onem\times\onep}

\newcommand{\onel}{\gc 1,l \dc}
\newcommand{\mmptnn}{\mc^{\geq 0}_{m,p}(\bbR)}
\newcommand{\xbar}{\overline{x}}
\newcommand{\Xbar}{\overline{X}}
\newcommand{\Ybar}{\overline{Y}}
\newcommand{\transp}{^{\mathrm{tr}}}

\newcommand{\tr}{t^{(r)}}
\newcommand{\trp}{t^{(r^+)}}
\newcommand{\tinv}{t^{-1}}
\renewcommand{\ia}{{i,\alpha}}
\newcommand{\ld}{{l,\delta}}
\newcommand{\ib}{{i,\beta}}
\newcommand{\jb}{{j,\beta}}
\newcommand{\ja}{{j,\alpha}}
\newcommand{\idl}{{i,\delta}}
\newcommand{\la}{{l,\alpha}}
\newcommand{\jd}{{j,\delta}}
\newcommand{\lb}{{l,\beta}}
\newcommand{\kg}{{k,\gamma}}
\newcommand{\ig}{{i,\gamma}}
\newcommand{\ka}{{k,\alpha}}
\newcommand{\kb}{{k,\beta}}
\newcommand{\jg}{{j,\gamma}}

\newcommand{\Gr}{\mathrm{Gr}}
\newcommand{\Grtnn}{\Gr^{\mathrm{tnn}}_{m,n}}
\newcommand{\GrmnR}{\Gr_{m,n}(\bbR)} 
\newcommand{\GrmnC}{\Gr_{m,n}(\bbC)} 
\newcommand{\pc}{\mathcal{P}}
\newcommand{\rc}{\mathcal{R}}
\newcommand{\onen}{\gc 1,n \dc}
\newcommand{\StnnM}{\mathcal{S}_{\mc}^{\mathrm{tnn}}}
\newcommand{\Xtil}{\widetilde{X}}

\newcommand{\Bhat}{\widehat{B}}
\newcommand{\Ghat}{\widehat{G}}
\newcommand{\Phat}{\widehat{P}}
\newcommand{\pchat}{\widehat{\pc}}
\def\bysame{\leavevmode\hbox to3em{\hrulefill}\thinspace}

\title{Totally nonnegative cells and matrix Poisson varieties}

\author{K.R. Goodearl\thanks{\,The research of the first named author was supported
by a grant from the National Science Foundation (USA).},~~S. Launois\thanks{\,The research of the second named author was supported by a Marie Curie European Reintegration Grant within the
$7^{\mbox{th}}$ European Community Framework Programme.}~~and
T.H. Lenagan}

\begin{document}

\maketitle

\abstract{\footnotesize We describe explicitly the admissible
families of minors for the totally nonnegative cells of real matrices, that is, the families
of minors that produce nonempty cells in the cell decompositions of spaces of totally nonnegative matrices introduced by A. Postnikov. In order to do this, we relate the totally nonnegative cells to torus orbits of symplectic leaves of the Poisson varieties of complex  matrices. In particular, we describe the minors that vanish on a torus orbit of symplectic leaves, we prove that such families of minors are exactly the admissible families, and we show that the nonempty totally nonnegative cells are the intersections of the torus orbits of symplectic leaves with the spaces of totally nonnegative matrices.}

\section*{Introduction.}

In this paper, we investigate two related decompositions of matrix spaces. The first concerns the space $\mmptnn$ of $m\times p$
totally nonnegative real matrices. (Recall that a matrix
$M\in\mc_{m,p}(\bbR)$ is totally nonnegative if every minor of $M$ is
nonnegative.) Postnikov gives a cell decomposition of $\mmptnn$ in \cite{postnikov}.
The second space is the affine matrix variety $\Mmpc$, endowed
with its standard Poisson structure. Here the relevant decomposition is that into orbits of symplectic leaves under a standard torus action, as investigated in \cite{bgy}. Both decompositions are determined by sets of minors (via equations and inequations), and they are known to be parametrised by the same combinatorial objects. We determine the precise sets of minors defining nonempty totally nonnegative cells (respectively, torus orbits of symplectic leaves), we show that these sets of minors coincide, and we use this to prove that the nonempty cells in $\mmptnn$ are precisely the intersections of $\mmptnn$ with the torus orbits of symplectic leaves in $\Mmpc$. More detail follows.

In the unfinished paper \cite{postnikov}, first posted on the arxiv in 2006,
Postnikov investigates the totally nonnegative (parts of) Grassmannians. He gives
stratifications of the totally nonnegative Grassmannians via cells, and provides
parametrisations of these cells via combinatorial objects
that he calls \emph{$Le$-diagrams}. He also describes an algorithm which has as
output a list of the minors that vanish on the cell corresponding to a given $Le$-diagram. These results easily translate, via dehomogenisation, to corresponding statements
about spaces of totally nonnegative matrices. The cell in $\mmptnn$ corresponding to a collection $\mathcal{F}$ of minors consists of those matrices $M$ for which the minors vanishing on $M$ are precisely those in $\mathcal{F}$. Many such cells are empty, leaving the problem of which collections of minors define nonempty cells in $\mmptnn$. We solve this problem in Theorem \ref{TheoDescription}. Further, we develop an alternative algorithm for
calculating the minors that vanish on a totally nonnegative cell. This is a version of the Restoration Algorithm originally constructed by the
second named author in \cite{lauPEMS} in order to study quantum matrices. 

Postnikov's $Le$-diagrams had already appeared in the literature in Cauchon's
study of the torus invariant prime ideals in quantum matrices, see \cite{c2}, and were denoted \emph{Cauchon diagrams} in subsequent work in that area.
For that reason, in this paper we use the term ``Cauchon diagram'' instead of
``$Le$-diagram''. 

The method we employ to describe the sets of minors that define nonempty
cells in $\mmptnn$ is indirect; it is based on the matrix Poisson affine space $\Mmpc$ and its coordinate ring, the Poisson algebra $\pmmpc$. There is a
natural action of the torus $\hc:=\left( \bbC^{\times} \right)^{m+p}$ on
$\Mmpc$ and a corresponding induced action on $\pmmpc$. In
\cite{bgy}, the stratification of $\Mmpc$ by $\hc$-orbits of symplectic leaves is studied. These orbits are parametrised by certain ``restricted
permutations'' $w$ from the symmetric group $S_{m+p}$, namely permutations that do not move any integer more than $m$ units to the right nor more than $p$ units to the left. 

One of the main results of \cite{bgy}, Theorem 4.2, describes the matrices
that belong to the torus orbit of symplectic leaves corresponding to a given restricted
permutation in terms of rank conditions on the matrices. Here,
our first main aim in the Poisson setting is to determine exactly
which minors vanish on the
(closure of) such a torus orbit of symplectic leaves. This is complementary to a recent result of Yakimov, who showed that the ideal of polynomial functions vanishing on such an orbit is generated by a set of minors \cite[Theorem 5.3]{Yak}.

Once this first main aim has been achieved, we study the link between totally nonnegative cells in $\mmptnn$ and torus orbits of symplectic leaves in $\Mmpc$. In particular, we introduce the notion of $\hc$-invariant Cauchon matrices that allows us to prove that a family of minors is admissible (that is, the corresponding totally nonnegative cell is nonempty) if and only if it is the list of all coordinate minors in $\pmmpc$ that belong to the defining ideal of the (closure of) some torus orbit of
symplectic leaves. This leads to our main Theorem \ref{TheoDescription}, which provides an explicit description of the sets of minors that determine nonempty cells in $\mmptnn$. To prove it, we trace vanishing properties of minors through the restoration algorithm and relate that information to $\hc$-invariant prime Poisson ideals of $\pmmpc$. Once this theorem is established, finally, we derive the correspondence between totally nonnegative cells and torus orbits of symplectic leaves: Postnikov's partition of $\mmptnn$ into nonempty cells coincides with the partition obtained by intersecting $\mmptnn$ with the partition of $\Mmpc$ into $\hc$-orbits of symplectic leaves. Both partitions are thus parametrised by the restricted permutations mentioned above.

Note that the parametrisations of the nonempty totally nonnegative cells by
two seemingly distinct combinatorial objects is illusory -- there is a natural way to
construct a restricted permutation from a Cauchon diagram via the notion of
pipe dreams (see \cite[Section 19]{postnikov}). At the end of this paper, we present an algorithm that, starting only from a Cauchon diagram, constructs 
an admissible family of minors. Of course, it would be interesting to 
know exactly which restricted permutation parametrises the admissible family obtained from a given Cauchon diagram via this algorithm. We
will return to this question in a subsequent paper.  

In \cite{GLL2}, we use ideas developed in the present article in order to prove that the quantum analogues of the admissible families of minors are exactly the sets of quantum minors contained in the $\hc$-prime ideals of the algebra $\mmpc$ of quantum matrices. When the quantum parameter $q$ is transcendental over $\mathbb{Q}$, these quantum minors generate the $\hc$-prime ideals of $\mmpc$, as proved by the second-named author \cite[Th\'eor\`eme 3.7.2]{lauJAlg}. A different approach to this result, applicable to many quantized coordinate algebras, is developed by Yakimov in \cite{Yak} (see \cite[Theorem 5.5]{Yak}).
\medskip

Throughout this paper, we use the following conventions: 
\begin{enumerate}
\item[$\bullet$] $\bbN$ denotes 
the set of positive integers, and
$\bbC^{\times}:=\bbC\setminus \{0\}$. 
\item[$\bullet$] If $I$ is any nonempty
finite subset of $\bbN$, then $|I|$ denotes its cardinality.
\item[$\bullet$] If $k$ is a positive integer, then $S_k$ denotes the group of
permutations of $\gc 1,k \dc:=\{1, \cdots, k\}$.
\item[$\bullet$] $m$ and $p$ denote two fixed positive integers with $m,p\geq2$.
\item[$\bullet$] $\mc_{m,p}(\bbR)$ denotes the space of $m \times p$
matrices with real entries, equipped with the Zariski topology.
\item[$\bullet$] If $K$ is a field and $I \subseteq \onem$ and $\Lambda
\subseteq \onep$ with $|I|=|\Lambda |=t \geq 1$, then we denote by
$[I | \Lambda ]$ the minor in $\oc(\mc_{m,p}(K))=
K[Y_{1,1}, \dots, Y_{m,p}]$ defined by: 
$$[I | \Lambda ]:=\det \left( Y_{\ia}
\right)_{(\ia) \in I \times \Lambda}.$$
It is convenient to also allow the empty minor: $[\emptyset|\emptyset] := 1 \in \oc(\MmpK)$. 
If $I= \{i_1,\dots,i_l\}$ and $\Lambda= \{\alpha_1,\dots,\alpha_l\}$, we write the
minor $[I|\Lambda]$ in the form
$$[i_1,\dots,i_l|\alpha_1,\dots,\alpha_l].$$
Whenever we write a minor in this form, we tacitly assume that the row and column
indices are listed in ascending order, that is, $i_1< \cdots< i_l$ and $\alpha_1<
\cdots< \alpha_l$.
\end{enumerate}

\section{Totally nonnegative matrices and cells.}

\subsection{Totally nonnegative matrices.}

A matrix $M \in \mc_{m,p}(\bbR)$ is said to be \emph{totally
nonnegative} (\emph{tnn} for short) if all of 
its minors are nonnegative. The set of all $m\times p$ tnn matrices is denoted by
$\mmptnn$. This set is
a closed subspace of $\mc_{m,p}(\bbR)$. Further, a matrix is
said to be \emph{totally positive} if all its minors are
strictly positive; the set of all $m\times p$ totally positive
matrices is denoted by  
$\mc^{> 0}_{m,p}(\bbR)$. 
As a result of their importance in various domains of mathematics and
science, these classes of matrices have been extensively studied for more
than a century (see for instance \cite{Ando,Gasca}). 

\subsection{Cell decomposition.}

The space 
$\mmptnn$ admits a natural partition into
so-called \emph{totally nonnegative cells} in the following way. 
For any family $\mathcal{F}$
of minors (viewed as elements of the coordinate ring $\oc\left(
\mc_{m,p}(\bbR) \right)$), we define the totally nonnegative
cell $S_{\mathcal{F}}$ associated with $\mathcal{F}$ by: 
\begin{equation} \label{deftnncell}
S_{\mathcal{F}}:= \{
M \in \mmptnn \mid [I |J](M)=0 
~\mbox{if~and~only~if~}[I |J]\in\mathcal{F}
\}, 
\end{equation}
where $[I |J]$ runs through all minors in $\oc\left(
\mc_{m,p}(\bbR) \right)$.

Note that some cells are empty. For example, in $\mc^{\geq
0}_{2,2}(\bbR)$, the cell associated with $[2|2]$ is empty. Indeed, 
suppose 
that this cell were nonempty. Then there would exist a tnn matrix 
$\begin{bmatrix} a & b \\ c & 0 \end{bmatrix}$ 
such that $a,b,c>0$, but 
$-bc=\det \begin{bmatrix} a & b \\ c & 0 \end{bmatrix} > 0$, which is impossible.
\medskip

\begin{defi} 
A family of minors is \emph{admissible} if the corresponding 
totally nonnegative cell is nonempty.
\end{defi}

Hence, we have the following partition of the space $\mmptnn$: $$\mmptnn =\bigsqcup_{\mathcal{F} \mbox{ admissible}}
S_{\mathcal{F}},$$ which explains the importance of the tnn cells.
\medskip

The main aim of this paper is to give an explicit description of the families
of minors that are admissible.
\medskip

\subsection{An algorithmic description of the nonempty cells.}
\label{section:CauchonDiagrams}

In \cite{postnikov}, Postnikov considers the cell decomposition of the
\emph{totally nonnegative Grassmannian}. His results can be easily used to get
information about totally nonnegative matrices via dehomogenisation. Postnikov
parametrises the nonempty cells in the Grassmannian in the following way. First,
he shows that the nonempty cells are parametrised by combinatorial objects called
$Le$-diagrams. It is remarkable to note that
$Le$-diagrams have appeared simultaneously and independently in the study by
Cauchon of the so-called $\hc$-primes of the algebra of quantum matrices
\cite{c2}. The importance of $\hc$-primes in the algebra $\mmpc$ of generic
quantum matrices was pointed out by Letzter and the first named author who constructed a
stratification of the prime spectrum of this algebra, which is indexed by the
set of $\hc$-primes. In \cite{c2}, Cauchon has constructed a natural
one-to-one correspondence between the set of $\hc$-primes in quantum matrices
and so-called Cauchon diagrams which in turn are the same as the
$Le$-diagrams. Recall that an $m\times p$ \emph{Cauchon diagram} $C$ is simply
an $m\times p$ grid consisting of $mp$ squares in which certain squares are
coloured black. We require that the collection of black squares have the
following property. If a square is black, then either every square strictly to
its left is black or every square strictly above it is black. Denote the set
of $m\times p$ Cauchon diagrams by $\mathcal{C}_{m,p}$.

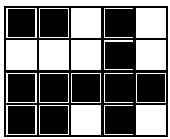
\begin{figure}[h]
\label{fig:CauchonDiagram}
\begin{center}
$$\xymatrixrowsep{0.01pc}\xymatrixcolsep{0.01pc}
\xymatrix{
\plc\edge[0,10]\edge[8,0] &&\plc\edge[8,0] &&\plc\edge[8,0]
&&\plc\edge[8,0] &&\plc\edge[8,0] &&\plc\edge[8,0] \\
 &\mdblk &&\mdblk &&&&\mdblk && &&\plc \\
\plc\edge[0,10] &&&&&&&&&&\plc \\
 &&&&&&&\mdblk && \\
\plc\edge[0,10] &&&&&&&&&&\plc \\
 &\mdblk &&\mdblk &&\mdblk &&\mdblk &&\mdblk \\
\plc\edge[0,10] &&&&&&&&&&\plc \\
 &\mdblk &&\mdblk &&&&\mdblk \\
\plc\edge[0,10] &&\plc &&\plc &&\plc &&\plc &&\plc 
}$$
\caption{An example of a $4\times 5$ Cauchon diagram}
\end{center}
\end{figure}

By convention, an ordered pair of integers $(\ia)$ \emph{belongs to} the Cauchon diagram $C$ if
the box $(\ia)$ in $C$ is black.

One easily obtains the following parametrisation of the nonempty cells in
$\mmptnn$ from Postnikov's work. 

\begin{theo} [{\rm\cite{postnikov}}] 
\label{theo:parametrisationnonemotycells}
The nonempty tnn cells in $\mmptnn$ are 
parametrised by $m \times p$ Cauchon diagrams.
\end{theo}

We will not use this result. However, we will recover it by using different methods.
In particular, in Corollary \ref{bound}, we show that the number of nonempty tnn cells in
$\mmptnn$ is less than or equal to the number of $m
\times p$ Cauchon diagrams. Equality follows from Theorem \ref{TheoDescription}.

At this point it is worth recalling that Cauchon diagrams are also closely
related to restricted permutations. More precisely, set 
\[
\mathcal{S}= S^{[-p,m]}_{m+p}:=\{w \in
S_{m+p} \ | \ -p \leq w(i) -i \leq m \mbox{ for all }i=1,2, \dots, m+p\}.
\]
The set 
$\mathcal{S}$ is a sub-poset of the symmetric group $S_{m+p}$ endowed with the
Bruhat order. Namely, we have \cite[Proposition 1.3]{lau2}, \cite[Lemma 3.12]{bgy}: 
\[
\mathcal{S}=\left\{ w
\in S_{m+p} \ \biggm| \ w \leq 
\begin{bmatrix} 1 & 2 & \dots & p &
p{+}1 & p{+}2 & \dots & m{+}p \\ 
m{+}1 & m{+}2 & \dots & m{+}p & 1 & 2& \dots & m
\end{bmatrix} \right\}.
\] 
It was proved in \cite[Corollary 1.5]{lau2} that the
cardinality of $\mathcal{S}$ is equal to the number of $m \times p$ Cauchon
diagrams. Note that one can construct an explicit bijection between these
two sets by using the concept of pipe-dreams. (See \cite[Section 19]{postnikov})

Postnikov also constructs an algorithm that starts with a Cauchon diagram as
input and produces an admissible family of minors for a nonempty cell as
output. However, although Postnikov's algorithm produces a list of the minors
for such an admissible family, it does not give an explicit description of the
admissible families. As we do not use Postnikov's algorithm in this paper, we
refer the reader to \cite{postnikov} for details of the algorithm.

In the present paper, we give an explicit description of the admissible
families and also develop another algorithmic method to find them. Our
strategy to do so is to relate tnn cells with the $\hc$-orbits of symplectic
leaves of the Poisson algebra $O \left( \mc_{m,p} (\bbC)
\right)$ (viewed as the semiclassical limit of the algebra $\mmpc$ of quantum
matrices). In the next section we recall this Poisson structure and the
description of the $\hc$-orbits of symplectic leaves that has been obtained
by Brown, Yakimov and the first named author \cite{bgy}.

\section{Poisson $\hc$-prime ideals of $\pmmpc$.}
\label{sectionPoisson}

In this section, we investigate the standard Poisson structure of the
coordinate ring $\pmmpc$ coming from the commutators of $\mmpc$. Recall that a
\emph{Poisson algebra} (over $\bbC$) is a commutative
$\bbC$-algebra $A$ equipped with a Lie bracket $\{-,-\}$ which is a
derivation (for the associative multiplication) in each variable. The
derivations $\{a,-\}$ on $A$ are called \emph{Hamiltonian derivations}.
When $A$ is the algebra of complex-valued $C^{\infty}$ functions on a smooth
affine variety $V$, one can use Hamiltonian derivations in order to define
Hamiltonian paths in $V$. A \emph{Hamiltonian path in $V$} is a smooth path
$\gamma : [0,1] \rightarrow V$ such that there exists $f \in C^{\infty}(V)$
with $\frac{d\gamma}{dt}(t)=\xi_f(\gamma(t))$ for all $0 < t <1$, where
$\xi_f$ denotes the vector field associated to the Poisson derivation
$\{f,-\}$. It is easy to check that the relation ``connected by a piecewise
Hamiltonian path" is an equivalence relation. The equivalence classes of this
relation are called the \emph{symplectic leaves} of $V$; 
they form a partition of $V$.

A \emph{Poisson ideal} of $A$ is any ideal $I$ such that $\{A,I\}
\subseteq I$, and a \emph{Poisson prime} ideal is any prime ideal which
is also a Poisson ideal. The set of Poisson prime ideals in $A$ forms the
\emph{Poisson prime spectrum}, denoted $\mathrm{PSpec} (A)$, which is given the
relative Zariski topology inherited from $\mathrm{Spec} (A)$.

\subsection{The Poisson algebra $\pmmpc$.}

The coordinate ring of the variety $\mc_{m,p}(\bbC)$ will be
denoted by $\pmmpc$; it is a (commutative) polynomial algebra in $mp$
indeterminates $Y_{\ia}$ with $1 \leq i \leq m$ and $1 \leq \alpha \leq p$. 

The variety $\mc_{m,p}(\bbC)$ is a Poisson variety: 
one defines a Poisson structure on its coordinate ring $\pmmpc$ 
by the following data. 
\[
\{Y_{\ia} ,Y_{k,\gamma} \} = \begin{cases}
Y_{\ia}Y_{k,\gamma} &  \mbox{ if } i=k \mbox{ and } \alpha < \gamma \\
Y_{\ia} Y_{k,\gamma} &  \mbox{ if } i< k \mbox{ and } \alpha = \gamma \\
0 &   \mbox{ if } i < k \mbox{ and } \alpha > \gamma \\
2 Y_{i,\gamma} Y_{k,\alpha} &  
\mbox{ if } i< k  \mbox{, } \alpha < \gamma \,. 
\end{cases}
\]
This is the standard Poisson bracket on 
$\pmmpc$ and it arises as the 
semiclassical limit of the family of noncommutative algebras 
$\mmpc$, see \cite{BookBrownGoodearl}. 

Also, note that the Poisson bracket on $\pmmpc$ extends uniquely to a Poisson
bracket on $\mathcal{C}^{\infty}(\mc_{m,p}(\bbC))$, so that
$\mc_{m,p}(\bbC)$ can be viewed as a Poisson manifold. Hence
$\mc_{m,p}(\bbC)$ can be decomposed as the disjoint union of its
symplectic leaves.

\subsection{Torus action.}

The torus $\hc:=\left( \bbC^{\times} \right)^{m+p}$ acts on $\pmmpc$ by
Poisson automorphisms via:
\[
(a_1,\dots,a_m,b_1,\dots,b_p).Y_{\ia} = a_i b_\alpha Y_{\ia} \quad {\rm
for~all} \quad \: (\ia)\in \gc 1,m \dc \times \gc 1,p \dc.
\] 
The set of Poisson primes of
$\pmmpc$ that are invariant under this action of $\hc$ is denoted by 
$\hc$-$\mathrm{PSpec}(\pmmpc)$. Note that $\hc$ is
acting rationally on $\pmmpc$.

At the geometric level, this action of the algebraic torus $\hc$ on the
coordinate ring comes from the left action of $\hc$ on
$\mc_{m,p}(\bbC)$ by Poisson isomorphisms via:
$$(a_1,\dots,a_m,b_1,\dots,b_p).M := \diag(a_1,\dots,a_m) M \diag(b_1,\dots,b_p).$$
This action of $\hc$ on $\mc_{m,p}(\bbC)$ induces an action of
$\hc$ on the set $\mathrm{Sympl}(\mc_{m,p}(\bbC))$ of symplectic
leaves in $\mc_{m,p}(\bbC)$. As in \cite{bgy}, we view the
$\hc$-orbit of a symplectic leaf $\lc$ as the set-theoretic union
$\bigcup_{h\in\hc} h.\lc \subseteq \Mmpc$, rather than as the family $\{h.\lc \mid
h\in\hc\}$. We denote the set of such orbits by
$\hc$-$\mathrm{Sympl}(\mc_{m,p}(\bbC))$. These orbits were described
by Brown, Yakimov and the first named author who obtained the following results.

We use the notation of \cite{bgy} except that we replace $n$ by $p$. 
In particular, we set $N = m+p$. Let $\wom$, $\wop$ and $\woN$ denote the
respective longest elements in $S_m$, $S_p$ and $S_N$, respectively, so that 
$w_\circ^r(i) = r + 1 - i$ for $i = 1, . . . , r$. 
Recall from equation (3.24) and Lemma 3.12 of
\cite{bgy} that 
\begin{equation} \label{woNS}
\woN \mathcal{S}= S_N^{\geq (\wop,\wom)} := \{w\in S_N\mid w\geq (\wop,\wom)\},
\end{equation} 
where
\[
(\wop,\wom) := \begin{bmatrix} 1 & 2 & \dots & p & p+1 & p+2 & \dots & p+m \\ 
p & p-1 & \dots & 1 & p+m & p+m-1& \dots & p+1\end{bmatrix}.
\]

\begin{theo} 
{\rm \cite[Theorems 3.9, 3.13, 4.2]{bgy}}
\begin{enumerate}
\item There are only finitely many $\hc$-orbits of symplectic leaves in
$\mc_{m,p}(\bbC)$, and they are smooth irreducible locally
closed subvarieties. 
\item The set $\hc$-$\mathrm{Sympl}(\mc_{m,p}(\bbC))$ of orbits 
{\rm(}partially ordered by
inclusions of closures{\rm)} is isomorphic to the set $S_N^{\geq (\wop,\wom)}$ with
respect to the Bruhat order. 
\item Each $\hc$-orbit of symplectic leaves is defined by the vanishing and
nonvanishing of certain sets of minors. 
\item Each closure of an $\hc$-orbit of symplectic leaves is defined by the
vanishing of a certain set of minors.
\end{enumerate}
\end{theo}

For $y \in S_N^{\geq (\wop,\wom)}$, we denote by $\mathcal{P}_{y}$ the $\hc$-orbit
of symplectic leaves described in \cite[Theorem 3.9]{bgy}.

\subsection{On the minors that vanish on the closure of an orbit of leaves.}

In this section, we describe explicitly the minors that vanish on a
given $\hc$-orbit of symplectic leaves in $\Mmpc$. For later purposes, we need to
parametrize these $\hc$-orbits by $\mathcal{S}$ rather than by $S_N^{\geq
(\wop,\wom)}$. Hence, factors of $\woN$ are required when carrying over results
from \cite{bgy} (recall \eqref{woNS}).

We identify permutations in $S_N$ with the corresponding permutation matrices in
$\mc_N(\bbZ)$. Thus, $w\in S_N$ is viewed as the matrix with entries $w_{ij}=
\delta_{i,w(j)}$.

Let $w \in \mathcal{S}$, and write $w$ in block form as
\begin{equation} \label{blockform1}
w = \begin{bmatrix}
w_{11} & w_{12}\\
w_{21} &  w_{22}
\end{bmatrix} \qquad\qquad 
\begin{pmatrix} 
w_{11} \in \mc_{m,p}(\bbZ) &w_{12} \in \mc_{m}(\bbZ)\\
w_{21} \in \mc_{p}(\bbZ) &w_{22} 
\in \mc_{p,m}(\bbZ)
\end{pmatrix}.
\end{equation}
Hence, 
\begin{equation} \label{blockform2}
\woN w=\begin{bmatrix}
0 & w_{\circ}^p\\
w_{\circ}^m &  0
\end{bmatrix} 
\begin{bmatrix}
w_{11} & w_{12}\\
w_{21} &  w_{22}
\end{bmatrix}=\begin{bmatrix}
\wop w_{21} & \wop w_{22}\\[1ex]
\wom w_{11} &  \wom w_{12}
\end{bmatrix},
\end{equation}
which is the block form of $\woN w$ as in \cite[\S4.2]{bgy}.

Now \cite[Theorem 4.2]{bgy} shows that the closure
$\overline{\mathcal{P}_{\woN w}}$ of $\mathcal{P}_{\woN w}$ consists of the
matrices $x \in \mc_{m,p}(\bbC)$ such that each of the following
four conditions holds. Here $y[a,\dots,b;c,\dots,d]$ denotes the submatrix of a
matrix $y$ involving the rows $a,\dots,b$ and columns $c,\dots,d$.

\begin{description}
\label{4conditions}
\item[Condition 1.] $\rank ( x[r,\dots,m;1,\dots,s]) \leq 
\rank ( (\wom w_{11})[r,\dots,m;1,\dots,s])$ for
 $r\in\onem$ and $s\in\onep$.
\item[Condition 2.] $ \rank (x[1,\dots,r;s,\dots,p]) \leq 
\rank (( \wom w\transp_{22})[1,\dots,r;s,\dots,p])$ for
 $r\in\onem$ and $s\in\onep$.
\item[Condition 3.] For $2\le r\le s\le p$,\\
$\rank (x[1,\dots,m;r,\dots,s]) \leq 
s + 1 - r - \rank (w_{21}[r,\dots,p;r,\dots,s])$.
\item[Condition 4.] For  $1\le r\le s\le m-1$,\\
$\rank (x[r,\dots,s;1,\dots,p]) \leq 
s + 1 - r - \rank (( \wom w_{12} \wom)[r,\dots,s;1,\dots,s])$.
\item[Modifications.] We can, and do, allow $r=1$ in Condition 3 and $s=m$ in Condition 4, even though these cases are redundant. First, Condition 1 with $r=1$ says
\begin{align*} 
\rank ( x[1,\dots,m;1,\dots,s])  &\leq 
\rank ( (\wom w_{11})[1,\dots,m;1,\dots,s]) \\
 &= \rank ( w_{11}[1,\dots,m;1,\dots,s]) \\
  &= s- \rank (w_{21}[1,\dots,p;1,\dots,s]),
\end{align*}
because appending the first $s$ columns of $w_{21}$ to those of $w_{11}$ yields the first $s$ columns of the permutation $w$. This gives Condition 3 with $r=1$. Second, Condition 1 with $s=p$ says
\begin{align*}
\rank ( x[r,\dots,m;1,\dots,p])  &\leq 
\rank ( (\wom w_{11})[r,\dots,m;1,\dots,p]) \\
&= m+1-r- \rank ( (\wom w_{12})[r,\dots,m;1,\dots,m]) \\
&= m+1-r- \rank ( (\wom w_{12} \wom)[r,\dots,m;1,\dots,m]),
\end{align*}
which gives Condition 4 with $s=m$.
\end{description}

Our next aim is to rewrite these conditions in terms of the vanishing of minors. 
This needs first a result on the vanishing of minors on a Bruhat cell.

Let $R_{m,p}$ denote the set of all partial permutation matrices in 
$\mc_{m,p}(\bbC)$, and identify any $ w \in R_{m,p}$ with the
corresponding bijection from its domain $\mathrm{dom}(w)$ onto its range
$\mathrm{rng}(w)$; thus, $w(j) = i$ if and only if $w_{ij} = 1$. Note that $w\transp$
is the partial permutation matrix corresponding to the inverse bijection $w^{-1} :
\mathrm{rng}(w) \rightarrow \mathrm{dom}(w)$. Let
$B_m^{\pm}$ and $B_p^{\pm}$ denote the standard Borel subgroups in $\mathrm{GL}_m$ 
and $\mathrm{GL}_p$. 

Recall that $I \leq J$, for finite $I, J \subseteq \bbN$ with $|I| = |J|$, 
means that when $I$ and $J$ are written in ascending order, say $I = \{i_1 < \cdots
< i_t\}$ and $J = \{j_1 < \cdots < j_t\}$, we have
$i_l \leq j_l$ for all $l$.

\begin{lem}
\label{singleBruhat}
Let $w \in R_{m,p}$, and let 
$I \subseteq\onem$ 
and $\Lambda \subseteq \onep$ with $|I| = |\Lambda |$.
The following are equivalent:

{\rm (a)} $[I | \Lambda]$ vanishes on $B_m^+ w B_p^+$.

{\rm (b)} $I \nleq w(L)$ for all $L \subseteq \mathrm{dom}(w)$  such that 
$L \leq \Lambda$.

{\rm (c)} $\Lambda \ngeq w^{-1}(L')$ for all $L' \subseteq \mathrm{rng}(w)$ such
that $L' \geq I$.
\end{lem}

Note that as we only allow  $\leq$ for index sets of the same cardinality,
conditions (b) and (c) hold automatically when 
$|I| = |\Lambda| > \mathrm{rank}(w).$

\begin{proof} Set $l := |I| = |\Lambda |$.

(a)$\Rightarrow$(b):  Note that $[I | \Lambda]$ vanishes on $\overline{B_m^+ w
B_p^+}$, and hence on $\overline{B_m^+} w \overline{B_p^+}$. Suppose there exists
$L \subseteq \mathrm{dom}(w)$ such that $L\subseteq \Lambda$  and $I \subseteq
w(L)$. Write the relevant index sets in ascending order:
$$\begin{aligned}
I &= \{i_1 < \cdots < i_l\} &\qquad\qquad\qquad \Lambda &= \{\lambda_1 < \cdots <
\lambda_l\} \\
L &= \{l_1 < \cdots < l_l\} &w(L) &= \{m_1 < \cdots < m_l\}.
\end{aligned}$$
Set $a= \sum_{s=1}^l e_{i_s,m_s}$ and $b = \sum_{s=1}^l e_{l_s,\lambda_s}$, where
$e_{i,\alpha}$ denotes the matrix with a 1 in position $(i,\alpha)$ and 0 anywhere
else. Since $i_s \leq m_s$ and $l_s \leq \lambda_s$ for all $s$, we have $a \in 
\overline{B_m^+}$ and $b \in \overline{B_p^+}$. Also, $a$ and $b$ are partial
permutation matrices, representating bijections $w(L) \rightarrow I$ and $\Lambda
\rightarrow L$, respectively, whence $awb$ is a partial permutation matrix
representing a bijection $\Lambda \rightarrow I$. Therefore 
$[I | \Lambda](awb) = \pm 1$. Since $awb \in \overline{B_m^+} w \overline{B_p^+}$, 
we have a contradiction.

(b)$\Rightarrow$(a): Let $x = awb$ for some $a \in  B_m^+$ and $b \in B_p^+$. 
To show that $[I | \Lambda ](x) = 0$,
we operate in $\mc_{\mu}(\bbC)$ where $\mu = \max\{m, p\}$, and we identify $\mc_m(\bbC)$, $\mc_p(\bbC)$, and $\mc_{m,p}(\bbC)$ with
the upper left blocks of $\mc_{\mu}(\bbC)$ of the appropriate sizes. In particular, $a$  and $b$ remain upper triangular.

The only minors of $w$ which do not vanish are the minors $[w(L) |L]$ for 
$L \subseteq \mathrm{dom}(w)$, and $[w(L) |L](w) = \pm 1$  (depending on the sign
of $w|_L$).  We claim that $[U | V ](a) = 0$ whenever $|U| = |V |$ and $U \nleq V$.
Write $U = \{u_1 < \cdots < u_t\}$ and $V = \{v_1 < \cdots < v_t\}$; then $u_s >
v_s$ for some $s$.  For $\beta \geq s \geq \alpha$, we have $u_{\beta} \geq u_s >
v_s \geq v_{\alpha}$, and so 
$a_{u_{\beta},v_{\alpha}} = 0$ because $a$ is upper triangular. 
Thus, the $U \times V$ submatrix of $a$ has zero lower left $(t + 1 - s) \times s$
block. That makes this submatrix singular, so $[U | V ](a) = 0$, as claimed.
Likewise, $[U | V ](b) = 0$. Now expand $[I | \Lambda ](x)$ in $\mc_{\mu}(\bbC)$,
to obtain 
\[
[I | \Lambda ](x) = [I | \Lambda ](awb) =  \sum_{\substack{K,L\subseteq
\{1,\dots,\mu
\}\\ |K|=|L|=l}} [I |K](a)\, [K |L](w)\, [L| \Lambda ](b).
\]
The terms in the above sum vanish whenever  $I \nleq K$ or $L \nleq \Lambda$, and
they also vanish unless $L \subseteq \mathrm{dom}(w)$ and $K = w(L)$. Hence,
\[ 
[I | \Lambda ](x) = \sum_{\substack{L \subseteq\mathrm{dom}(w), \ |L|=l \\
L \leq \Lambda , \ w(L) \geq I}}
\pm [I |w(L)](a)\, [L| \Lambda ](b).
\]
However, by assumption (b),  there are no index sets $L$ satisfying the conditions
of this summation. Therefore $[I | \Lambda](x) = 0$.

(b)$\Leftrightarrow$(c): Take $L' = w(L)$.
\end{proof}

\begin{prop} 
\label{belongsingleBruhat}
Let $w \in R_{m,p}$ and $x\in \Mmpc$. Then the following conditions are equivalent:

{\rm (a)} $x\in \ov{B_m^+wB_p^+}$.

{\rm (b)} $\rank(x[r,\dots,m;1,\dots,s]) \leq 
\rank(w[r,\dots,m;1,\dots,s])$ for $r\in\onem$ and $s\in\onep$.

{\rm (c)} $[I|\Lambda](x)=0$ for all $I \subseteq\onem$ 
and $\Lambda \subseteq \onep$ with $|I| = |\Lambda |$ such that
\begin{enumerate}
\item[$(*)$] $I \nleq w(L)$ for all $L
\subseteq
\mathrm{dom}(w)$ such that 
$L \leq \Lambda$.
\end{enumerate}
\end{prop}

\begin{proof} (a)$\Rightarrow$(c) is immediate from Lemma \ref{singleBruhat}, and
(b)$\Rightarrow$(a) follows from \cite[Proposition 3.3(a)]{ful} (as noted in
\cite[Proposition 4.1]{bgy}).

(c)$\Rightarrow$(b): It suffices to show, for any $r\in\onem$ and $s\in\onep$, that
$[I|\Lambda](x)=0$ whenever $I \subseteq \gc r, m \dc$ and $\Lambda \subseteq
\gc 1, s \dc$ with 
$$|I|=|\Lambda| > \rho_{rs}:= \rank(w[r,\dots,m;1,\dots,s]).$$

Let $I$ and $\Lambda$ satisfy the above conditions, and suppose there exists $L
\subseteq
\mathrm{dom}(w)$ such that 
$L \leq \Lambda$ and $I\le w(L)$. Since $\Lambda \subseteq
\gc 1, s \dc$ and $L\leq \Lambda$, we have $L \subseteq
\gc 1, s \dc$. Then, since $|L|= |\Lambda| > \rho_{rs}$, we must have $w(L)
\nsubseteq \gc r, m \dc$, whence $\min w(L) <r$. But that implies $\min I <r$
(because $I\leq w(L)$), contradicting the assumption that $I \subseteq
\gc r, m \dc$. Thus, $I$ and $\Lambda$ satisfy $(*)$, and so $[I|\Lambda](x)=0$
by hypothesis. 
\end{proof}

Similarly, we have the following result.

\begin{prop} 
\label{belongsingleBruhat2}
Let $w \in R_{m,p}$ and $x\in \Mmpc$. Then the following conditions are equivalent:

{\rm (a)} $x\in \ov{B_m^-wB_p^-}$.

{\rm (b)} $\rank(x[1,\dots,r;s,\dots,p]) \leq 
\rank(w[1,\dots,r;s,\dots,p])$ for $r\in\onem$ and $s\in\onep$.

{\rm (c)} $[I|\Lambda](x)=0$ for all $I \subseteq\onem$ 
and $\Lambda \subseteq \onep$ with $|I| = |\Lambda |$ such that
\begin{enumerate}
\item[$(*)$] $\Lambda \nleq w^{-1}(L)$ for all $L \subseteq \mathrm{rng}(w)$  such
that $L \leq I$.
\end{enumerate}
\end{prop}

\begin{proof}
Observe that $(B_m^- w B_p^-)\transp=B_p^+ w\transp B_m^+$. 
Switch the roles of $m$, $p$, $I$ and $\Lambda$ in the previous proposition, and
take the transpose.
\end{proof}

We can now reformulate the four rank conditions above (as modified to extend the ranges of Conditions 3 and 4)  in terms of the vanishing of
certain minors.

\begin{obs} \label{translateconditions1234}
Let $w\in \mathcal{S}$, and note that $w\woN= \begin{bmatrix}
w_{12}\wom &w_{11}\wop\\ w_{22}\wom &w_{21}\wop \end{bmatrix}$.
\medskip

\noindent $\bullet$ Condition 1. By Proposition \ref{belongsingleBruhat}, this
occurs if and only if $[I | \Lambda ] (x) = 0$ whenever $I \nleq \wom  w_{11}(L)$
for all
$L \subseteq \mathrm{dom}(\wom w_{11})= \mathrm{dom}(w_{11})$ with $|L| = |I|$ and
$L \leq \Lambda$. Thus, it occurs if and only if 
\begin{enumerate}
\item[] $[I | \Lambda](x) = 0 \mbox{ whenever } I \nleq \wom 
w(L)$ for all
 $L \subseteq \onep \cap w^{-1}\onem$ with
$|L| = |I|$ and $L \leq \Lambda$.
\end{enumerate}

\noindent $\bullet$  Condition 2. By Proposition \ref{belongsingleBruhat2}, this
occurs if and only if $[I | \Lambda ](x) = 0$ whenever $\Lambda \nleq (\wom 
w_{22}\transp)^{-1}(L)$ for all $L \subseteq \mathrm{rng}(\wom 
w_{22}\transp) = \mathrm{dom}(w_{22}\wom)$ with $|L| = |\Lambda |$ and $L
\leq I$. Thus, it occurs if and only if 
\begin{enumerate}
\item[] $[I | \Lambda](x) = 0$ whenever $m+\Lambda \nleq w \woN 
(L)$ for all
$L \subseteq \onem \cap \woN w^{-1}\gc m+1, N \dc$
with $|L| = |\Lambda| \mbox{ and } L \leq I$.
\end{enumerate}

\noindent $\bullet$  Condition 3. The rank of $w_{21}{[r,\dots,p;r,\dots,s]}$
is the number of $j$ in the set $\gc r,s\dc \cap \mathrm{dom}(w_{21})$ such that
$w_{21}(j) \geq r$, hence the number of $j \in \gc r, s\dc$ such that
$w(j)\geq m + r$. Consequently,
\[
s + 1 - r - \mathrm{rank}(w_{21}{[r,\dots,p;r,\dots,s]}) = |\{
j \in \gc r, s \dc \mid w(j) < m + r\}|.
\]
Since $w(j) \leq m + s$ for $j \in \gc r, s\dc$ (because $w \in \mathcal{S}$),
Condition 3 may be rewritten as:
\begin{enumerate}
\item[] $[I | \Lambda ](x) = 0$ whenever $\Lambda \subseteq \gc r, s \dc$
and
$|\Lambda | > | \gc r, s \dc \setminus w^{-1}\gc m+r ,\, m+s
\dc|$.
\end{enumerate}
Using a Laplace expansion, it is easy to see that this is equivalent to 
\begin{enumerate}
\item[] $[I | \Lambda ](x) = 0 $ whenever there exist $\Lambda'
\subseteq \Lambda$ and $ 1\le r \leq s\le p$ such that
 $\Lambda' \subseteq \gc  r, s \dc$
 and $|\Lambda' | > | \gc  r, s \dc\setminus w^{-1}\gc m+r , m+s \dc|$, that is, whenever there exist $ 1\le r \leq s\le p$ such that $|\Lambda\cap \gc  r, s \dc| > | \gc  r, s \dc\setminus w^{-1}\gc m+r , m+s \dc|$.
\end{enumerate}

\noindent $\bullet$ Condition 4. The rank of $(\wom
w_{12}\wom){[r,\dots,s;1,\dots,s]}$ is the number of $j$ in $\gc 1,s\dc \cap
\wom\mathrm{dom}(w_{12})$ such that $\wom w_{12}\wom (j) \in \gc r, s\dc$, and
so equals the number of $j \in \gc 1, s\dc$ such that $w(N+1-j) \in \wom \gc r, s \dc$. For $j < r$, we have 
$$w(N+1-j) \geq (N+1-j)-p=m+1-j >  m + 1-r$$
(because $w\in \mathcal{S}$), and so the rank of 
$(\wom w_{12}\wom){[r,\dots,s;1,\dots,s]}$ equals the number of 
$j \in \gc r, s\dc$ such that $w(N+1-j)\in \wom\gc r, s\dc$. Hence,
\begin{multline}
s + 1 - r - \mathrm{rank}((\wom w_{12}\wom){[r,\dots,s;1,\dots,s]})\\
 = |\{ 
j' \in \woN\gc r, s \dc \mid  w(j') \notin \wom\gc r, s\dc \}|.
\notag\end{multline}
Consequently, Condition 4 may be rewritten as:
\begin{enumerate}
\item[] $[I | \Lambda ](x) = 0$ whenever $I \subseteq \gc r, s \dc$
and
 $|I | > | \woN\gc r, s \dc \setminus w^{-1}\wom\gc r, s \dc|$.
\end{enumerate}
Using a Laplace expansion (as with Condition 3), it is easy to see that this is equivalent to
\begin{enumerate}
\item[] $[I | \Lambda ](x) = 0$ whenever there exist $1\le r \leq s\le m$ such that
 $|I\cap \gc r, s \dc| > |
\woN\gc r, s \dc \setminus w^{-1}\wom\gc r, s \dc|$.
\end{enumerate}
\end{obs}

In view of Observations \ref{translateconditions1234}, it is natural to introduce
the following notation.

\begin{defi} \label{definition-M(w)}
For $w\in \mathcal{S}$, define $\mcw$ to be the set of minors $[I|\Lambda]$, with
$I\subseteq \onem$ and $\Lambda\subseteq \onep$, that satisfy at least one
of the following conditions. 
\begin{enumerate}
\item $I\not\leq \wom w(L)$ for all $L\subseteq \onep\cap
w^{-1}\onem$ such that $|L|=|I|$ and $L\leq \Lambda$. 
\item $m+\Lambda\not\leq 
w\woN(L)$ for all $L\subseteq \onem\cap
\woN w^{-1}\gc m+1, N\dc$ such that $|L|=|\Lambda|$ and $L\leq I$. 
\item There exist $1\leq r \leq s\leq p$ such that $|\Lambda\cap \gc  r, s \dc| > | \gc  r, s \dc\setminus w^{-1}\gc m+r , m+s \dc|$.
\item There exist $1\leq r \leq s\leq m$ such that $|I\cap \gc r, s \dc| > |
\woN\gc r, s \dc \setminus w^{-1}\wom\gc r, s \dc|$.
\end{enumerate}
\end{defi}

\begin{remarks} \label{Mwexplain}
The collections of minors of types 1--4 appearing in $\mcw$ for a given $w\in \mathcal{S}$ can be described in the following way:
\begin{enumerate}
\item ``Down-left-closed'' sets of minors, meaning that the row index set $I$ is allowed to increase (with respect to our ordering on index sets) while $\Lambda$ is allowed to decrease.
\item ``Up-right-closed'' sets, meaning that $I$ is allowed to decrease while $\Lambda$ is allowed to increase.
\item ``Vertical stripes'' of the form $[-|\Lambda] := \{ [I|\Lambda] \mid I \subseteq \gc 1,m\dc \text{\ with\ } |I|=|\Lambda| \}$.
\item ``Horizontal stripes'' of the form $[I|-]$.
\end{enumerate}
These descriptions  are clear from the corresponding conditions in Definition \ref{definition-M(w)}. 

Many of the minors in these sets appear automatically as a consequence of the appearance of others. For instance, if $\mcw$ contains a vertical stripe $[-|\Lambda]$, then it necessarily contains all the vertical stripes $[-|\Lambda']$ for which $\Lambda \subseteq \Lambda' \subseteq \onep$, since any minor in $[-|\Lambda']$ can be expressed in terms of minors from $[-|\Lambda]$ and complementary minors, via a Laplace relation.

Further information about the shapes of these sets is related to the block decomposition \eqref{blockform1} of $w$. By Propositions \ref{belongsingleBruhat} and \ref{belongsingleBruhat2} and Observations \ref{translateconditions1234}, the minors of type 1 in $\mcw$ are precisely the minors that vanish on $\overline{B_m^+ \wom w_{11} B_p^+}$, while the minors of type 2 are those which vanish on $\overline{B_m^- \wom w_{22}\transp B_p^-}$. (Recall the original forms of Conditions 1 and 2 at the beginning of this subsection.) Consequently,
\begin{enumerate}
\item[(a)] Whenever $1\le r\le m$ and $1\le s\le p$ with 
$$\rank((\wom w_{11}) [r,\dots,m;1,\dots,s]) < t:= \min\{m+1-r,\; s\},$$
all $t\times t$ minors $[I|\Lambda]$ with $I\subseteq \gc r,m\dc$ and $\Lambda\subseteq \gc 1,s\dc$ must lie in $\mcw$.
\item[(b)] Whenever $1\le r\le m$ and $1\le s\le p$ with 
$$\rank(( \wom w_{22}\transp) [1,\dots,r;s,\dots,p]) < t':= \min\{r,\; p+1-s\},$$
all $t'\times t'$ minors $[I|\Lambda]$ with $I\subseteq \gc 1,r\dc$ and $\Lambda\subseteq \gc s,p\dc$ must lie in $\mcw$.
\end{enumerate} 
\end{remarks}

\begin{examples}  \label{Mwexamples}
{\bf (a)} Take $w= \begin{bmatrix} 1&2&3&4&5&6&7\\ 3&1&4&2&7&6&5 \end{bmatrix}$, and let $m=3$ and $p=4$. Then
$$w_\circ^3 w_{11}= \begin{bmatrix} 1&0&0&0\\ 0&0&0&1\\ 0&1&0&0 \end{bmatrix} \qquad\qquad\qquad w^3_\circ w_{22}\transp = \begin{bmatrix} 0&1&0&0\\ 0&0&1&0\\ 0&0&0&1 \end{bmatrix}.$$
Observe that the $\gc 2,3\dc \times \{1\}$, $\gc 2,3\dc \times \gc 1,3\dc$, and $\gc 1,3\dc \times \gc 1,3\dc$ submatrices of $w_\circ^3 w_{11}$ have ranks $0$, $1$, and $2$, respectively. Hence, point (a) of Remarks \ref{Mwexplain} immediately implies that $\mcw$ contains the following minors:
$$[2|1],\, [3|1],\, [2,3|1,2],\, [2,3|1,3],\, [2,3|2,3],\, [1,2,3|1,2,3].$$
In fact, the only other minor of type 1 in $\mcw$ is $[2,3|1,4]$.
Of course, some of these are consequences of the others -- once we have $[2|1], [3|1] \in \mcw$, we must have $[2,3|1,2], [2,3|1,3], [2,3 |1,4] \in \mcw$, and once we have $[2|1], [3|1], [2,3|2,3] \in \mcw$, we must have $[1,2,3|1,2,3] \in \mcw$.

Similarly, the minors of type 2 in $\mcw$ are $[1|3],\, [1|4],\, [2|4],\, [1,2|1,4],\, [1,2|2,4]$, $[1,2|3,4],\, [1,3|3,4]$. 
Only $[1,2,3|1,2,3]$ occurs in type 3 (as a vertical stripe), and no horizontal stripes occur because $w_{12}$ is a zero matrix. Thus,
\begin{align*}
\mcw &= \{ [2|1],\, [3|1],\, [1|3],\, [1|4],\, [2|4],\, [1,2|1,4],\, [1,2|2,4],\, [1,2|3,4],\, [1,3|3,4], \\
 &\qquad [2,3|1,2],\, [2,3|1,3],\, [2,3|2,3],\, [2,3|1,4],\, [1,2,3|1,2,3] \}.
 \end{align*}
 
{\bf (b)} Now take $w= \begin{bmatrix} 1&2&3&4&5&6&7&8\\ 1&3&6&4&5&2&7&8 \end{bmatrix}$, and let $m=p=4$. In this case, the submatrix $w_{21}$ of \eqref{blockform1} has the form $w_{21}= \begin{bmatrix} 0&0&0&0\\[-0.8ex] 0&0&1&0\\[-0.8ex] 0&0&0&0\\[-0.8ex] 0&0&0&0 \end{bmatrix}$, and the submatrices $w_{21}[r,\dots,4;r,\dots,s]$ have rank 1 precisely when $r\le 2$ and $s\ge 3$. It follows from the original Condition 3 that the minors of type 3 in $\mcw$ are those which vanish on the matrices $x$ for which $x[1,\dots,4;r,\dots,s]$ has rank at most $s-r$ with $r\le 2$ and $s\ge 3$. These consist of four vertical stripes: $[-|1,2,3]$,  $[-|1,2,3,4]$, $[-|2,3]$, $[-|2,3,4]$. Of course, once we have $[-|2,3] \subseteq \mcw$, the other three stripes must be in $\mcw$ as well.
 
Similarly, the minors of type 4 in $\mcw$ consist of eight horizontal stripes: 
$$[3|-],\, [1,3|-],\, [2,3|-],\, [3,4|-],\, [1,2,3|-],\, [1,3,4|-], \, 
[2,3,4|-],\, [1,2,3,4|-].$$
As it turns out, the minors of types 1 and 2 in $\mcw$ are already included among those of types 3 and 4. Thus, $\mcw$ equals the union of the above stripes.
\end{examples} 

As a result of the foregoing discussion, \cite[Theorem 4.2]{bgy} may be
reformulated in the following way. 

\begin{theo}\label{theorem-4.2reformulation} 
Let $w\in \mathcal{S}$. Then
\[
\overline{\mathcal{P}_{\woN w}}
=\{x\in \Mmpc\mid [I|J](x)=0~{\rm for~all}~[I|J]\in\mcw\}.
\]
\end{theo} 

Geometrically, Theorem \ref{theorem-4.2reformulation} only shows that the minors in $\mc(w)$ cut out the subvariety $\overline{\mathcal{P}_{\woN w}}$, i.e., the ideal of polynomial functions vanishing on $\overline{\mathcal{P}_{\woN w}}$ is the smallest radical ideal of $\pmmpc$ containing $\mc(w)$. In fact, taking account of recent work of Yakimov \cite[Theorem 5.3]{Yak}, this ideal is generated by $\mc(w)$.

Our aim is to show that the minors that vanish on
$\overline{\mathcal{P}_{\woN w}}$ are precisely the minors that are in
$\mcw$. Given the above result, what remains to be proved is that 
if $[I | \Lambda] \notin \mc(w)$ then there exists a
matrix $x \in \overline{P_{\woN w}}$ such that $[I | \Lambda](x) \neq 0$.

We start with a preliminary result.

\begin{theo}\label{theorem-preliminary} 
Let $w\in \mathcal{S}$, and suppose that 
$$[U|V]:=[u_1 < \dots < u_{k+1}\mid v_1 <
\dots < v_{k+1}]$$ 
is a minor in $\pmmpc$ which
is not in $\mc(w)$. 
Let 
$$[U'|V'] = [u_{1} < \dots < \widehat{u_{\alpha}} < \dots < u_{k+1}
\mid v_{1} < \dots < \widehat{v_{\alpha}} < \dots < v_{k+1}],$$
 for some 
$\alpha \in\gc 1,\, k+1\dc$. Then 
$[U'|V']\notin \mc(w)$. 
\end{theo}

\begin{proof}
As $[U | V] \notin \mc(w)$, each of the four conditions in
Definition~\ref{definition-M(w)} is not satisfied by $[U | V]$. It is
necessary to show that each of these four conditions also fails for 
$[U'| V']$.\\

\noindent
$\bullet$ As Condition 1 of Definition~\ref{definition-M(w)} fails for $[U
| V]$, there exists 
$$L=\{l_1 < \dots < l_{k+1}\} \subseteq \onep \cap
w^{-1}\onem$$
such that $L \leq V$ and $U \leq \wom w (L):=\{m_1 <
\dots < m_{k+1} \}$. Let $\gamma \geq \alpha$ such that $ \wom w(l_{\gamma})$
is minimal. Then we set 
$$
L':= \{ l_1 < \dots < \widehat{l_{\gamma}} < \dots <
l_{k+1}\}.
$$ 
We show that $L'$ demonstrates the failure of Condition 1 for
$[U'| V']$.

Note first that $L' \subset L \subseteq \onep \cap
w^{-1}\onem$.
As $L \leq V$ and $\gamma \geq \alpha$, we also have $L' \leq V'$.

Let $\delta$ be such that $m_\delta = \wom w(l_{\gamma})$. Then 
\begin{align*}
\wom w (L') &=\{m'_1<m'_2\dots<m'_{\delta-1} <m'_{\delta}<\dots<m'_{k}\} \notag\\
 &=\{m_1<m_2\dots<\widehat{m_\delta} <\dots<m_{k+1}\}. 
\end{align*}
So, $m'_i = m_i$ for $i\leq \delta -1$ and 
$m'_i = m_{i+1}$ for $i\geq\delta$.
Note that, as $m_\delta$ is the least element in the set  
$\{\wom w (l_i)\mid i\geq \alpha\}$, we must have 
$\delta\leq\alpha$.

Now write
\[
U' := 
\{u'_{1} < \dots < u'_{\alpha-1} <u'_{\alpha} < \dots <
u'_{k}\} = 
\{u_{1} < \dots < \widehat{u_{\alpha}} < \dots <
u_{k+1}\}.
\]
So, 
$u'_i = u_i$ for $i\leq \alpha -1$ and 
$u'_i = u_{i+1}$ for $i\geq\alpha$. Consequently,
\begin{enumerate}
\item[(i)] when $i\leq\delta -1$, we also know that $i\leq\alpha-1$, and so
$u'_i = u_i\leq m_i=m'_i$;
\item[(ii)] when $\delta\leq i\leq\alpha-1$, we have $u'_i = u_i\leq m_i<m_{i+1} 
=m'_i$;
\item[(iii)] finally, when $i\geq\alpha$, note that $i\geq\delta$ and so
$u'_i = u_{i+1}\leq m_{i+1} =m'_i$. 
\end{enumerate}
Thus, $u'_i\leq m'_i$ for all $i=1,\dots,k$ and so $U'\leq 
\wom w (L')$, as required. 

This establishes that Condition 1 of Definition~\ref{definition-M(w)} fails
for $[U'| V']$. \\

\noindent
$\bullet$ Similar arguments show that 
Condition 2 of Definition~\ref{definition-M(w)} fails
for $[U'| V']$. \\

\noindent
$\bullet$ Assume that Condition 3 of Definition~\ref{definition-M(w)} holds for
$[U'| V']$. Then  there exist $1\leq r \leq s\leq p$ such that
 $$|V'\cap \gc r, s \dc| >| \gc r, s \dc \setminus
w^{-1} \gc m+r , \, m+s\dc |. $$
Since $V' \subseteq V$, this contradicts the fact that Condition 3
fails for $[U| V]$. Thus, Condition 3 fails for $[U'|V']$.\\

\noindent
$\bullet$ Similar arguments show that 
Condition 4 of Definition~\ref{definition-M(w)} fails
for $[U'| V']$. \\

Thus, all four conditions in Definition~\ref{definition-M(w)} fail for 
$[U'| V']$, and therefore $[U'| V']\not\in\mcw$. 
\end{proof}

\begin{theo}
\label{theo:descriptionclosure}
Let $w \in \mathcal{S}$, and let $I\subseteq \onem$ and $\Lambda\subseteq
\onep$ with $|I|=|\Lambda|$. Then $[I | \Lambda ] \in \mc(w)$ if
and only if
$[I | \Lambda ](x)=0$ for all $x \in \overline{\mathcal{P}_{\woN w}}$.
\end{theo}

\begin{proof}
If $[I | \Lambda ] \in \mc(w)$, then 
$[I | \Lambda ](x)=0$ for all $x \in \overline{\mathcal{P}_{\woN w}}$ by 
Theorem~\ref{theorem-4.2reformulation}. 

Next, suppose that $[I | \Lambda ] \notin \mc(w)$. We
show that there is a matrix $x \in \overline{P_{\woN w}}$ such that $[I|
\Lambda](x) \neq 0$.
Suppose that $I=\{i_1 < \dots < i_t\} $ and 
$\Lambda=\{j_1 < \dots < j_t\}$. Let $x$
be the matrix whose entries are defined by: $x_{i_k,j_k}=1$ for all
$k=1,\dots,t$ and $x_{k,l}=0$ otherwise. Then clearly 
$[I | \Lambda](x)=1\neq 0$;
so it is enough to show that $x\in\overline{P_{\woN w}}$. 

By Theorem~\ref{theorem-4.2reformulation}, 
it is enough to prove that $[A |B](x) = 0$ 
for all $[A|B] \in \mc(w)$. In fact, we will prove that
if $[A | B](x) \neq 0$, then $[A|B] \notin \mc(w)$. 

Suppose that $[A| B]$ is a minor such that 
$[A | B](x) \neq 0$. 
First, observe that because $[A | B](x) \neq 0$, we must have $[A| B]= 
[i_{a_1} < \dots < i_{a_k} \mid j_{a_1} <\dots < j_{a_k}]$
for some $a_1, \dots,a_k \in \gc 1,t \dc$.
By applying Theorem~\ref{theorem-preliminary} repeatedly,
starting with $[I| \Lambda] \notin \mcw$, we conclude that $[A| B] \notin
\mcw$.  
\end{proof}

\begin{cor}
\label{Mwcontain}
Let $w,z\in \mathcal{S}$. Then $\mc(w) \subseteq \mc(z)$ if and
only if $w\le z$.
\end{cor}

\begin{proof} If $w\le z$, then $\woN z \le \woN w$, and so $\overline{P_{\woN z}}
\subseteq \overline{P_{\woN w}}$ by \cite[Theorem 3.13]{bgy}. In view of Theorem
\ref{theo:descriptionclosure}, it follows that all the minors $[I|\Lambda] \in
\mc(w)$ vanish on $\overline{P_{\woN z}}$, and thus that  $\mc(w)
\subseteq
\mc(z)$.

Conversely, assume that $\mc(w) \subseteq \mc(z)$. By Theorem
\ref{theo:descriptionclosure}, all the minors in $\mc(w)$ vanish on
$\overline{P_{\woN z}}$, and thus Theorem \ref{theorem-4.2reformulation} implies
that $\overline{P_{\woN z}}
\subseteq \overline{P_{\woN w}}$. Now $\woN z \le \woN w$ by \cite[Theorem
3.13]{bgy}, and therefore $w\le z$.
\end{proof}

\subsection{On Poisson $\hc$-prime ideals of $\pmmpc$.}
\label{section:PoissonHprimes}

To conclude Section \ref{sectionPoisson}, let us mention that the symplectic leaves in
$\mc_{m,p}(\bbC)$ are algebraic \cite[Theorem 0.4]{bgy}, that is, they are locally closed subvarieties of $\mc_{m,p}(\bbC)$. As a consequence,
\cite[Proposition 4.8]{Good2006} applies in this situation: there are
only finitely many Poisson $\hc$-primes in $\pmmpc$, and they are the ideals
$$J_w := \{ f\in \pmmpc \mid f=0 \text{\ on\ } \ov{\mathcal{P}_{\woN w}}\, \}$$
where $w\in \mathcal{S}$, that is, where $\mathcal{P}_{\woN w}$ runs through the
$\hc$-orbits of symplectic leaves in $\pmmpc$. In particular, the set of Poisson
$\hc$-primes in $\pmmpc$ is in bijection with $\mathcal{S}$. As a consequence,
thanks to
\cite[Corollary 1.5]{lau2}, the number of Poisson $\hc$-primes in $\pmmpc$ is the
same as the number of $m \times p$ Cauchon diagrams. 

In a recent and independent preprint, Yakimov proves that $J_w$ is generated by minors \cite[Theorem 5.3]{Yak}. He also gives an explicit description of all the elements of $J_w$, as matrix coefficients of Demazure modules \cite[Theorem 4.6]{Yak}. (These results are obtained from much more general ones, concerning the ideals of the closures of torus orbits of symplectic leaves in Schubert cells of arbitrary flag varieties.) For our purposes here, we do not need to know a generating set for $J_w$. On the other hand, we must pin down the complete set of all minors contained in $J_w$, a set that is generally much larger than the set of generators given in \cite{Yak}. The required result, which we deduce from the previous
discussion and Theorem \ref{theo:descriptionclosure}, is the following. 

\begin{theo}
\label{theo:PoissonMinors}
Let $w \in \mathcal{S}$. Then $J_w$ is the unique Poisson $\hc$-prime ideal
of $\pmmpc$ such that the set of minors that belong to $J_w$ is exactly
$\mc(w)$.
\end{theo}

\begin{proof} By definition of $J_w$, the minors that belong to $J_w$ are exactly
those that vanish on the closure of $\mathcal{P}_{\woN w}$. Hence, the statement about minors
follows from Theorem \ref{theo:descriptionclosure}. We then obtain the uniqueness of $J_w$ from Corollary \ref{Mwcontain}, since $\mc(z)\ne \mc(w)$ for all $z\in \mathcal{S}$ with $z\ne w$.
\end{proof}

In the next section, we will describe the restoration algorithm that will
allow us to:
\begin{enumerate}
\item Describe a new algorithm that constructs totally nonnegative matrices from matrices with nonnegative coefficients.
\item Construct an explicit bijection between Cauchon diagrams and Poisson $\hc$-primes.
\item Prove that a family of minors is admissible if and only if it is the set of
minors that vanish on the closure of some $\hc$-orbit of symplectic leaves.
\end{enumerate} 
As a consequence of point 3 and Theorem \ref{theo:descriptionclosure}, we will
obtain an explicit description of the admissible families defining nonempty
tnn cells.\\


\section{The restoration algorithm.}

In this section, $K$ denotes a field of characteristic zero and, except where
otherwise stated, all the matrices considered have their entries in $K$.

\subsection{Description and origin of the restoration algorithm.}

The deleting derivations algorithm was introduced by Cauchon in
\cite{c2} in order to study the prime spectrum of the algebra of quantum
matrices. 
The restoration algorithm, which is the inverse of the deleting derivations
algorithm, was then introduced in \cite{lauPEMS} in
this framework. However, in this paper,  we  restrict ourselves to the
commutative setting.

In order to define the restoration algorithm, we will need the following 
convention.

\begin{nota}$ $
\begin{itemize}

\item The lexicographic ordering on $\bbN^2$ is denoted by $\leq$. 
Recall that:
\[
(\ia) \leq (j,\beta) \quad\Longleftrightarrow\quad
(i < j) \mbox{ or } (i=j \mbox{ and } \alpha \leq \beta ).
\]

\item Set $E:=\left(\gc 1,m \dc \times \gc 1,p \dc \cup \{(m,p+1)\} \right)
\setminus \{(1,1)\}$.

\item Set $\Ecirc= \left(\gc 1,m \dc \times \gc 1,p \dc \right)
\setminus \{(1,1)\}$.

\item Let $(j,\beta) \in \Ecirc$. 
Then $(j,\beta)^{+}$ denotes the smallest element 
(relative to $\leq$) of the set 
$\left\{ (\ia) \in E \ | \ (j,\beta) < (\ia) \right\}$.

\end{itemize}
 \end{nota}

\begin{conv}[Restoration algorithm]
\label{conv1}$ $\\
Let $X=(x_{\ia}) \in \MmpK$. 
As $r$ runs over the set $E$, we define matrices 
$X^{(r)} :=(x_{\ia}^{(r)}) \in \MmpK$ 
as follows:
\begin{enumerate}

\item \underline{When $r=(1,2)$}, we set $X^{(1,2)}=X$, that is, $x_{\ia}^{(1,2)}:=x_{\ia}$ for all $(\ia) \in \gc 1,m \dc \times
\gc 1,p \dc$. 

\item \underline{Assume that $r=(j,\beta) \in \Ecirc$} and
that the matrix $X^{(r)}=(x_{\ia}^{(r)})$ is already known. The entries
$x_{\ia}^{(r^+)}$ of the matrix $X^{(r^+)}$ are defined as follows:

\begin{enumerate}

\item If $x_{j,\beta}^{(r)}=0$, then $x_{\ia}^{(r^+)}=x_{\ia}^{(r)}$ 
for all $(\ia) \in \gc 1,m \dc \times \gc 1,p \dc$.

\item If $x_{j,\beta}^{(r)}\neq 0$ and 
$(\ia) \in \gc 1,m \dc \times \gc 1,p
\dc$, then \\
$x_{\ia}^{(r^+)}= \left\{ \begin{array}{ll}
x_{\ia}^{(r)}+x_{i,\beta}^{(r)} \left( x_{j,\beta}^{(r)}\right)^{-1}
x_{j,\alpha}^{(r)} & \qquad \mbox{if }i <j \mbox{ and } \alpha < \beta \\
x_{\ia}^{(r)} & \qquad \mbox{otherwise.} \end{array} \right.$ 

\end{enumerate}

We say that \textbf{$X^{(r)}$ is the matrix obtained from $X$ by applying the
restoration algorithm at step $r$}, and $x_{j,\beta}^{(r)}$ is called the 
{\bf pivot
at step $r$}. 

\item Set $\overline{X}:= X^{(m,p+1)}$; this is the matrix
obtained from $X$ at the end of the restoration algorithm.
\end{enumerate}
\end{conv}

\begin{example}
If $X= \begin{bmatrix} 0 & 1 \\ 2 & 3 \end{bmatrix}$, then
$\overline{X}= \begin{bmatrix} 2/3 & 1 \\ 2 & 3 \end{bmatrix}$. 
Observe that in this example, $X$ is not totally nonnegative, while
$\overline{X}$ is. This observation will be generalised in due course.
\end{example}

Observe also that the construction in Convention \ref{conv1} is closely related to minors. Indeed, if
$i <j$, $\alpha < \beta$ and 
$x_{j,\beta}^{(r)}\neq 0$ then  

\begin{eqnarray}
\label{eq:restorationMinor}
x_{\ia}^{(r)} 
= x_{\ia}^{(r^+)}-x_{i,\beta}^{(r^+)} \left(
x_{j,\beta}^{(r^+)}\right)^{-1} x_{j,\alpha}^{(r^+)} 
= \det \left(
\begin{array}{ll} x_{\ia}^{(r^+)} & x_{i,\beta}^{(r^+)} \\[1ex]
x_{j,\alpha}^{(r^+)} & x_{j,\beta}^{(r^+)} 
\end{array} \right) \times\left(
x_{j,\beta}^{(r^+)}\right)^{-1}. 
\end{eqnarray}

More generally, the formulae of Convention \ref{conv1} allow us 
to express the entries of $X^{(r)}$ in terms of those of $X^{(r^+)}$, as follows. These expressions also describe the deleting derivations algorithm (see Convention \ref{conv2}).

\begin{obs}
\label{deleting}
Let $X=(x_{\ia}) \in \MmpK$, and
let 
$r=(j,\beta) \in \Ecirc$. 
\begin{enumerate}

\item If $x_{j,\beta}^{(r^+)}=0$, then $x_{\ia}^{(r)}=x_{\ia}^{(r^+)}$ 
for all $(\ia) \in \gc 1,m \dc \times \gc 1,p \dc$.

\item If $x_{j,\beta}^{(r^+)}\neq 0$ and 
$(\ia) \in \gc 1,m \dc \times \gc 1,p \dc$, then 
$$x_{\ia}^{(r)}= 
\left\{ \begin{array}{ll}
x_{\ia}^{(r^+)}-x_{i,\beta}^{(r^+)} 
\left( x_{j,\beta}^{(r^+)}\right)^{-1} 
x_{j,\alpha}^{(r^+)}
& \qquad \mbox{if }i <j \mbox{ and } \alpha < \beta \\[1ex] 
x_{\ia}^{(r^+)} & \qquad \mbox{otherwise.} 
\end{array} \right.$$
\end{enumerate}
\end{obs}

The following proposition is easily obtained from the definition of the
restoration algorithm.

\begin{prop}$ $
\label{ObservationDeletingDerivation}
Let $X=(x_{\ia}) \in \MmpK$, and
$(j,\beta) \in E$.
\begin{enumerate}
\item $x_{\ia}^{(k,\gamma)}=x_{\ia}$ for all $(k,\gamma) \leq (i+1,1)$. In
particular, $x^{(\ia)}_{\ia}= x^{(\ia)^+}_{\ia}= x_{\ia}$ if $(\ia) \in \Ecirc$.

\item If $x_{i, \beta}=0$ for all $i \leq j$, then
$x_{i,\beta}^{(k,\gamma)}=0$ for all $i \leq j$ and $(k,\gamma) \leq
(j,\beta)^+$. 

\item If $x_{j, \alpha}=0$ for all $\alpha \leq \beta$, then
$x_{j,\alpha}^{(k,\gamma)}=0$ for all $\alpha \leq \beta$ and $(k,\gamma) \leq
(j,\beta)^+$. 

\item $x_{\ia}^{(j,\beta)} =x_{\ia} +Q^{(j,\beta)}_{\ia}$, where
$$Q^{(j,\beta)}_{\ia} \in \bbQ[ x_{k,\gamma}^{\pm1} \mid (\ia) < (k,\gamma) < (j,\beta),\ x_{k,\gamma} \ne 0 ].$$
\end{enumerate}
\end{prop}

\begin{proof} Parts 
1, 2 and 3 are easily proved by induction on $(k,\gamma)$;
we concentrate on 4. 
The proof is by induction on $(j,\beta)$. 
The induction starts with  
$(j,\beta)=(1,2)$ where $x_{\ia}^{(1,2)} =x_{\ia}$ 
by construction, as desired. 

Assume now that $(j,\beta) \in \Ecirc$ and that $x_{\ia}^{(j,\beta)} =x_{\ia} +Q_1$, 
where 
$$Q_1 \in \bbQ[ x_{k,\gamma}^{\pm1} \mid (\ia) < (k,\gamma) < (j,\beta),\ x_{k,\gamma} \ne 0 ].$$

We distinguish between two cases.
\\
$\bullet$ If $x_{i,\alpha}^{(j,\beta)^+}=x_{i,\alpha}^{(j,\beta)}$, 
the induction hypothesis immediately implies that $x_{\ia}^{(j,\beta)^+} =x_{\ia}+Q$ for a suitable $Q$.
\\
$\bullet$ If $x_{i,\alpha}^{(j,\beta)^+} \neq x_{i,\alpha}^{(j,\beta)}$, 
then $x_{j,\beta}=x_{j,\beta}^{(j,\beta)}=x_{j,\beta}^{(j,\beta)^+}$ 
is nonzero, and 
\[
x_{i,\alpha}^{(j,\beta)^+}
=x_{i,\alpha}^{(j,\beta)}+x_{i,\beta}^{(j,\beta)}
x_{j,\beta}^{-1}\,x_{j,\alpha}^{(j,\beta)}.
\]
Hence, it follows from the induction hypothesis that
\[
x_{i,\alpha}^{(j,\beta)^+}=x_{i,\alpha}+Q_1+Q_2\,x_{j,\beta}^{-1}\,Q_3,
\]
where each $Q_l$ belongs to $\bbQ[ x_{k,\gamma}^{\pm1} \mid (\ia) < (k,\gamma) < (j,\beta),\ x_{k,\gamma} \ne 0 ]$. Thus,
$x_{i,\alpha}^{(j,\beta)^+}=x_{i,\alpha}+Q$, where
$$Q=Q_1+Q_2\,x_{j,\beta}^{-1}\,Q_3 \in \bbQ[ x_{k,\gamma}^{\pm1} \mid (\ia) < (k,\gamma) < (j,\beta)^+,\ x_{k,\gamma} \ne 0 ],$$ as desired. This concludes the
induction step.
\end{proof}

\subsection{The effect of the restoration algorithm on minors.}
\label{effect1}

Let $X$ be an $m \times p$ matrix. The aim of this section is to obtain a
characterisation of the minors of $X^{(j,\beta)^+}$ that are equal to zero in
terms of the minors of $X^{(j,\beta)}$ that are equal to zero.  

We start by introducing some notation for the minors.

\begin{nota}
\label{notadet}
Let $X=(x_{\ia})$ be a matrix in $\MmpK$, and 
$\delta= [I|\Lambda](X)$ 
a  minor of $X$. 
If $(j,\beta) \in E$, set
\[ 
\delta^{(j,\beta)}:= [I|\Lambda](X^{(j,\beta)}).
\]
For $i\in I$ and $\alpha\in \Lambda$, set 
\[
\delta_{\widehat{i},\widehat{\alpha}}^{(j,\beta)}
:= [I\setminus\{i\} \mid \Lambda\setminus\{\alpha\}](X^{(j,\beta)}),
\] 
while 
\[
\delta_{\alpha \rightarrow \gamma}^{(j,\beta)}
:= [I \mid \Lambda\cup \{\gamma\} \setminus\{\alpha\}](X^{(j,\beta)}) \quad (\gamma
\notin \Lambda)
\]
and 
\[
\delta_{i \rightarrow k}^{(j,\beta)}
:= [I\cup \{k\} \setminus\{i\} \mid \Lambda](X^{(j,\beta)}) \quad (k \notin I).
\]
\end{nota}

We start by studying minors involving the pivot.

\begin{prop} 
\label{Pivot1}
Let $X=(x_{i,\alpha}) \in \MmpK$ and 
 $(j,\beta) \in \Ecirc$. 
 Set $u:=x_{j,\beta}^{(j,\beta)^+} =x_{j,\beta}$ and assume that $u \neq 0$. 
Let 
$\delta= [i_1,\dots,i_l|\alpha_1,\dots,\alpha_l](X)$  be a minor of $X$ with
$(i_l,\alpha_l)  = (j,\beta)$.
Then $\delta^{(j,\beta)^+}=\delta_{\widehat{j},\widehat{\beta}}^{(j,\beta)} u$.
\end{prop}

\begin{proof} This is a consequence of Sylvester's identity. 
More precisely, 
it follows from (\ref{eq:restorationMinor}) that, with $r=(j,\beta)$, 
\begin{align*}
\delta_{\widehat{j},\widehat{\beta}}^{(j,\beta)} 
&= \det \left( \det \left(
\begin{array}{ll} x_{\ia}^{(r^+)} & x_{i,\beta}^{(r^+)} \\[1ex]
x_{j,\alpha}^{(r^+)} & x_{j,\beta}^{(r^+)} \end{array} \right) \times 
u^{-1} 
\right)_{\substack{i=i_1,\dots,i_{l-1} \\
\alpha=\alpha_1, \dots,\alpha_{l-1}}}  \\
 &= \det \left( \det \left(
\begin{array}{ll} x_{\ia}^{(r^+)} & x_{i,\beta}^{(r^+)} \\[1ex]
x_{j,\alpha}^{(r^+)} & x_{j,\beta}^{(r^+)} \end{array} \right)
\right)_{\substack{i=i_1,\dots,i_{l-1} \\ \alpha=\alpha_1,
\dots,\alpha_{l-1}}}\times\;
u^{-(l-1)}. 
\end{align*}
Now, it follows from Sylvester's identity \cite[p 33]{gantmacher} that:

\[ \det \left( \det \left( \begin{array}{ll} x_{\ia}^{(r^+)} &
x_{i,\beta}^{(r^+)} \\[1ex] 
x_{j,\alpha}^{(r^+)} & x_{j,\beta}^{(r^+)} \end{array}
\right) \right)_{\substack{i=i_1,\dots,i_{l-1} \\ \alpha=\alpha_1,
\dots,\alpha_{l-1}}} = \delta^{(j,\beta)^+} \times
u^{l-2}. 
\]
The result easily follows from these last two equalities.
\end{proof}

The following result is a direct consequence of the previous proposition. 

\begin{cor} 
\label{cor:part1}
Let $X=(x_{i,\alpha}) \in \MmpK$ and 
 $(j,\beta) \in \Ecirc$, and let 
$$\delta=[i_1,\dots,i_l|\alpha_1,\dots,\alpha_l](X)$$
be a minor of $X$ with
$(i_l,\alpha_l)  = (j,\beta)$.
 If $\delta^{(j,\beta)^+}=0$, then $x_{j,\beta}=0$ or 
$\delta_{\widehat{j},\widehat{\beta}}^{(j,\beta)}=0$.
\end{cor}

The converse of this result is not true in general. However, it does hold for the following class of matrices.

\begin{defi}
Let $X=(x_{i,\alpha}) \in \MmpK$
and $C$ a Cauchon diagram (of size $m\times p$). We say that $X$ is a \emph{Cauchon
matrix associated to the Cauchon
diagram $C$} provided that for all $(\ia) \in \gc 1,m \dc \times \gc 1,p \dc$, we
have
$x_{\ia}=0$ if and only if $(\ia) \in C$. If $X$ is a Cauchon matrix associated to
an unnamed Cauchon diagram, we just say that $X$ is a \emph{Cauchon matrix}.
\end{defi}

A key link with tnn matrices is the easily observed fact that every tnn matrix is a Cauchon matrix (Lemma \ref{tnnCauchon}).

\begin{prop} 
\label{propPivot}
Let $X=(x_{i,\alpha}) \in \MmpK$ 
be a  Cauchon matrix, and let
 $(j,\beta) \in \Ecirc$. Set $u:=x_{j,\beta}^{(j,\beta)^+} =x_{j,\beta}$, and let
$\delta= [i_1,\dots,i_l|\alpha_1,\dots,\alpha_l](X)$  be a minor of $X$ with
$(i_l,\alpha_l)  = (j,\beta)$.
Then $\delta^{(j,\beta)^+}=\delta_{\widehat{j},\widehat{\beta}}^{(j,\beta)} u$; 
so that 
$\delta^{(j,\beta)^+}=0$ if and only if $u=0$ 
or $\delta_{\widehat{j},\widehat{\beta}}^{(j,\beta)}=0$.
\end{prop}

\begin{proof} 
It only remains to prove that if $u=0$ 
then $\delta^{(j,\beta)^+}=0$. 

Assume that $u=0$. As $X$ is a Cauchon
matrix, this implies that either $x_{i,\beta}=0$ for all $i \leq j$ or
$x_{j,\alpha}=0$ for all $\alpha \leq \beta$. Now, it follows from
Proposition \ref{ObservationDeletingDerivation} that either
$x_{i,\beta}^{(j,\beta)^+}=0$ for all $i \leq j$ or
$x_{j,\alpha}^{(j,\beta)^+}=0$ for all $\alpha \leq \beta$. Of course, in both
cases we get $\delta^{(j,\beta)^+}= 0$ as
either the last column or the last row of the submatrix
$X^{(j,\beta)^+}[i_1,\dots,i_l;\alpha_1,\dots,\alpha_l]$ is zero. 
\end{proof}

The formulae for the deleting derivations algorithm and the restoration
algorithm show how the individual elements of a matrix change during the
running of the algorithms. As we are concerned with arbitrary minors for much
of the time, we need more general formulae that apply to minors. These are 
given in the following results.

\begin{prop} 
\label{Form0Start}
 Let $X=(x_{i,\alpha}) \in \MmpK$ {\rm(}not necessarily a Cauchon matrix{\rm)} and $(j,\beta) \in \Ecirc$. Set
$u:=x_{j,\beta}$ 
 and let $\delta
 =[i_1,\dots,i_l|\alpha_1,\dots,\alpha_l](X)$  
be a minor of $X$ with $(i_l,\alpha_l) < (j,\beta)$. 
If $u=0$, or if $i_l=j$, or if $\beta \in \{\alpha_1,\dots,\alpha_l\}$, or if
$\beta  < \alpha_1$, then $\delta^{(j,\beta)^+}=\delta^{(j,\beta)}$.
\end{prop}

\begin{proof}
If $u=0$, then $x_{\ia}^{(j,\beta)^+}=x_{\ia}^{(j,\beta)}$ 
for all $(\ia)$; so the result is clear. Hence, we assume that $u \neq 0$. If
$\beta \leq \alpha_1$, then $x_{\ia}^{(j,\beta)^+}=x_{\ia}^{(j,\beta)}$ for all $i$ and
all $\alpha\ge \alpha_1$, and again the result is clear.

Next, assume that $i_l=j$. Then $\alpha_l < \beta $ because $(i_l,\alpha_l) <
(j,\beta)$. We proceed by induction on $l$. In case $l=1$, we have
$$\delta^{(j,\beta)^+}= x_{j,\alpha_1}^{(j,\beta)^+}= x_{j,\alpha_1}^{(j,\beta)}= \delta^{(j,\beta)}.$$
Now suppose that $l>1$ and that the result holds in the $(l-1)\times(l-1)$ case. By the induction hypothesis, $\delta_{\widehat{i_1},\widehat{\alpha_k}}^{(j,\beta)^+} = \delta_{\widehat{i_1},\widehat{\alpha_k}}^{(j,\beta)}$ for all $k\in \onel$. Since $i_1<j$ and all $\alpha_k <\beta$, expansion of the minor $\delta^{(j,\beta)^+}$ along its first row yields
\begin{align*}
\delta^{(j,\beta)^+} &= \sum_{k=1}^l (-1)^{k+1} x_{i_1,\alpha_k}^{(j,\beta)^+} \delta_{\widehat{i_1},\widehat{\alpha_k}}^{(j,\beta)^+} \\
 &= \sum_{k=1}^l (-1)^{k+1} \bigl( x_{i_1,\alpha_k}^{(j,\beta)}+ x_{i_1,\beta}^{(j,\beta)} x_{j,\beta}^{-1} x_{j,\alpha_k}^{(j,\beta)} \bigr) \delta_{\widehat{i_1},\widehat{\alpha_k}}^{(j,\beta)} \\
 &= \sum_{k=1}^l (-1)^{k+1} x_{i_1,\alpha_k}^{(j,\beta)} \delta_{\widehat{i_1},\widehat{\alpha_k}}^{(j,\beta)}  + x_{i_1,\beta}^{(j,\beta)} x_{j,\beta}^{-1}  \sum_{k=1}^l (-1)^{k+1} x_{j,\alpha_k}^{(j,\beta)} \delta_{\widehat{i_1},\widehat{\alpha_k}}^{(j,\beta)} = \delta^{(j,\beta)},
 \end{align*}
by Laplace expansions. (The last summation vanishes because it expands a minor whose first and last rows are equal.) This completes the induction.

Finally, assume that $i_l<j$ and $\beta\in \{\alpha_2,\dots,\alpha_l\}$. Note that $l\ge 2$ and $\alpha_l \geq \beta$. We again proceed by induction on $l$. In case $l=2$, we have
\begin{align*}
\delta^{(j,\beta)^+} &= \det \begin{pmatrix} x_{i_1,\alpha_1}^{(j,\beta)^+} &x_{i_1,\beta}^{(j,\beta)^+}\\ x_{i_2,\alpha_1}^{(j,\beta)^+} &x_{i_2,\beta}^{(j,\beta)^+} \end{pmatrix} = \det \begin{pmatrix} x_{i_1,\alpha_1}^{(j,\beta)}+ x_{i_1,\beta}^{(j,\beta)} x_{j,\beta}^{-1} x_{j,\alpha_1}^{(j,\beta)} &x_{i_1,\beta}^{(j,\beta)}\\ x_{i_2,\alpha_1}^{(j,\beta)}+ x_{i_2,\beta}^{(j,\beta)} x_{j,\beta}^{-1} x_{j,\alpha_1}^{(j,\beta)} &x_{i_2,\beta}^{(j,\beta)} \end{pmatrix} \\
 &= \det \begin{pmatrix} x_{i_1,\alpha_1}^{(j,\beta)} &x_{i_1,\beta}^{(j,\beta)}\\ x_{i_2,\alpha_1}^{(j,\beta)} &x_{i_2,\beta}^{(j,\beta)} \end{pmatrix} = \delta^{(j,\beta)}.
 \end{align*}
Now suppose that $l>2$ and that the result holds in the $(l-1)\times(l-1)$ case. By the induction hypothesis, $\delta_{\widehat{i_k},\widehat{\alpha_1}}^{(j,\beta)^+} = \delta_{\widehat{i_k},\widehat{\alpha_1}}^{(j,\beta)}$ for all $k\in \onel$. Since $\alpha_1< \beta$ and all $i_k<j$, expansion of the minor $\delta^{(j,\beta)^+}$ along its first column yields
\begin{align*}
\delta^{(j,\beta)^+} &= \sum_{k=1}^l (-1)^{k+1} x_{i_k,\alpha_1}^{(j,\beta)^+} \delta_{\widehat{i_k},\widehat{\alpha_1}}^{(j,\beta)^+} \\
 &= \sum_{k=1}^l (-1)^{k+1} \bigl( x_{i_k,\alpha_1}^{(j,\beta)}+ x_{i_k,\beta}^{(j,\beta)} x_{j,\beta}^{-1} x_{j,\alpha_1}^{(j,\beta)} \bigr) \delta_{\widehat{i_k},\widehat{\alpha_1}}^{(j,\beta)} \\
 &= \sum_{k=1}^l (-1)^{k+1} x_{i_k,\alpha_1}^{(j,\beta)} \delta_{\widehat{i_k},\widehat{\alpha_1}}^{(j,\beta)}  + x_{j,\beta}^{-1}  x_{j,\alpha_1}^{(j,\beta)} \sum_{k=1}^l (-1)^{k+1} x_{i_k,\beta}^{(j,\beta)} \delta_{\widehat{i_k},\widehat{\alpha_1}}^{(j,\beta)} = \delta^{(j,\beta)},
 \end{align*}
by Laplace expansions. This completes the induction.
\end{proof}

\begin{lem} \label{form0lemma}
Let 
$X=(x_{i,\alpha}) \in \MmpK$
 and  
$(j,\beta) \in \Ecirc$. 
Set $u:=x_{j,\beta}$ and let 
$\delta= [i_1,\dots,i_l|\alpha_1,\dots,\alpha_l](X)$  be a minor of $X$ with 
$i_l < j$ and $\alpha_l < \beta$.
Assume that 
$u \neq 0$. 
Then 
$$\delta^{(j,\beta)^+} = \delta^{(j,\beta)} +
  \sum_{k=1}^l (-1)^{k+l}\delta_{i_k \rightarrow j}^{(j,\beta)}\, 
 x_{i_k,\beta}^{(j,\beta)} u^{-1}.
$$
\end{lem}

\begin{proof}
It follows from Proposition~\ref{Pivot1} that 
\begin{eqnarray*}
\delta^{(j,\beta)}\, u & = & \det \left( 
\begin{array}{ccc}
x_{i_1,\alpha_1}^{(j,\beta)} & \dots & x_{i_1,\alpha_l}^{(j,\beta)} \\[1ex]
\vdots & & \vdots \\[1ex]
x_{i_l,\alpha_1}^{(j,\beta)} & \dots & x_{i_l,\alpha_l}^{(j,\beta)}
\end{array} \right) \times u  \\[2ex]
 & = & \det \left( 
\begin{array}{cccc}
x_{i_1,\alpha_1}^{(j,\beta)^+} & \dots & x_{i_1,\alpha_l}^{(j,\beta)^+} 
& x_{i_1,\beta}^{(j,\beta)^+} \\[1ex]
\vdots & & \vdots &  \vdots \\[1ex]
x_{i_l,\alpha_1}^{(j,\beta)^+} & \dots & x_{i_l,\alpha_l}^{(j,\beta)^+} 
& x_{i_l,\beta}^{(j,\beta)^+} \\[1ex]
x_{j,\alpha_1}^{(j,\beta)^+} & \dots & x_{j,\alpha_l}^{(j,\beta)^+} & u  \\
\end{array} \right).  \\
\end{eqnarray*}
Expanding this determinant along its last column leads to
\begin{eqnarray*}
\delta^{(j,\beta)}\, u & = & u\, \delta^{(j,\beta)^+} + 
\sum_{k=1}^l (-1)^{k+l+1}  x_{i_k,\beta}^{(j,\beta)^+}
  \delta_{i_k \rightarrow j}^{(j,\beta)^+}.
\end{eqnarray*}
By construction, 
$x_{i_k,\beta}^{(j,\beta)^+}=x_{i_k, \beta}^{(j,\beta)}$, 
and it follows from Proposition~\ref{Form0Start} that 
$\delta_{i_k \rightarrow j}^{(j,\beta)^+}
=\delta_{i_k \rightarrow j}^{(j,\beta)}$. 
Hence,   
 \begin{eqnarray*}
 \delta^{(j,\beta)^+} = \delta^{(j,\beta)}
  - \sum_{k=1}^l (-1)^{k+l+1} u^{-1} x_{i_k,\beta}^{(j,\beta)}
  \delta_{i_k \rightarrow j}^{(j,\beta)}.
\end{eqnarray*}
\end{proof}

\begin{prop} 
\label{form0}
Let $X=(x_{i,\alpha}) \in \MmpK$
 and $(j,\beta) \in \Ecirc$.
Set $u:= x_{j,\beta}$, and let $\delta=
[i_1,\dots,i_l|\alpha_1,\dots,\alpha_l](X)$ be a minor of $X$. Assume that $u \neq 0$ and that $i_l < j$ while  $\alpha_h < \beta < \alpha_{h+1}$ for some $h \in \onel$. 
{\rm(}By convention, $\alpha_{l+1}=p+1$.{\rm)} Then
$$\delta^{(j,\beta)^+} = \delta^{(j,\beta)}
  + \delta_{\alpha_h \rightarrow \beta}^{(j,\alpha_h)}\, 
 x_{j,\alpha_h} u^{-1} .$$
\end{prop}

\begin{proof} We proceed by induction on $l+1-h$. If $l+1-h=1$, then $h=l$ and $\alpha_l < \beta$. 
It follows from Lemma \ref{form0lemma} that
$$
\delta^{(j,\beta)^+} = \delta^{(j,\beta)}
  +\sum_{k=1}^l (-1)^{k+l}\delta_{i_k \rightarrow j}^{(j,\beta)}\, 
 x_{i_k,\beta}^{(j,\beta)} u^{-1}.
 $$
Moreover, it follows from Proposition \ref{Form0Start} that
$$
\delta_{i_k \rightarrow j}^{(j,\beta)} 
= \delta_{i_k \rightarrow j}^{(j,\beta-1)} 
= \dots =\delta_{i_k \rightarrow j}^{(j,\alpha_l+1)}. 
$$
Then we deduce from Proposition \ref{Pivot1} that 
$$
\delta_{i_k \rightarrow j}^{(j,\beta)} 
=\delta_{\widehat{i_k}, \widehat{\alpha_l}}^{(j,\alpha_l)} 
x_{j,\alpha_l}.
$$
As  
$x_{i_k,\beta}^{(j,\beta)}
= x_{i_k,\beta}^{(j,\beta-1)}= \dots 
= x_{i_k,\beta}^{(j,\alpha_l)}$ by construction,  we obtain 
$$
\delta^{(j,\beta)^+} 
= \delta^{(j,\beta)} +
  \sum_{k=1}^l (-1)^{k+l}
  \delta_{\widehat{i_k}, \widehat{\alpha_l}}^{(j,\alpha_l)}\, 
  x_{j,\alpha_l}
 x_{i_k,\beta}^{(j,\alpha_l)} u^{-1} .$$
Hence, by using a Laplace expansion, we obtain 
$$
\delta^{(j,\beta)^+} = \delta^{(j,\beta)}
  + \delta_{\alpha_l \rightarrow \beta}^{(j,\alpha_l)}\, 
  x_{j,\alpha_l} u^{-1},
  $$ 
as desired.

Now let $l+1-h>1$, and assume the result holds for smaller values of $l+1-h$. Expand the minor $\delta^{(j,\beta)^+} $ along its last column, to get
$$\delta^{(j,\beta)^+} = \sum_{k=1}^l (-1)^{k+l} x_{i_k,\alpha_l}^{(j,\beta)^+} \delta_{\widehat{i_k}, \widehat{\alpha_l}}^{(j,\beta)^+}.$$
The value corresponding to $l+1-h$ for the minors $\delta_{\widehat{i_k}, \widehat{\alpha_l}}^{(j,\beta)^+}$ is $l-h$, and so the induction hypothesis applies. We obtain
$$\delta_{\widehat{i_k}, \widehat{\alpha_l}}^{(j,\beta)^+}= \delta_{\widehat{i_k}, \widehat{\alpha_l}}^{(j,\beta)}+ \delta^{(j,\alpha_h)}_{\substack{ \widehat{i_k}, \widehat{\alpha_l}\\ \alpha_h \rightarrow \beta }} x_{j,\alpha_h} u^{-1}$$
for $k\in\onel$. As  
$x_{i_k,\alpha_l}^{(j,\beta)^+}
= x_{i_k,\alpha_l}^{(j,\beta)}= \dots 
= x_{i_k,\alpha_l}^{(j,\alpha_h)}$ by construction,  we obtain 
\begin{align*}
\delta^{(j,\beta)^+} &= \sum_{k=1}^l (-1)^{k+l} x_{i_k,\alpha_l}^{(j,\beta)} \delta_{\widehat{i_k}, \widehat{\alpha_l}}^{(j,\beta)} +\sum_{k=1}^l (-1)^{k+l} x_{i_k,\alpha_l}^{(j,\alpha_h)}  \delta^{(j,\alpha_h)}_{\substack{ \widehat{i_k}, \widehat{\alpha_l}\\ \alpha_h \rightarrow \beta }} x_{j,\alpha_h} u^{-1} \\
 &= \delta^{(j,\beta)}
  + \delta_{\alpha_h \rightarrow \beta}^{(j,\alpha_h)}
  x_{j,\alpha_h} u^{-1},
\end{align*}
by two final Laplace expansions. This concludes the induction step.
\end{proof}

Even though Propositions \ref{Form0Start} and \ref{form0} constitute important steps towards a
characterisation of the minors of $X^{(j,\beta)^+}$ that are equal to zero in
terms of the minors of $X^{(j,\beta)}$ that are equal to zero, 
the sum in the last part of Proposition \ref{form0} causes problems. In
order to overcome this, we introduce a new class of matrices in the next
section.

\subsection{The effect of the restoration algorithm on the minors of an 
$\hc$-invariant Cauchon matrix.}  
\label{effect2}

\begin{defi}
Let $X=(x_{\ia}) \in \MmpK$. 
Then $X$
is said to be \emph{$\hc$-invariant} if
$$ \delta ^{(j,\beta)^+}=0 \quad\Longrightarrow\quad\delta^{(j,\beta)}=0$$
for all $(j,\beta) \in
\Ecirc$ and all minors $\delta= [i_1,\dots,i_l|\alpha_1,\dots,\alpha_l](X)$
of $X$ such that $(i_l,\alpha_l) < (j,\beta)$. 
\end{defi}

In the following sections, we will construct several examples of 
 $\hc$-invariant Cauchon matrices. One of the main
examples of an 
$\hc$-invariant Cauchon matrix is the
matrix $(Y_{\ia}+J)$, where the $Y_{\ia}$ denote the canonical generators of
$\pmmpc$, and $J$ is a Poisson $\hc$-prime ideal of this Poisson algebra (see
Section \ref{sectionPoisson}). This is the reason why we use the terminology
``$\hc$-invariant" in the previous definition. \\

When $X$ is an $\hc$-invariant Cauchon matrix, we deduce from Propositions \ref{Form0Start} and
\ref{form0} the following characterisation of the minors of $X^{(j,\beta)^+}$
that are equal to zero in terms of the minors of $X^{(j,\beta)}$ that are
equal to zero.

\begin{prop}
\label{prop:criterion}
Let $X=(x_{i,\alpha}) \in \MmpK$ be 
an $\hc$-invariant  Cauchon matrix,  and let 
 $(j,\beta) \in \Ecirc$. 
 Set $u:= x_{j,\beta}$.
Suppose that 
$\delta= [i_1,\dots,i_l|\alpha_1,\dots,\alpha_l](X)$  
is  a minor of $X$ with $(i_l,\alpha_l) < (j,\beta)$.
\begin{enumerate}

\item Assume that $u=0$. Then $\delta^{(j,\beta)^+}=0 $ if and only if
$\delta^{(j,\beta)}=0$.

\item Assume that $u\neq 0$. 
If $i_l=j$, or if $\beta\in \{\alpha_1,\dots,\alpha_l\}$, or if $\beta < \alpha_1$,
then
$\delta^{(j,\beta)^+}=0$ if and only if $\delta^{(j,\beta)}=0$.

\item Assume that 
$u \neq 0$ and $i_l < j$ while $\alpha_{h} < \beta < \alpha_{h+1}$ 
for some $h \in \onel$.
Then $\delta^{(j,\beta)^+}=0$ 
if and only if $\delta^{(j,\beta)}=0$ and 
either $\delta_{\alpha_h \rightarrow \beta}^{(j,\alpha_h)} =0$ 
or $ x_{j,\alpha_h}=0$.
\end{enumerate}
\end{prop}

\begin{proof} This follows easily from the previous formulae and the 
fact that $X$ is $\hc$-invariant.
\end{proof}

We are now able to prove that the minors of an $\hc$-invariant Cauchon matrix $X$
associated to a Cauchon diagram
$C$ that are equal to zero only depend on the Cauchon diagram $C$ and not on
the matrix $X$ itself. More precisely, we have the following result.

\begin{theo}
\label{independence}
Let $C$ be an $m \times p$ Cauchon diagram. 
Suppose that  $K$ and $L$ are  fields of characteristic $0$.   
Let 
$X=(x_{i,\alpha}) \in \MmpK$ 
and 
$Y=(y_{i,\alpha})  \in \mc_{m,p}(L)$ 
be two matrices. 
Assume that $X$ and $Y$ are both $\hc$-invariant 
Cauchon matrices associated to the same Cauchon diagram $C$.  
 Let $(j,\beta) \in E$, let $\delta= [i_1,\dots,i_l|\alpha_1,\dots,\alpha_l](X)$  be
a minor of $X$ with $(i_l,\alpha_l)  < (j,\beta)$, and let $\Delta=
[i_1,\dots,i_l|\alpha_1,\dots,\alpha_l](Y)$  be the corresponding minor of $Y$.

Then $\delta^{(j,\beta)}=0 $ if and only if $\Delta^{(j,\beta)}=0 $. 
\end{theo}

\begin{proof} It is enough to prove that 
$\delta^{(j,\beta)}=0 $ implies that $\Delta^{(j,\beta)}=0$. 

The proof is  by induction on $(j,\beta)$. 
Assume first that  $(j,\beta)=(1,2)$:  we have to prove that if 
$x_{1,1}^{(1,2)}=0$, then  $y_{1,1}^{(1,2)}=0$. 
Assume that  $x_{1,1}^{(1,2)}=0$. 
Then $x_{1,1}=x_{1,1}^{(1,2)}=0$. 
As $X$ is associated to the Cauchon diagram $C$, 
this implies that $(1,1) \in C$. 
As $Y$ is also associated to $C$, 
it follows that $0=y_{1,1}=y_{1,1}^{(1,2)}$, as desired.

Now let $(j,\beta) \in E$ with $(j,\beta) \neq (m,p+1)$, 
and assume the result proved at step $(j,\beta)$. 
Let $\delta= [i_1,\dots,i_l|\alpha_1,\dots,\alpha_l](X)$  
be a minor of $X$ with $(i_l,\alpha_l) < (j,\beta)^+$, and let
$\Delta= [i_1,\dots,i_l|\alpha_1,\dots,\alpha_l](Y)$ be the corresponding minor of
$Y$. Assume that
$\delta^{(j,\beta)^+}=0
$. In order to prove that  
$\Delta^{(j,\beta)^+}=0 $, we consider 
several cases.

\
\\$\bullet$ Assume that $(i_l,\alpha_l) = (j,\beta)$. 
Then it follows from Proposition \ref{propPivot} that 
$0=\delta^{(j,\beta)^+} 
= \delta_{\widehat{j},\widehat{\beta}}^{(j,\beta)} x_{j,\beta}$, 
so that 
$\delta_{\widehat{j},\widehat{\beta}}^{(j,\beta)}=0$ 
or $x_{j,\beta}=0$. 

If $\delta_{\widehat{j},\widehat{\beta}}^{(j,\beta)}=0$, 
then it follows from the induction hypothesis that 
$\Delta_{\widehat{j},\widehat{\beta}}^{(j,\beta)}=0$. 
As $\Delta^{(j,\beta)^+} 
= \Delta_{\widehat{j},\widehat{\beta}}^{(j,\beta)} y_{j,\beta}$, 
by Proposition \ref{propPivot}, 
it follows that 
$\Delta^{(j,\beta)^+} =0$, as required. 

If  $x_{j,\beta}=0$, then $(j,\beta) \in C$ as $X$ is associated to $C$. 
Now, as $Y$ is associated to $C$ as well, we get $y_{j,\beta}=0$, 
and then it follows from Proposition \ref{propPivot} 
that $\Delta^{(j,\beta)^+} =0$, as required.

\
\\$\bullet$ Assume that $(i_l,\alpha_l) < (j,\beta)$. 
We distinguish between three cases 
(corresponding to the three cases of Proposition 
\ref{prop:criterion}). 

\
\\$\bullet \bullet$  Assume that $x_{j,\beta} = 0$. 
As we are assuming that $\delta^{(j,\beta)^+}=0 $, it follows from 
Proposition \ref{prop:criterion} 
that $\delta^{(j,\beta)}=\delta^{(j,\beta)^+}=0 $. 
Hence, we deduce from the induction hypothesis that  $\Delta^{(j,\beta)}=0$. 
On the other hand, as 
$x_{j,\beta} = 0$, we have $(j,\beta)\in C$ and so $y_{j,\beta}=0$.  Thus, it
follows from Proposition \ref{prop:criterion} that 
$\Delta^{(j,\beta)^+}=\Delta^{(j,\beta)}=0 $, as desired.

\
\\$\bullet \bullet$  Assume that $x_{j,\beta} \neq 0$, 
and that $i_l=j$, or that $\beta\in \{\alpha_1,\dots,\alpha_l\}$,
or that $\beta < \alpha_1$. 
As we are assuming that $\delta^{(j,\beta)^+}=0 $, 
it follows from Proposition 
\ref{prop:criterion} that $\delta^{(j,\beta)}=\delta^{(j,\beta)^+}=0 $. 
Hence, we deduce from the induction hypothesis that  $\Delta^{(j,\beta)}=0$. 
On the other hand, as 
$x_{j,\beta} \neq  0$, we have $(j,\beta) \notin C$ and so $y_{j,\beta} \neq 0$. 
Moreover, as  
$i_l=j$,  or $\beta\in \{\alpha_1,\dots,\alpha_l\}$, 
or $\beta < \alpha_1$, 
it follows from Proposition \ref{prop:criterion} that 
$\Delta^{(j,\beta)^+}=\Delta^{(j,\beta)}=0 $, as desired.

\
\\$\bullet \bullet$ Assume that 
$x_{j,\beta} \neq 0$ and $i_l 
< j$, while $\alpha_{h} < \beta < \alpha_{h+1}$ 
for some $h \in \onel$. 
Then as in the previous case, 
$y_{j,\beta} \neq 0$. 
Moreover, it follows from Proposition \ref{prop:criterion} that 
$\delta^{(j,\beta)^+}=0 $ implies $\delta^{(j,\beta)}=0$ and 
either $\delta_{\alpha_h \rightarrow \beta}^{(j,\alpha_h)} =0$ 
or $ x_{j,\alpha_h} =0$. 
Hence, we deduce from the induction hypothesis that $\Delta^{(j,\beta)}=0$ 
and 
either 
$\Delta_{\alpha_h \rightarrow \beta}^{(j,\alpha_h)} =0$ 
or $ y_{j,\alpha_h} =0$. 
Finally, it follows from Proposition \ref{prop:criterion} that 
$\Delta^{(j,\beta)^+}=0$, as desired. 
\end{proof}

In the case where $(j,\beta)=(m,p+1)$, the previous theorem leads to the
following result. (Recall here that $\overline{X}$ is the matrix obtained from
$X$ at the end of the restoration algorithm.)

\begin{cor}
\label{cor:independence}
Retain the notation of the previous theorem. 
Let $I\subseteq \onem$ and $\Lambda \subseteq \onep$ with $|I|=
|\Lambda|$. Then
$[I|\Lambda](\overline{X})=0
$  if and only if $[I|\Lambda](\overline{Y})=0$.
\end{cor}


\section{The restoration algorithm and totally nonnegative matrices.}
\label{section:restotnn}

Let $N=(n_{i,\alpha}) \in \mc_{m,p}(\bbR)$ and let  
$\overline{N}$ be the matrix obtained from $N$ 
at the end of the restoration algorithm. 

\begin{theo} \label{Nbarnonneg}
Assume that $N$ is a Cauchon matrix and also that $N$ is 
nonneg\-a\-tive; that is, $n_{i,\alpha} \geq 0$ for all $(i,\alpha)$. 
Then $\overline{N}$ is a totally nonnegative matrix.
\end{theo}

\begin{proof}
We will prove by induction on $(j,\beta) \in E$ 
that 
\begin{enumerate} \item[]
\begin{enumerate}
\item[$(*_{j,\beta})$]  For any minor $\delta= [i_1,\dots,i_l|\alpha_1, \dots,\alpha_l](N)$ of $N$ with $(i_l,\alpha_l) < (j,\beta)$, we have $\delta^{(j,\beta)} \ge 0$.  \end{enumerate}
\end{enumerate}

Assume first that 
$(j,\beta)=(1,2)$. Then $\delta^{(j,\beta)}= n_{1,1}^{(1,2)}=n_{1,1} \geq 0$, as 
$N$ is nonneg\-a\-tive. 

Now assume that $(j,\beta) \in \Ecirc$, and that $(*_{j,\beta})$ holds.
Let 
$$\delta^{(j,\beta)^+}
=[i_1,\dots,i_l|\alpha_1, \dots,\alpha_l](N^{(j,\beta)^+})$$ 
be a minor of $N^{(j,\beta)^+}$ with $(i_l,\alpha_l) \leq (j,\beta)$. 
We distinguish between two cases in order 
to prove that $\delta^{(j,\beta)^+} \geq 0$. 

First, assume that $(i_l,\alpha_l) = (j,\beta)$. 
Then 
$\delta^{(j,\beta)^+}
=\delta_{\widehat{j},\widehat{\beta}}^{(j,\beta)} n_{j,\beta}$, by 
Proposition \ref{propPivot}. 
As $\delta_{\widehat{j},\widehat{\beta}}^{(j,\beta)}$ is 
nonnegative by the induction hypothesis and 
$n_{j,\beta}$ is nonnegative by assumption, 
it follows that $\delta^{(j,\beta)^+} \geq 0$ in this case.

Next, assume that $(i_l,\alpha_l) < (j,\beta)$. 
Then it follows from Propositions \ref{Form0Start} and  \ref{form0} 
that either $\delta^{(j,\beta)^+}=\delta^{(j,\beta)}$, 
or  $n_{j,\beta} > 0$ and $i_l <j$ 
while 
$$\delta^{(j,\beta)^+} n_{j,\beta}
= \delta^{(j,\beta)} n_{j,\beta}
  + \delta_{\alpha_h \rightarrow \beta}^{(j,\alpha_h)}  
  n_{j,\alpha_h}$$
  for some $h \in \onel$ such that $\alpha_h< \beta<
\alpha_{h+1}$.  In each of the two cases, it easily follows from the induction
hypothesis  and the assumption that $N$ is nonnegative 
that  $\delta^{(j,\beta)^+} \geq 0$, as desired. This completes the induction step.
 
The final case, where $(j,\beta)=(m,p+1)$, shows that 
every minor of $\overline{N}=N^{(m,p+1)}$ is nonnegative, as required.
\end{proof}

\begin{example}
Set $N:=\begin{bmatrix}
1 & 0 & 1 & 1 \\
0 & 0 & 1 & 1 \\
1 & 1 & 1 & 1 \\
1 & 1 & 1 & 1 \\
\end{bmatrix}$. 
 Clearly, $N$ is a nonnegative matrix associated to the Cauchon diagram 
 of Figure \ref{fig:CauchonDiagram2}.
 
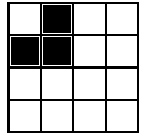
\begin{figure}[h]
\begin{center}
$$\xymatrixrowsep{0.01pc}\xymatrixcolsep{0.01pc}
\xymatrix{
\plc\edge[0,8]\edge[8,0] &&\plc\edge[8,0] &&\plc\edge[8,0]
&&\plc\edge[8,0] &&\plc\edge[8,0]  \\
 & &&\mdblk &&&& && \plc \\
\plc\edge[0,8] &&&&&&&&\plc \\
 &\mdblk &&\mdblk && &&  \\
\plc\edge[0,8] &&&&&&&&\plc \\
 & && && && \\
\plc\edge[0,8] &&&&&&&&\plc \\
 & && && && \\
\plc\edge[0,8] &&\plc &&\plc &&\plc &&\plc 
}$$
\caption{An example of a $4\times 4$ Cauchon diagram}
\label{fig:CauchonDiagram2}
\end{center}
\end{figure}

The previous result shows that  
the matrix $\overline{N}$ is totally nonnegative. 
Six nontrivial steps are needed to compute the matrix $\overline{N}$ 
when  using the restoration algorithm. 
Indeed, here are the detailed calculations:\\
\begin{align}
N^{(2,4)} &=N^{(2,3)}=N^{(2,2)}=N^{(2,1)}=N^{(1,4)}=N^{(1,3)}=N^{(1,2)}
=\begin{bmatrix}
1 & 0 & 1 & 1 \\
0 & 0 & 1 & 1 \\
1 & 1 & 1 & 1 \\
1 & 1 & 1 & 1 \\
\end{bmatrix};  \notag\\
N^{(3,3)} &=N^{(3,2)}=N^{(3,1)}=\begin{bmatrix}
1 & 0 & 2 & 1 \\
0 & 0 & 1 & 1 \\
1 & 1 & 1 & 1 \\
1 & 1 & 1 & 1 \\
\end{bmatrix}; \qquad 
N^{(3,4)}=\begin{bmatrix}
3 & 2 & 2 & 1 \\
1 & 1 & 1 & 1 \\
1 & 1 & 1 & 1 \\
1 & 1 & 1 & 1 \\
\end{bmatrix};  \notag\\
N^{(4,2)} &=N^{(4,1)}=\begin{bmatrix}
4 & 3 & 3 & 1 \\
2 & 2 & 2 & 1 \\
1 & 1 & 1 & 1 \\
1 & 1 & 1 & 1 \\
\end{bmatrix};\qquad 
N^{(4,3)}=\begin{bmatrix}
7 & 3 & 3 & 1 \\
4 & 2 & 2 & 1 \\
2 & 1 & 1 & 1 \\
1 & 1 & 1 & 1 \\
\end{bmatrix}; \notag\\
N^{(4,4)} &=\begin{bmatrix}
10 & 6 & 3 & 1 \\
6 & 4 & 2 & 1 \\
3 & 2 & 1 & 1 \\
1 & 1 & 1 & 1 \\
\end{bmatrix} \quad\text{and}\quad
\overline{N}=N^{(4,5)}=\begin{bmatrix}
11 & 7 & 4 & 1 \\
7 & 5 & 3 & 1 \\
4 & 3 & 2 & 1 \\
1 & 1 & 1 & 1 \\
\end{bmatrix}.  \notag
\end{align}
One can check that $\overline{N}$ is indeed totally nonnegative. Note that it is
only at this last step of the algorithm that a tnn matrix is obtained -- e.g.,
$[1,3,4|1,3,4](N^{(4,4)}) = -4$.
\end{example}

\begin{remark} {\rm A careful analysis of the restoration algorithm 
reveals the following. Suppose that $N$ is a Cauchon matrix with
indeterminates as entries. Then the minors of $\overline{N}$ are 
Laurent polynomials with nonnegative integer coefficients 
in the original indeterminates. 
This suggests a connection with cluster algebras which we intend to
investigate further in a subsequent paper.} 
\end{remark} 

We end this section by
constructing a totally nonnegative
$\hc$-invariant Cauchon matrix associated to each Cauchon diagram. We will also
need analogous $\hc$-invariant Cauchon matrices (although not tnn) over other
fields of characteristic zero, defined as follows.

\begin{defi}
\label{defNC}
Let $K$ be a field of characteristic zero, with transcendence degree at least $mp$
over $\bbQ$. Choose a set $\{
\xi_{i,\alpha} \}$  of $mp$ elements of $K$ that are 
algebraically independent
over $\bbQ$. Moreover, if $K\subseteq \bbR$, choose the $\xi_{i,\alpha}$ to
be positive.

Given any $m \times p$ Cauchon
diagram $C$, denote by $N_C$ the $m \times p$ matrix whose entries
$n_{i,\alpha}$ are defined by $n_{i,\alpha} = \xi_{i,\alpha}$ if $ (i,\alpha)
\notin C$ and $n_{i,\alpha} =0$ if $ (i,\alpha) \in C$.  
\end{defi} 

\begin{theo}
\label{TheTheoTNN}
\begin{enumerate}
\item The matrix $N_C \in \MmpK$ is an $\hc$-invariant Cauchon matrix 
associated to the Cauchon diagram $C$. 
\item If $K=\bbR$ {\rm(}and so all $\xi_{i,\alpha} >0${\rm)}, then
$\overline{N_C}$ is totally nonnegative.
\end{enumerate}
\end{theo}

\begin{proof}
Part 2 holds by Theorem \ref{Nbarnonneg}.

1. Let $(j,\beta) \in \Ecirc$ and  let 
$\delta= [i_1,\dots,i_l|\alpha_1, \dots,\alpha_l](N_C)$ be a minor of $N_C$ such
that
$(i_l,\alpha_l) < (j,\beta)$. Assume that $ \delta ^{(j,\beta)^+}=0$. We need to
prove that $\delta^{(j,\beta)}=0$.  If $ \delta ^{(j,\beta)}= \delta
^{(j,\beta)^+}$,  then there is nothing to do; so  assume that 
$ \delta ^{(j,\beta)} \neq \delta ^{(j,\beta)^+}$. 
In the notation of Proposition \ref{form0}, 
$u=n_{j,\beta} =\xi_{j,\beta} \neq 0$ and 
$i_l <j$ while 
$\alpha_h < \beta < \alpha_{h+1}$ for some $h \in \onel$; moreover   
$$
\delta^{(j,\beta)^+} = \delta^{(j,\beta)}
  + \delta_{\alpha_h \rightarrow \beta}^{(j,\alpha_h)}\, 
 n_{j,\alpha_h} u^{-1}.
  $$
Hence,   
  $$
0= \delta^{(j,\beta)} 
  + \delta_{\alpha_h \rightarrow \beta}^{(j,\alpha_h)} \,
 n_{j,\alpha_h} \, \xi_{j,\beta}^{-1} \,.
  $$
In order  to conclude, recall from   
Proposition \ref{ObservationDeletingDerivation}
that each $n_{\ia}^{(j,\beta)} =n_{\ia} +Q_{\ia}$
where $Q_{\ia}$ is a Laurent polynomial with coefficients in $\bbQ$ 
in the nonzero $n_{k,\gamma}$ such that  $(\ia) < (k,\gamma) < (j,\beta)$, and
that each $n_{\ia}^{(j,\alpha_h)} =n_{\ia} +Q'_{\ia}$ 
where $Q_{\ia}$ is a Laurent polynomial with coefficients in $\bbQ$ 
in the nonzero $n_{k,\gamma}$ such that  $(\ia) < (k,\gamma) < (j,\alpha_h)$. 
Hence, $\delta^{(j,\beta)}$  and $\delta_{\alpha_h \rightarrow
\beta}^{(j,\alpha_h)}\, 
 n_{j,\alpha_h}$ are Laurent polynomials in the 
 $\xi_{k,\gamma}$ such that $(k,\gamma) < (j,\beta)$. 
It follows that  $\delta^{(j,\beta)}=0$, 
 as desired, because 
the $\xi_{i,\alpha}$ are algebraically independent over $\bbQ$.
\end{proof}

The minors which vanish on the tnn matrices $\overline{N_C}$ will be identified, in
terms of Poisson $\hc$-primes of $\pmmpc$, in the following section.


\section{The restoration algorithm and 
Poisson $\hc$-prime ideals of $\pmmpc$.}

In this section, we investigate the standard Poisson structure of the
coordinate ring $\pmmpc$ that comes 
from the commutators of $\mmpc$ (see Section
\ref{sectionPoisson}). Recall, from Section \ref{section:PoissonHprimes}, 
that the
number of Poisson $\hc$-primes in $\pmmpc$ is the same as the number of $m
\times p$ Cauchon diagrams. In this section, we use the restoration algorithm
to construct an explicit bijection between the set of $m \times p$ Cauchon
diagrams and the set of Poisson $\hc$-primes of $\pmmpc$. As a corollary, we
will attach to each Poisson $\hc$-prime an $\hc$-invariant Cauchon 
matrix.  This is an essential step in order to describe the
admissible families of minors.

Let $C$ be a $m \times p$ Cauchon diagram. Denote by $A_C$ the following commutative
polynomial algebra over $\bbC$ in $mp - | C |$ indeterminates: 
$$A_C:=\bbC[t_{\ia} \mid (\ia) \in \left( \onemxp \right)
\setminus C] .$$ 
In the sequel, it will be convenient to set $t_{\ia}:=0$ when
$(\ia) \in C$. While $A_C$ can be identified with a subalgebra of $\pmmpc$, we
label its indeterminates $t_{\ia}$ rather than $Y_{\ia}$ because we require $A_C$
to have a different Poisson structure than $\pmmpc$, as follows.

There is a unique Poisson bracket on $A_C$ determined by the following data:
$$
\{ t_{\ia}, t_{k,\gamma} \}:= \left\{ \begin{array}{ll}
t_{\ia}t_{k,\gamma} &  \mbox{ if } i=k \mbox{ and } \alpha < \gamma \\
t_{\ia}t_{k,\gamma} &  \mbox{ if } i< k \mbox{ and } \alpha = \gamma \\
0 &   \mbox{ if } i < k \mbox{ and } \alpha \ne \gamma\,. \\
\end{array} \right.
$$
Denote by $L_C$ the corresponding Laurent polynomial algebra; that is, 
$$L_C:=\bbC[t_{\ia}^{\pm 1} \mid (\ia) \in \left( \onemxp \right) \setminus C] .$$
The Poisson bracket defined on $A_C$ extends uniquely to a Poisson bracket on 
the algebra $L_C$, so that $L_C$ is also a Poisson algebra.
Denote the field of fractions of $A_C$  by $G_C$. 
The Poisson bracket on $A_C$ extends uniquely to a 
Poisson bracket on $G_C$; so that $G_C$ is also a Poisson algebra.

Observe that the torus $\hc:=(\bbC^{\times})^{m+p}$ acts by Poisson automorphisms on
$A_C$ such that
$$(a_1,\dots,a_m,b_1,\dots,b_p).t_{\ia} = a_i b_\alpha t_{\ia}$$
for all $(a_1,\dots,a_m,b_1,\dots,b_p) \in \hc$ and $(\ia)\in \gc 1,m \dc \times \gc
1,p \dc$. This Poisson action extends naturally 
to Poisson automorphism actions of $\hc$ on $L_C$ and $G_C$.

Set $M_C:=(t_{\ia}) \in \mc_{m,p}(G_C)$; 
this is a Cauchon matrix associated to the Cauchon diagram $C$. 
For all $(j,\beta) \in E$, set 
$$M_C^{(j,\beta)}:=(t_{\ia}^{(j,\beta)}) \in \mc_{m,p}(G_C);$$ 
that is,   $M_C^{(j,\beta)}$ is 
the matrix obtained 
from $M_C$ at step $(j,\beta)$ of the restoration algorithm. 
Let $A_C^{(j,\beta)}$ be the subalgebra of $G_C$ generated 
by the entries of $M_C^{(j,\beta)}$.

\begin{theo}
\label{theoPoissonrestoration}
Let $(j,\beta) \in E$. 
\begin{enumerate}
\item $\mathrm{Frac}(A_{C}^{(j,\beta)})=G_C$.

\item For all $(\ia) \in \onemxp$, we
have $t_{\ia}^{(j,\beta)} =t_{\ia} +Q_{\ia}^{(j,\beta)}$, where
$Q_{\ia}^{(j,\beta)}$ is a Laurent polynomial with coefficients in $\bbQ$ in the
nonzero $t_{k,\gamma}$ such that $(\ia) < (k,\gamma) < (j,\beta)$.

\item Let $B_C^{(j,\beta)}$ be the subalgebra of $G_C$ generated by the
$t_{\ia}^{(j,\beta)}$ with $(\ia) < (j,\beta)$.  If $t_{j,\beta} \neq 0$,
then the powers $t_{j,\beta}^k$, with $k \in \bbN\cup\{0\}$, are linearly
independent over
$B_C^{(j,\beta)}$.

\item If $(i,\alpha),(k,\gamma) \in \onemxp$, then 
$$
\{t_{\ia}^{(j,\beta)} ,t_{k,\gamma}^{(j,\beta)} \}
=\begin{cases}
t_{\ia}^{(j,\beta)}t_{k,\gamma}^{(j,\beta)} 
&  \mbox{ if } i=k \mbox{ and } \alpha < \gamma \\[1ex]
t_{\ia}^{(j,\beta)} t_{k,\gamma}^{(j,\beta)} 
&  \mbox{ if } i< k \mbox{ and } \alpha = \gamma \\[1ex]
0 &   \mbox{ if } i < k \mbox{ and } \alpha > \gamma \\[1ex]
2 t_{i,\gamma}^{(j,\beta )} t_{k,\alpha}^{(j,\beta )} 
&  \mbox{ if } i< k  \mbox{, } \alpha < \gamma 
\mbox{ and } (k,\gamma ) < (j,\beta ) \\[1ex]
0 &  \mbox{ if } i< k  \mbox{, } \alpha < \gamma 
\mbox{ and } (k,\gamma ) \geq (j,\beta ) \,. 
\end{cases}
$$

\item $(a_1, \cdots ,a_m, b_1, \cdots , b_p). t_{\ia}^{(j,\beta)} = a_i b_\alpha
t_{\ia}^{(j,\beta)}$  for all $(a_1, \cdots ,a_m, b_1, \cdots , b_p) \in \hc$ and
$(i,\alpha) \in
\onemxp$.

\end{enumerate}
\end{theo}

\begin{proof}
1. This is an easy induction on $(j,\beta) \in E$ (recall Observations
\ref{deleting}).

2. This is part 4 of Proposition \ref{ObservationDeletingDerivation}. 
 
3. This claim easily follows from the previous part 
and the fact that the
powers $t_{j,\beta}^k$ are linearly independent over the subalgebra of $G_C$
generated by the $t_{\ia}$ with $(\ia) < (j,\beta)$. 
 
4. We relegate the proof of this part to Appendix \ref{appendixPoisson}, due to the large number of cases to be checked.
 
5. This is an easy induction and is left to the reader.
\end{proof}

For each Cauchon diagram $C$, we thus obtain from the restoration algorithm a
Poisson algebra 
$A'_C:=A_C^{(m,p+1)}$ generated by $mp$ elements $y_{\ia}:=t_{\ia}^{(m,p+1)}$ such that, for all $(\ia) < (k,\gamma)$, 
we have: 
$$\{y_{\ia} ,y_{k,\gamma} \}=\left\{ \begin{array}{ll}
y_{\ia}y_{k,\gamma} &  \mbox{ if } i=k \mbox{ and } \alpha < \gamma \\
y_{\ia} y_{k,\gamma} &  \mbox{ if } i< k \mbox{ and } \alpha = \gamma \\
0 &   \mbox{ if } i < k \mbox{ and } \alpha > \gamma \\
2 y_{i,\gamma} y_{k,\alpha} &  \mbox{ if } i< k  
\mbox{ and } \alpha < \gamma  \,. \\ 
\end{array} \right.$$
Hence, there exists a surjective Poisson homomorphism 
$\varphi_C : \pmmpc \rightarrow A'_C$ that sends 
$Y_{\ia}$ to $y_{\ia}$ for all $(\ia)$. 
Moreover, we deduce from Theorem \ref{theoPoissonrestoration} 
that this homomorphism is $\hc$-equivariant, so that 
the kernel $J'_C$ of $\varphi_C$ is a Poisson $\hc$-prime of $\pmmpc$. 

Recall the notation $\mathcal{C}_{m,p}$ for the set of all $m\times p$ Cauchon
diagrams.

\begin{lem}
The map $C \mapsto J'_C$ is an embedding of $\mathcal{C}_{m,p}$ into the set of 
Poisson $\hc$-primes of $\pmmpc$.
\end{lem}

\begin{proof}
Let $C$ and $C'$ be two Cauchon diagrams, and assume that 
$J'_{C}=J'_{C'}$. In order to avoid any confusion, we will denote 
the natural generators of $A'_{C'}$
by $y'_{\ia}$ rather than  $y_{\ia}$. 
As $J'_{C}=J'_{C'}$, there exists a Poisson isomorphism 
$\psi : A'_{C} \rightarrow A'_{C'}$ that sends $y_{\ia}$ to $y'_{\ia}$ 
for all $(\ia)$. Of course, this isomorphism extends to an isomorphism 
$\psi : \mathrm{Frac}(A'_{C}) \rightarrow \mathrm{Frac}(A'_{C'})$, 
and a descreasing induction on $(j,\beta)$ shows  that 
$\psi (t_{\ia}^{(j,\beta)})={t'}_{\ia}^{(j,\beta)}$ for all $(\ia)\in \onemxp$ 
and $(j,\beta)\in E$. In particular, we get:
$$
(\ia) \in C \ \quad\Longleftrightarrow\quad \ t_{\ia}^{(1,2)} 
=0 \ \quad\Longleftrightarrow\quad \ {t'}_{\ia}^{(1,2)}=0 \ 
\quad\Longleftrightarrow\quad \ (\ia) \in C'.
$$
Hence, $C=C'$, as desired.
\end{proof}

\begin{theo}
\label{theo:PoissonHspectrum}
$\hc$-$\mathrm{PSpec}(\pmmpc)
=\{ J'_C \mid C \in \mathcal{C}_{m,p} \}$.
\end{theo}

\begin{proof}
We have just proved that 
$$
\hc \mbox{-}\mathrm{PSpec}(\pmmpc) 
\supseteq \{ J'_C \mid C \in \mathcal{C}_{m,p} \}.
$$ 
In order to conclude, 
recall from the discussion in 
Section \ref{section:PoissonHprimes}   
that the number of Poisson $\hc$-primes in $\pmmpc$ 
is equal to $|\mathcal{C}_{m,p}|$. In view of the previous lemma, the
displayed inclusion must be an equality.
\end{proof}

\begin{theo}
\label{theo:PoissonMatrix}
Let $C $ be an $m \times p$ Cauchon diagram. 
\begin{enumerate}

\item The matrix 
$M_C =(t_{\ia}) \in \mc_{m,p}(G_C)$ is an $\hc$-invariant 
Cauchon matrix associated to $C$.

\item A minor $[I|\Lambda]$ belongs to $J'_C$ if and only if the 
corresponding minor of $\overline{M_C}:=M_C^{(m,p+1)}$ is zero.

\end{enumerate}
\end{theo}
\begin{proof} The first part follows from 
Theorem \ref{TheTheoTNN}. 
The second part is a consequence of the construction of the 
Poisson $\hc$-prime $J'_C$ as the kernel of the 
surjective $\hc$-equivariant Poisson  homomorphism 
$\varphi_C : \pmmpc \rightarrow A'_C$ that sends $Y_{\ia}$ 
to $y_{\ia}$ for all $(\ia)$.
\end{proof}

\begin{cor}
\label{minorsvanishNCbar}
Let $C$ be an $m\times p$ Cauchon diagram, and construct the matrix $N_C \in
\mc_{m,p}(\bbR)$ as in Definition {\rm\ref{defNC}}. Then $\overline{N_C} =
N_C^{(m,p+1)}$ is a tnn matrix, and the minors which vanish on $\overline{N_C}$ are
precisely those which belong to the ideal $J'_C$ of $\pmmpc$.
\end{cor}

\begin{proof} Theorems \ref{TheTheoTNN} and \ref{theo:PoissonMatrix}, and Corollary
\ref{cor:independence}.
\end{proof}


\section{Explicit description of the admissible families of minors.}

Recall that a family of minors is admissible if it defines a 
nonempty totally nonnegative cell. We are now ready to prove our main result that gives an explicit description of the admissible families of minors. 

Recall the families $\mc(w)$ with $w \in \mathcal{S}$ defined in
Definition \ref{definition-M(w)}. Our work so far immediately shows that these define nonempty tnn cells, as follows.

\begin{lem}  \label{Mwadmissible}
For each $w\in \mathcal{S}$, the tnn cell $S_{\mc(w)}$ is nonempty.
\end{lem}

\begin{proof} It follows from Theorem
\ref{theo:PoissonMinors} that there exists a (unique) Poisson $\hc$-prime
ideal $J_w$ in $\pmmpc$ such that the minors that belong to $J_w$ are exactly those
from $\mc(w)$. Moreover, it follows from Theorem \ref{theo:PoissonHspectrum}
that there exists a Cauchon diagram $C$ such that $J_w =J'_C$. By Corollary
\ref{minorsvanishNCbar}, the minors that vanish on the totally nonnegative matrix
$\overline{N_C}$ are exactly those in
$\mc(w)$. Therefore, $\overline{N_C} \in S_{\mc(w)}$.
\end{proof}

\begin{theo}
\label{TheoDescription}
The admissible families of minors for the space $\mmptnn$ of $m\times p$
totally nonnegative matrices are exactly the families $\mc(w)$ for $w \in
\mathcal{S} = S^{[-p,m]}_{m+p}$.\end{theo} 

\begin{proof}
We already know that the number of nonempty totally nonnegative cells is less
than or equal to $|\mathcal{S}|$ by Corollary \ref{bound}.
Note that Postnikov, in \cite{postnikov}, has shown that this
is in fact an equality, but, in Appendix B, we prove this inequality via different
methods. We thus recover the equality by our methods, as a consequence of the
present theorem.

By Corollary \ref{Mwcontain}, the sets $\mc(w)$, for $w\in \mathcal{S}$,
are all distinct. We conclude by invoking Lemma \ref{Mwadmissible}. 
\end{proof}

Notice that the families of minors $\mc(w)$ are not that easy to
compute. Let us mention, however, that the results of the present paper provide
also an algorithmic way to produce these families. Indeed, it follows from the
proof of Corollary \ref{minorsvanishNCbar} that the admissible families of minors
are exactly the families of vanishing minors of $\overline{M_C}$ with $C \in
\mathcal{C}_{m,p}$. Hence, we have the following algorithm that, starting only
from a Cauchon diagram, constructs an admissible family of minors. \\

  \begin{alg}\label{admissiblealgorithm} $  $\\
\underline{\bf Input:}
\begin{enumerate}
\item[]
Fix $C \in \mathcal{C}_{m,p}$,  and denote by
$A_C$ the following commutative polynomial algebra over $\bbC$ generated by $mp -
| C|$ indeterminates: 
$$A_C:=\bbC[t_{\ia} \mid (\ia) \in \left( \onemxp \right) \setminus C] .$$ 
Let $G_C$
denote the field of fractions of
$A_C$. 
\end{enumerate}

\noindent\underline{\bf Step 1: Restoration of $\overline{M_C}$.} \quad
As $(j,\beta)$ runs over the set $E$, define matrices $M_C^{(j,\beta)}
=(t_{\ia}^{(j,\beta)}) \in
\mc_{m,p}(G_C)$ as follows:

\begin{enumerate}
\item \underline{If $(j,\beta)=(1,2)$}, then the entries of the matrix 
$M_C=M_C^{(1,2)}$ are defined by 
$$t_{\ia} := t_{\ia}^{(1,2)}:= \left\{ \begin{array}{ll}
t_{\ia} & \mbox{ if } (\ia) \notin C\\
0 & \mbox{ otherwise.}
\end{array} \right.$$
\item \underline{Assume that $(j,\beta) \in \Ecirc$} and that the matrix
$M_C^{(j,\beta)}$ is  already known. The entries $t_{\ia}^{(j,\beta)^+}$ of the
matrix $M_C^{(j,\beta)^+}$ are defined as follows:
\begin{enumerate}
\item If $t_{j,\beta} =0$, then
$t_{\ia}^{(j,\beta)^+}= t_{\ia}^{(j,\beta)}$ for all $(\ia) \in \gc 1,m \dc
\times \gc 1,p \dc$. 
\item If $t_{j,\beta} \neq 0$ and $(\ia) \in
\gc 1,m \dc \times \gc 1,p \dc$, then 
$$t_{\ia}^{(j,\beta)^+}= \begin{cases} t_{\ia}^{(j,\beta)}+t_{i,\beta}^{(j,\beta)}
t_{j,\beta}^{-1} t_{j,\alpha}^{(j,\beta)} & \mbox{if }i
<j \mbox{ and } \alpha < \beta \\ 
t_{\ia}^{(j,\beta)} & \mbox{otherwise.} \end{cases}$$ 
\end{enumerate}
\end{enumerate}

\noindent\underline{\bf Step 2: Calculate all minors of $\overline{M_C}= M_C^{(m,p+1)}$.}$ $
\\

\noindent\underline{\bf Result:}
\begin{enumerate}
\item[] Denote by $\mc(C)$ the following set of minors:
$$\mc(C):= \{ [I| \Lambda] \ | \ [I| \Lambda](\overline{M_C})=0 \}.$$
Then, $\mc(C)$ is an admissible family of minors and, 
if $C$ and $C'$ are two distinct Cauchon diagrams, 
then $\mc(C) \neq \mc(C')$.
\end{enumerate}
\end{alg}

\begin{example} Assume that $m=p=3$.
\begin{figure}[h]
\begin{center}
$$\xymatrixrowsep{0.01pc}\xymatrixcolsep{0.01pc}
\xymatrix{
\plc\edge[0,6]\edge[6,0] &&\plc\edge[6,0] &&\plc\edge[6,0]
&&\plc\edge[6,0]   \\
 &\mdblk && && \mdblk &&  \plc \\
\plc\edge[0,6] &&&&&&\plc \\
 &\mdblk &&\mdblk &&   \\
\plc\edge[0,6] &&&&&&\plc \\
 & && &&  \\
\plc\edge[0,6] &&\plc &&\plc &&\plc  
}$$
\caption{A $3 \times 3$ Cauchon diagram}
\label{fig:CauchonDiagram4}
\end{center}
\end{figure}
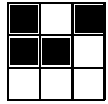

Applying Algorithm \ref{admissiblealgorithm} to the Cauchon
diagram  of Figure \ref{fig:CauchonDiagram4} shows that the family of minors 
$$
\bigl\{ [1 |3]\,, [1,2|1,2]\,,  [1,3|1,2]\,, [2,3|1,2]\,, [2,3| 1,3]\,, 
[2,3| 2,3]\,, [1,2,3| 1,2,3] \bigr\}
$$ 
is admissible. As is easily checked, the above admissible set equals $\mc(w)$ where $w= \begin{bmatrix} 1&2&3&4&5&6\\ 1&4&3&2&6&5 \end{bmatrix}$.
\end{example}


\section{Tnn cells and $\hc$-orbits of symplectic leaves.}

One can construct another partition of the space $\mmptnn$ using the $\hc$-orbits of symplectic leaves in $\Mmpc$ as follows.
For each restricted permutation $w \in \mathcal{S}$, we set: 
\[
U_w:=
\mathcal{P}_{\woN w} \cap \mmptnn.
\]
As the $\hc$-orbits of symplectic leaves $\mathcal{P}_{\woN w}$ form a
partition of $\mc_{m,p}(\bbC)$, the sets $U_w$ with $w \in
\mathcal{S}$ form a partition of $\mmptnn$. 

On the other hand, we know that the nonempty totally nonnegative cells also form 
a partition of this space $\mmptnn$ and are parametrised
by the same set $\mathcal{S}$ of permutations.  For each $w \in \mathcal{S}$, we
have a nonempty tnn cell denoted by
$\mathcal{S}_{\mc(w)}$ in the notation of equation \eqref{deftnncell}, so
that one can write:
\begin{equation}
\label{2partition}
\bigsqcup_{w \in \mathcal{S}}
U_w=\mmptnn = \bigsqcup_{w \in \mathcal{S}}
\mathcal{S}_{\mc(w)}.
\end{equation}
 Thus we have two partitions of the same space $\mmptnn$ indexed by the same set $\mathcal{S}$. 
Our final theorem asserts that these two partitions are the same. We first give a proof based on our present methods, and then we sketch an alternate proof that follows ideas and results in the literature.
 
\begin{theo}  \label{tnnHsymp}
For each $w \in \mathcal{S}$, we have $\mathcal{S}_{\mc(w)}=U_{w}$. 
Thus, the nonempty tnn cells in the space $\mmptnn$ of tnn
matrices are the intersections of the $\hc$-orbits of symplectic leaves in $\Mmpc$
with $\mmptnn$.
\end{theo}

\begin{proof}
The theorem follows easily from (\ref{2partition}) once we show that
$\mathcal{S}_{\mc(w)} \subseteq U_{w}$ for all $w\in \mathcal{S}$, given
that each $\mathcal{S}_{\mc(w)}$ is nonempty (Lemma
\ref{Mwadmissible}).

Let $w\in \mathcal{S}$. Because of Theorem \ref{theorem-4.2reformulation}, we at
least have
$$\mathcal{S}_{\mc(w)} \subseteq  \overline{\mathcal{P}_{\woN w}} \cap
\mmptnn.$$
If $\mathcal{S}_{\mc(w)} \nsubseteq U_{w}$, it thus follows from
\cite[Theorem 3.13]{bgy} that there exists a matrix
$$x\in \mathcal{S}_{\mc(w)} \cap \mathcal{P}_{\woN z}$$
for some $z\in \mathcal{S}$ with $z>w$. By Corollary \ref{Mwcontain},
$\mc(w)$ is properly contained in $\mc(z)$, so there is a minor
$[I|\Lambda] \in \mc(z) \setminus \mc(w)$. Since $x\in
\mathcal{P}_{\woN z}$, Theorem \ref{theorem-4.2reformulation} shows that
$[I|\Lambda](x) =0$. However, this contradicts the assumption that $x\in
\mathcal{S}_{\mc(w)}$, since $[I|\Lambda] \notin \mc(w)$. Therefore
$\mathcal{S}_{\mc(w)} \subseteq U_{w}$, as required.
\end{proof}

\begin{subsec}
 {\bf Alternate proof of Theorem \ref{tnnHsymp}.}
Keep $m$ and $p$ fixed as usual, and set $n=m+p$.

We begin in the Grassmannian setting. Let $\Grtnn$ denote the $m\times n$ case of Postnikov's \emph{totally nonnegative Grassmannian} \cite[Definition 3.1]{postnikov}, a subset of the real Grassmannian $\GrmnR$. It consists of those points represented by matrices $A\in \mc_{m,n}(\bbR)$ such that $\Delta_I(A) \ge 0$ for all $m$-element subsets $I\subseteq \onen$, where $\Delta_I := [1,\dots,m|I]$ is the maximal minor with column index set $I$. The \emph{tnn Grassmann cells} \cite[Definition 3.2]{postnikov} are subsets $\StnnM \subseteq \Grtnn$, for collections $\mc$ of $m$-element subsets of $\onen$, defined as follows: $\StnnM$ consists of those points represented by matrices $A$ such that $\Delta_I(A) >0$ for all $I\in \mc$ and $\Delta_I(A) =0$ for all $I\notin \mc$. (Postnikov only defined $\StnnM$ for matroids $\mc$, but the definition works equally well in general. It is easily seen that if $\StnnM$ is nonempty, then $\mc$ must be a matroid.)
 
Lusztig has introduced nonnegative parts of real generalized flag varieties, one case of which leads to the Grassmannian. Let $G= GL_n(\bbR)$, with its usual Borel subgroups $B^+$ and $B^-$, consisting of invertible upper, respectively lower, triangular matrices. Let $W$ be the Weyl group of $G$, which we identify with both the symmetric group $S_n$ and the subgroup of permutation matrices in $G$. Then set
$$P^J= \begin{bmatrix} GL_m(\bbR) &\mc_{m,p}(\bbR)\\ 0 &GL_p(\bbR) \end{bmatrix},$$
a standard parabolic subgroup of $G$ containing $B^+$. Let $\pc^J= G/P^J$ denote the corresponding partial flag variety, and $\pi^J: G/B^+ \rightarrow \pc^J$ the natural projection. Lusztig first defined the nonnegative part of the full flag variety $G/B^+$ (\cite[\S8.8]{L1}, \cite[\S2.6]{L2}), and then defined the nonnegative part of $\pc^J$ as the projection: $\pc^J_{\ge0} := \pi^J\bigl( (G/B^+)_{\ge0} \bigr)$. 

Under the standard identification of $\GrmnR$ with $\pc^J$, we have
$$\Grtnn= \pc^J_{\ge0}.$$
This is tacitly assumed in the discussion of \cite[Theorem 3.8]{postnikov}; details have been worked out by Rietsch \cite{Rpers}.

Now partition $G/B^+$ into intersections of dual Schubert cells, namely the sets $\rc_{v,w} := B^-.vB^+\cap B^+.wB^+$ for $v,w\in W$. Then $\rc_{v,w} \ne \varnothing$ precisely when $v\le w$ (this is implicit in \cite{KL1, KL2}, as noted in \cite[\S1.3]{L2}; in the case of an algebraically closed base field, it is proved explicitly in \cite[Corollary 1.2]{deo}). This partition of $G/B^+$ projects onto a partition of $\pc^J$ into sets which Rietsch labelled
$$\pc^J_{x,u,w} := \pi^J(\rc_{x,wu})= \pi^J(\rc_{xu^{-1},w})$$
for $(x,u,w)\in W^J_{\max}\times W_J\times W^J$, with $\pc^J_{x,u,w} \ne\varnothing$ if and only if $x\le wu$ \cite[Section 5]{Rie}. Here $W^J_{\max}$ denotes the set of maximal length representatives for cosets in $W/W_J$, where $W_J$ is the Weyl group of the standard Levi factor of $P^J$. Now contract these sets to the nonnegative part of $\pc^J$, to get sets $\pc^J_{x,u,w;>0} := \pc^J_{x,u,w} \cap \pc^J_{\ge0}$ for $(x,u,w)\in W^J_{\max}\times W_J\times W^J$. Rietsch proved that the nonempty strata $\pc^J_{x,u,w}$ of $\pc^J$ all have nonempty intersection with the nonnegative part of $\pc^J$, namely $\pc^J_{x,u,w;>0} \ne\varnothing$ if and only if $x\le wu$ \cite[Section 6, p. 783]{Rie}. 

The nonempty sets $\pc^J_{x,u,w;>0}$ partition $\pc^J_{\ge0}$, and this partition coincides with Postnikov's partition of $\Grtnn$ into nonempty tnn Grassmann cells \cite[Theorem 3.8]{postnikov}. 

In order to discuss symplectic leaves, we have to move into the complex setting. Let us denote the complex versions of the above ingredients with hats.
 Thus, $\Ghat= GL_n(\bbC)$, with its usual Borel subgroups $\Bhat^{\pm}$, while
$$\Phat^J= \begin{bmatrix} GL_m(\bbC) &\mc_{m,p}(\bbC)\\ 0 &GL_p(\bbC) \end{bmatrix},$$
with corresponding partial flag variety $\pchat^J= \Ghat/\Phat^J$ partitioned into subsets $\pchat^J_{x,u,w}$. We identify the maximal torus $\Bhat^+ \cap \Bhat^-$ with the group $\hc= (\bbC^\times)^n$. Identify $\pc^J$ with its natural image in $\pchat^J$; then $\pc^J_{x,u,w}= \pchat^J_{x,u,w} \cap \pc^J$ for all $(x,u,w)$.

There is a standard Poisson structure on $\Ghat$, arising from the standard $r$-matrix on $\mathfrak{gl}_n(\bbC)$, making $\Ghat$ into a Poisson algebraic group (cf.~\cite[\S1.4]{bgy}). The Grassmannian $\GrmnC$ then becomes a Poisson variety and a Poisson homogeneous space for $\Ghat$ \cite[Proposition 3.2]{bgy}. Further, the torus $\hc$ acts on $\GrmnC$ by Poisson automorphisms \cite[\S0.2]{GY}, and so it permutes the symplectic leaves. It turns out that the $\hc$-orbits of symplectic leaves in $\GrmnC$ coincide with the nonempty sets $\pchat^J_{x,u,w}$ of Rietsch's partition \cite[\S0.4]{GY}. Therefore,
\begin{enumerate}
\item[(I)] The nonempty tnn Grassmann cells $\StnnM$ are precisely the intersections of $\Grtnn$ with the $\hc$-orbits of symplectic leaves in $\GrmnC$.
\end{enumerate}

Let $\Omega$ denote the top Schubert cell in $\GrmnC$, namely the set of points corresponding to matrices $A$ for which $\Delta_{\onem}(A) \ne 0$; then $\Omega= \Bhat^-.\Phat^J$ under the identification $\GrmnC= \pchat^J$. Postnikov has given an isomorphism between $\mmptnn$ and $\Omega \cap \Grtnn$ \cite[Proposition 3.10]{postnikov}, which we modify slightly. For matrices $X= (x_{ij}) \in \Mmpc$, define $\Xtil= \bigl( (-1)^{m-i}x_{m+1-i,j} \bigr)$. There is an isomorphism $\xi: \Mmpc \rightarrow \Omega$ sending $X$ to the point representated by the block matrix $\begin{bmatrix} I_m &\Xtil \end{bmatrix}$. As is easily seen, the maximal minors of $\begin{bmatrix} I_m &\Xtil \end{bmatrix}$ coincide with the minors of $X$ (with no changes of sign). Hence,
\begin{enumerate}
\item[(II)] $\xi$ restricts to an isomorphism of $\mmptnn$ onto $\Omega\cap \Grtnn$, which carries the nonempty tnn cells $\mathcal{S}_{\mathcal{F}} \subseteq \mmptnn$ to the nonempty tnn Grassmann cells $\StnnM \subseteq \Omega\cap \Grtnn$.
\end{enumerate}

We next need the Poisson isomorphism $\Psi: \mc_{p,m}(\bbC) \rightarrow \Bhat^-.\Phat^J= \Omega$ given in \cite[Proposition 3.4]{bgy}. The map $\xi$ can be expressed in terms of $\Psi$ by the formula
$$\xi(X)= \Psi(X\transp D),$$
where $D= \diag(1,-1,1,\dots,(-1)^{m-2},(-1)^{m-1})$, from which it follows that $\xi$ is a Poisson isomorphism. Neither transposition nor $\Psi$ is equivariant with respect to the relevant actions of $\hc$, but both permute $\hc$-orbits of symplectic leaves, because there are automorphisms $\alpha$, $\beta$ of $\hc$ such that $\xi(h.X)= \alpha(h).\xi(X)$ and $\Psi(h.Y)= \beta(h).\Psi(Y)$ for $X\in \Mmpc$, $Y\in \mc_{p,m}(\bbC)$, and $h\in\hc$ (see the proof of \cite[Theorem 3.9]{bgy} for the latter). Thus,
\begin{enumerate}
\item[(III)] $\xi$ sends $\hc$-orbits of symplectic leaves in $\Mmpc$ to $\hc$-orbits of symplectic leaves in $\Omega$.
\end{enumerate}

Combining (I), (II), and (III) yields a second proof of Theorem \ref{tnnHsymp}.
\end{subsec}


\appendix

\section*{Appendix.}

\section{Proof of Theorem \ref{theoPoissonrestoration}(4).}
\label{appendixPoisson}

Set $r= (j,\beta)$. It is clear that statement 5.1(4) holds when $r=(1,2)$. Now let $r\in \Ecirc$, and assume that the statement holds at step $r$. We distinguish between two cases. 

$\bullet$ Assume first that $t_{j,\beta}=0$. 
In this case, $t_{l,\delta}^{(r^+)}=t_{l,\delta}^{(r)}$ 
for all $(l,\delta)$, and so we may rewrite the result at step $r$ as
follows:
$$
\{t_{\ia}^{(r^+)} ,t_{k,\gamma}^{(r^+)} \}
=\begin{cases} t_{\ia}^{(r^+)}t_{k,\gamma}^{(r^+)}  
&  \mbox{ if } i=k \mbox{ and } \alpha < \gamma \\[1ex]
t_{\ia}^{(r^+)} t_{k,\gamma}^{(r^+)} 
&  \mbox{ if } i< k \mbox{ and } \alpha = \gamma \\[1ex]
0 &   \mbox{ if } i < k \mbox{ and } \alpha > \gamma \\[1ex]
2 t_{i,\gamma}^{(r^+)} t_{k,\alpha}^{(r^+)} 
&  \mbox{ if } i< k  \mbox{, } \alpha < \gamma 
\mbox{ and } (k,\gamma ) < r \\[1ex]
0 &  \mbox{ if } i< k  \mbox{, } \alpha < \gamma 
\mbox{ and } (k,\gamma ) \geq r \,. \\ 
\end{cases}
$$
To conclude in this case, it just remains to show that 
$t_{i,\gamma}^{(r^+)} t_{k,\alpha}^{(r^+)}=0$ 
when $(k,\gamma)=r$ with  
$i < k$ and $\alpha < \gamma$;  
that is, we need to prove that 
$t_{i,\beta}^{(r^+)} t_{j,\alpha}^{(r^+)}=0$ 
when $i < j$ and $\alpha < \beta$. 
By Proposition \ref{ObservationDeletingDerivation}, 
we have $ t_{j,\alpha}^{(r^+)}=t_{j,\alpha}$, 
and so we need to prove that 
$t_{i,\beta}^{(r^+)} t_{j,\alpha}=0$ 
when $i < j$ and $\alpha < \beta$. 
It is at this point that we  use our assumption that 
$C$ is a Cauchon diagram. Indeed, as $t_{j,\beta}=0$, 
then $r \in C$. As $C$ is a Cauchon diagram, 
this forces either $(j,\alpha) \in C$ or $(k,\beta) \in C$ 
for all $k \leq j$. Hence, by construction, we get that 
either $t_{j, \alpha}=0$ or $t_{k,\beta}=0$ for all $k \leq j$. 
In the first case, it is clear that 
$t_{i,\beta}^{(r^+)} t_{j,\alpha}=0$. If $t_{k,\beta}=0$ for all $k \leq j$, 
then it follows from Proposition \ref{ObservationDeletingDerivation} 
that $t_{k,\beta}^{(r^+)}=0$ 
for all $k \leq j$. 
Thus, $t_{i,\beta}^{(r^+)} t_{j,\alpha}=0$ 
in this case too.

Thus, 5.1(4) holds at step $r^+$ provided $t_\jb=0$.

$\bullet$ Now assume that $t_{j,\beta} \neq 0$. We first compute $\{\trp_\ia, \tr_\ld\}$ in a number of cases, where $(\ia),(\ld) \in \onemxp$ with $i<j$ and $\alpha<\beta$. Note that
\begin{multline} 
\{\trp_\ia, \tr_\ld\}= \\
\{\tr_\ia, \tr_\ld\}+ \{\tr_\ib, \tr_\ld\} \tinv_\jb \tr_\ja -\tr_\ib t^{-2}_\jb \{t_\jb, \tr_\ld\} \tr_\ja+ \tr_\ib \tinv_\jb \{\tr_\ja, \tr_\ld\}.
\label{E1}
\end{multline}
When expanding a bracket using \eqref{E1}, we will write out four terms in the given order, using $0$ as a placeholder where needed.

For terms in the same row, we claim that
\begin{equation}
\{\trp_\ia, \tr_\ld\}= \begin{cases} (\tr_\ia- \tr_\ib \tinv_\jb \tr_\ja) \tr_\ld &(i=l<j;\ \alpha<\delta<\beta) \\[1ex]
(\tr_\ia+ \tr_\ib \tinv_\jb \tr_\ja) \tr_\ld= \trp_\ia \tr_\ld &(i=l<j;\ \alpha<\beta\le \delta). \end{cases}
\label{E2}
\end{equation}
In the first case, we obtain the result by observing that
$$\{\trp_\ia, \tr_\ld\}= \tr_\ia \tr_\ld+ (-\tr_\ib \tr_\ld) \tinv_\jb \tr_\ja -0+0.$$
The second case splits into two subcases:
$$\{\trp_\ia, \tr_\ld\}= \begin{cases} \tr_\ia \tr_\ld+ 0- \tr_\ib t^{-2}_\jb (-t_\jb \tr_\ld) \tr_\ja+ 0 &(i=l<j;\ \alpha<\beta=\delta) \\[1ex]
\tr_\ia \tr_\ld+ (\tr_\ib \tr_\ld) \tinv_\jb \tr_\ja -0+0 &(i=l<j;\ \alpha<\beta<\delta).
\end{cases}$$
This establishes \eqref{E2}.

For terms in the same column, we claim that
\begin{equation}
\{\trp_\ia, \tr_\ld\}= \begin{cases} (\tr_\ia- \tr_\ib \tinv_\jb \tr_\ja) \tr_\ld &(i<l<j;\ \alpha=\delta<\beta) \\[1ex]
(\tr_\ia+ \tr_\ib \tinv_\jb \tr_\ja) \tr_\ld= \trp_\ia \tr_\ld &(i<j\le l;\ \alpha=\delta<\beta). \end{cases}
\label{E3}
\end{equation}
The first case follows from
$$\{\trp_\ia, \tr_\ld\}= \tr_\ia \tr_\ld +0-0+ \tr_\ib \tinv_\jb (-\tr_\ja \tr_\ld),$$
while the second follows from
$$\{\trp_\ia, \tr_\ld\}= \begin{cases} \tr_\ia \tr_\ld +0- \tr_\ib t^{-2}_\jb (-t_\jb \tr_\ld) \tr_\ja +0 &(i<j=l;\ \alpha=\delta<\beta) \\[1ex]
\tr_\ia \tr_\ld +0-0+ \tr_\ib \tinv_\jb (
\tr_\ja \tr_\ld) &(i<j<l;\ \alpha=\delta<\beta). \end{cases}$$
This establishes \eqref{E3}.

For terms in NE/SW relation, we claim that
\begin{equation}
\{\trp_\ia, \tr_\ld\}= \begin{cases} -2\tr_\ib \tinv_\jb \tr_\la \tr_\jd &(i<l<j;\ \delta<\alpha<\beta) \\[1ex]
0 &(i<j\le l;\ \delta<\alpha<\beta). \end{cases}
\label{E4}
\end{equation}
In the first case, we have
$$\{\trp_\ia, \tr_\ld\}= 0+0-0+ \tr_\ib \tinv_\jb (-2\tr_\la \tr_\jd),$$
while in the second, we have
\begin{multline}
\{\trp_\ia, \tr_\ld\}= \\
\begin{cases} 0+0- \tr_\ib t^{-2}_\jb (-t_\jb \tr_\ld) \tr_\ja+ \tr_\ib \tinv_\jb (-\tr_\ja \tr_\ld) &(i<j=l;\ \delta<\alpha<\beta) \\[1ex]
0+0-0+0 &(i<j<l;\ \delta<\alpha<\beta). \end{cases}
\end{multline}
This establishes \eqref{E4}.

For terms in NW/SE relation, we claim that
\begin{equation}
\{\trp_\ia, \tr_\ld\}= \begin{cases} 2\tr_\idl \tr_\la &(i<l<j;\ \alpha<\delta<\beta) \\[1ex]
2\tr_\idl \tr_\la+ 2\tr_\ib \tinv_\jb \tr_\ja \tr_\ld &(i<l<j;\ \alpha<\delta=\beta) \\[1ex]
2\tr_\idl (\tr_\la+ \tr_\lb \tinv_\jb \tr_\ja) &(i<l<j;\ \alpha<\beta<\delta) \\[1ex]
2\tr_\idl \tr_\la+ 2\tr_\ib \tinv_\jb \tr_\ja \tr_\ld &(i<l=j;\ \alpha<\delta<\beta) \\[1ex]
2\tr_\ib \tr_\ja= 2\tr_\idl \tr_\la &(i<l=j;\ \alpha<\delta=\beta) \\[1ex]
0 &(i<l=j;\ \alpha<\beta<\delta) \\[1ex]
0 &(i<j<l;\ \alpha<\beta;\ \alpha<\delta). \end{cases}
\label{E5}
\end{equation}
The computation of $\{\trp_\ia, \tr_\ld\}$ in the first six cases yields the following:
$$\begin{cases} 2\tr_\idl \tr_\la +0-0+0 &(i<l<j;\ \alpha<\delta<\beta) \\[1ex]
2\tr_\idl \tr_\la + (\tr_\ib \tr_\ld) \tinv_\jb \tr_\ja- \tr_\ib t^{-2}_\jb (-t_\jb \tr_\ld) \tr_\ja +0 &(i<l<j;\ \alpha<\delta=\beta) \\[1ex]
2\tr_\idl \tr_\la+ (2\tr_\idl \tr_\lb) \tinv_\jb \tr_\ja -0+0 &(i<l<j;\ \alpha<\beta<\delta) \\[1ex]
2\tr_\idl \tr_\la+0- \tr_\ib t^{-2}_\jb (-t_\jb \tr_\ld) \tr_\ja+ \tr_\ib \tinv_\jb (\tr_\ja \tr_\ld) &(i<l=j;\ \alpha<\delta<\beta) \\[1ex]
0+(\tr_\ib \tr_\ld) \tinv_\jb \tr_\ja-0+ \tr_\ib \tinv_\jb (\tr_\ja \tr_\ld) &(i<l=j;\ \alpha<\delta=\beta) \\[1ex]
0+0-\tr_\ib t^{-2}_\jb (t_\jb \tr_\ld) \tr_\ja+  \tr_\ib \tinv_\jb (\tr_\ja \tr_\ld) &(i<l=j;\ \alpha<\beta<\delta). \end{cases}$$
As for the final case of \eqref{E5}, when $i<j<l$ and $\alpha<\beta,\delta$ with $\beta\ne \delta$, we have
$$\{\trp_\ia, \tr_\ld\}= 0+0-0+0,$$
while when $i<j<l$ and $\alpha<\beta=\delta$, we have
$$\{\trp_\ia, \tr_\ld\}= 0+ (\tr_\ib \tr_\ld) \tinv_\jb \tr_\ja- \tr_\ib t^{-2}_\jb (t_\jb \tr_\ld) \tr_\ja +0.$$
This completes the proof of \eqref{E5}.

We are now ready to tackle step $r^+$ of 5.1(4) when $t_\jb\ne 0$, that is, to prove that
\begin{equation}
\{t_{\ia}^{(r^+)} ,t_{k,\gamma}^{(r^+)} \}
=\begin{cases} t_{\ia}^{(r^+)}t_{k,\gamma}^{(r^+)}  
&  \mbox{ if } i=k \mbox{ and } \alpha < \gamma \\[1ex]
t_{\ia}^{(r^+)} t_{k,\gamma}^{(r^+)} 
&  \mbox{ if } i< k \mbox{ and } \alpha = \gamma \\[1ex]
0 &   \mbox{ if } i < k \mbox{ and } \alpha > \gamma \\[1ex]
2 t_{i,\gamma}^{(r^+)} t_{k,\alpha}^{(r^+)} 
&  \mbox{ if } i< k  \mbox{, } \alpha < \gamma 
\mbox{ and } (k,\gamma ) < r^+ \\[1ex]
0 &  \mbox{ if } i< k  \mbox{, } \alpha < \gamma 
\mbox{ and } (k,\gamma ) \geq r^+ \,. \\ 
\end{cases}
\label{E6}
\end{equation}

Assume first that either $i\ge j$ or $\alpha\ge \beta$, whence $\trp_\ia= \tr_\ia$. If also $k\ge j$ or $\gamma\ge \beta$, we have $\trp_\kg= \tr_\kg$. Thus, under the current hypotheses,
\begin{multline}
\{\trp_\ia,\trp_\kg\}= \{\tr_\ia,\tr_\kg\}= \\
\begin{cases} \tr_\ia \tr_\kg= \trp_\ia \trp_\kg &(i=k;\ \alpha<\gamma) \\[1ex]
\tr_\ia \tr_\kg= \trp_\ia \trp_\kg &(i<k;\ \alpha=\gamma) \\[1ex]
0 &(i<k;\ \alpha>\gamma) \\[1ex]
2\tr_\ig \tr_\ka= 2\trp_\ig \trp_\ka &(i<k;\ \alpha<\gamma;\ (\kg)<r) \\[1ex]
0 &(i<k;\ \alpha<\gamma;\ (\kg)\ge r).  \end{cases}
\label{E7}
\end{multline}
Since $i\ge j$ or $\alpha\ge \beta$, we cannot have $(\kg)=r$ when $i<k$ and $\alpha<\gamma$. Hence, the fourth case of \eqref{E7} covers the range $(i<k;\ \alpha<\gamma;\ (\kg)<r^+)$. The fifth case includes the range $(i<k;\ \alpha<\gamma;\ (\kg)\ge r^+)$. 

Suppose now that $i<k<j$ and $\alpha\ge\beta>\gamma$. In this case, 
$$\{\trp_\ia,\trp_\kg\}= \{\tr_\ia, \tr_\kg\}+ \{\tr_\ia, \tr_\kb\} \tinv_\jb \tr_\jg- \tr_\kb t^{-2}_\jb \{\tr_\ia, t_\jb\} \tr_\jg+ \tr_\kb \tinv_\jb \{\tr_\ia, \tr_\jg\},$$
and we find that
\begin{multline}
\{\trp_\ia,\trp_\kg\}= \\
\begin{cases} 0+ (\tr_\ia \tr_\kb) \tinv_\jb \tr_\jg- \tr_\kb t^{-2}_\jb (\tr_\ia t_\jb) \tr_\jg+0 =0 &(\alpha=\beta) \\[1ex]
0+0-0+0=0 &(\alpha>\beta). \end{cases}
\label{E8}
\end{multline}
Equations \eqref{E7} and \eqref{E8} verify all cases of \eqref{E6} in which $i\ge j$ or $\alpha\ge \beta$.

For the remainder of the proof, we assume that $i<j$ and $\alpha<\beta$. Thus, equations \eqref{E2}--\eqref{E5} are applicable, and will be all we need when $k\ge j$ or $\gamma\ge\beta$. In cases where $k<j$ and $\gamma<\beta$, we have
\begin{multline}
\{\trp_\ia,\trp_\kg\}= \\
\{\trp_\ia, \tr_\kg\}+ \{\trp_\ia, \tr_\kb\} \tinv_\jb \tr_\jg- 2\tr_\kb t^{-2}_\jb \tr_\ib \tr_\ja \tr_\jg+ \tr_\kb \tinv_\jb \{\trp_\ia, \tr_\jg\},
\label{E9}
\end{multline}
because $\{\trp_\ia, \tr_\jb \}= 2\tr_\ib \tr_\ja$ by the fifth case of \eqref{E5}.

When $i=k$ and $\alpha<\beta\le\gamma$, we see by the second case of \eqref{E2} that
$$\{\trp_\ia,\trp_\kg\}= \{\trp_\ia,\tr_\kg\}= \trp_\ia \tr_\kg= \trp_\ia \trp_\kg.$$
When $i=k$ and $\alpha<\gamma<\beta$, we use \eqref{E9}, \eqref{E2}, \eqref{E5} to see that
\begin{align*}
\{\trp_\ia,\trp_\kg\} &= (\tr_\ia- \tr_\ib \tinv_\jb \tr_\ja) \tr_\kg+ (\trp_\ia \tr_\kb)\tinv_\jb \tr_\jg \\
 &\qquad -2\tr_\kb t^{-2}_\jb \tr_\ib \tr_\ja \tr_\jg+ \tr_\kb \tinv_\jb (2\tr_\ig \tr_\ja+ 2\tr_\ib \tinv_\jb \tr_\ja \tr_\jg) \\
 &=  (\tr_\ia+ \tr_\ib \tinv_\jb \tr_\ja) \tr_\kg+ \trp_\ia(\tr_\kb \tinv_\jb \tr_\jg)= \trp_\ia \trp_\kg.
 \end{align*}
This establishes the first case of \eqref{E6}.

The second case of \eqref{E6} is parallel to the first; we omit the details.

When $i<j\le k$ and $\gamma<\alpha<\beta$, we see by \eqref{E4} that
$$\{\trp_\ia,\trp_\kg\}= \{\trp_\ia,\tr_\kg\}= 0.$$
When $i<k<j$ and $\gamma<\alpha<\beta$, we use \eqref{E9}, \eqref{E4}, \eqref{E5} to see that
\begin{align*}
\{\trp_\ia,\trp_\kg\} &= (-2\tr_\ib \tinv_\jb \tr_\ka \tr_\jg)+ (2\tr_\ib \tr_\ka+ 2
\tr_\ib \tinv_\jb \tr_\ja \tr_\kb)\tinv_\jb \tr_\jg \\
 &\qquad - 2\tr_\kb t^{-2}_\jb \tr_\ib \tr_\ja \tr_\jg+0 \\
 &=0.
 \end{align*}
 This establishes the third case of \eqref{E6}.
 
If $i<k$, $\alpha<\gamma$, and $(\kg)>r$, we see by the last two cases of \eqref{E5} that
$$\{\trp_\ia,\trp_\kg\}= \{\trp_\ia,\tr_\kg\}= 0,$$
which establishes the fifth case of \eqref{E6}. If $i<k$, $\alpha<\gamma$, and $(\kg)=r$, we have, by the fifth case of \eqref{E5},
$$\{\trp_\ia,\trp_\kg\}= \{\trp_\ia,\tr_\kg\}= 2\tr_\ig \tr_\ka= 2\trp_\ig \trp_\ka.$$
It remains to deal with the cases when $i<k$, $\alpha<\gamma$, and $(\kg)<r$.

If $i<k<j$ and $\alpha<\gamma<\beta$, then by \eqref{E9} and \eqref{E5}, we get
\begin{align*}
\{\trp_\ia,\trp_\kg\} &= 2\tr_\ig \tr_\ka+ (2\tr_\ib \tr_\ka+ 2\tr_\ib \tinv_\jb \tr_\ja \tr_\kb) \tinv_\jb \tr_\jg \\
 &\qquad -2\tr_\kb t^{-2}_\jb \tr_\ib \tr_\ja \tr_\jg+ \tr_\kb \tinv_\jb (2\tr_\ig \tr_\ja+ 2\tr_\ib \tinv_\jb \tr_\ja \tr_\jg) \\
  &= 2\trp_\ig \trp_\ka.
  \end{align*}
If $i<k<j$ and $\alpha<\gamma=\beta$, we see by \eqref{E5} that
$$\{\trp_\ia,\trp_\kg\}= \{\trp_\ia,\tr_\kg\}= 2\tr_\ig \tr_\ka+ 2\tr_\ig \tinv_\jb \tr_\ja \tr_\kb= 2\trp_\ig \trp_\ka,$$
while if $i<k<j$ and $\alpha<\beta<\gamma$, we see by \eqref{E5} that
$$\{\trp_\ia,\trp_\kg\}= \{\trp_\ia,\tr_\kg\}= 2\tr_\ig (\tr_\ka+ \tr_\kb \tinv_\jb \tr_\ja) = 2\trp_\ig \trp_\ka.$$
Finally, if $i<k=j$ and $\alpha<\gamma<\beta$, then \eqref{E5} gives us
$$\{\trp_\ia,\trp_\kg\}= \{\trp_\ia,\tr_\kg\}= 2\tr_\ig \tr_\ja+ 2\tr_\ib \tinv_\jb \tr_\ja \tr_\jg= 2\trp_\ig \trp_\ka.$$
This verifies the fourth case of \eqref{E6}, and completes the proof of 5.1(4).
\hfill{$\square$}


\section{Number of tnn cells.}
\label{appendix}

The proof of Theorem \ref{TheoDescription} relies on a comparison of the number of nonempty tnn cells in $\mmptnn$ with the number of $m\times p$ Cauchon diagrams. These numbers are equal, as follows from Postnikov's work \cite{postnikov}. Our purpose in this appendix is to show how to obtain the key inequality (Corollary \ref{bound}) via our present methods. Equality then follows easily, as in Theorems \ref{theo:parametrisationnonemotycells} and \ref{TheoDescription}. 

\begin{lem} \label{tnnCauchon}
Every tnn matrix over $\bbR$ is a Cauchon matrix.
\end{lem}

\begin{proof} Let $X=(x_{i,\alpha})$ be a tnn matrix. Suppose that some $x_{\ia} =0$, and that $x_{k,\alpha} >0$ for some $k<i$. Let $\gamma<\alpha$. We need to prove that $x_{i,\gamma} =0$.
As $X$ is tnn, we have $-x_{k, \alpha} x_{i, \gamma} = \det \begin{pmatrix} x_{k, \gamma} & x_{k, \alpha} \\ x_{i, \gamma} & x_{i, \alpha} \end{pmatrix} \geq 0$. As $x_{k,\alpha} > 0$, this forces $x_{i,\gamma} \leq 0$. But since $X$ is tnn, we also have $x_{i,\gamma} \geq 0$, so that $x_{i,\gamma} = 0$, as desired. 
 
Therefore $X$ is Cauchon.
\end{proof}

We next give a detailed description of the deleting derivations algorithm, which is inverse to the restoration algorithm. In order to have matching notation for the steps of the two algorithms, we write the initial matrix for this algorithm in the form $\Xbar$.

\begin{conv}[Deleting derivations algorithm]
\label{conv2}$ $\\
Let $\Xbar=(\xbar_{\ia}) \in \MmpK$, where $K$ is a field of characteristic zero. 
As $r$ runs over the set $E$, we define matrices 
$X^{(r)} :=(x_{\ia}^{(r)}) \in \MmpK$ 
as follows:
\begin{enumerate}

\item \underline{When $r=(m,p+1)$}, we set $X^{(r)}=\Xbar$, that is, $x_{\ia}^{(m,p+1)}:=\xbar_{\ia}$ for all $(\ia) \in \gc 1,m \dc \times
\gc 1,p \dc$. 

\item \underline{Assume that $r=(j,\beta) \in \Ecirc$} and
that the matrix $X^{(r^+)}=(x_{\ia}^{(r^+)})$ is already known. The entries
$x_{\ia}^{(r)}$ of the matrix $X^{(r)}$ are defined as follows:

\begin{enumerate}

\item If $x_{j,\beta}^{(r^+)}=0$, then $x_{\ia}^{(r)}=x_{\ia}^{(r^+)}$ 
for all $(\ia) \in \gc 1,m \dc \times \gc 1,p \dc$.

\item If $x_{j,\beta}^{(r^+)}\neq 0$ and 
$(\ia) \in \gc 1,m \dc \times \gc 1,p
\dc$, then \\
$x_{\ia}^{(r)}= \begin{cases}
x_{\ia}^{(r^+)}-x_{i,\beta}^{(r^+)} \left( x_{j,\beta}^{(r^+)}\right)^{-1}
x_{j,\alpha}^{(r^+)} & \qquad \mbox{if }i <j \mbox{ and } \alpha < \beta \\
x_{\ia}^{(r^+)} & \qquad \mbox{otherwise.} \end{cases}$ 

\end{enumerate}

\item Set $X:= X^{(1,2)}$; this is the matrix obtained from $\Xbar$ at the end of the deleting derivations algorithm.

\item The matrices labelled $X^{(r)}$ in this algorithm are the same as the matrices with those labels obtained by applying the restoration algorithm to the matrix $X$ (cf.~Observations \ref{deleting}). Thus, the results of Subsections \ref{effect1} and \ref{effect2} are applicable to the steps of the deleting derivations algorithm.
\end{enumerate}
\end{conv}

In dealing with minors of the matrices $X^{(\jb)}$, we shall need the following variant of Proposition \ref{form0}, which is proved in the same manner. As in Notation \ref{notadet}, we will write minors of $X^{(\jb)}$ in the form $\delta^{(\jb)}$, this time viewing $\Xbar$ as the starting point.

\begin{lem} \label{form0lemmaBIS}
Let 
$\Xbar=(\xbar_{i,\alpha}) \in \MmpK$ be a matrix with entries in a field $K$ of characteristic $0$, and let 
$(j,\beta) \in \Ecirc$. 
Let $\delta=
[i_1,\dots,i_l|\alpha_1,\dots,\alpha_l](\Xbar)$ be a minor of $\Xbar$. Assume that $u:= x_{j,\beta} \neq 0$ and that $i_l < j$ while  $\alpha_h < \beta < \alpha_{h+1}$ for some $h \in \onel$. 
{\rm(}By convention, $\alpha_{l+1}=p+1$.{\rm)} Then
$$\delta^{(j,\beta)^+} = \delta^{(j,\beta)} +
  \sum_{t=1}^h (-1)^{t+h}\delta_{\alpha_t \rightarrow \beta}^{(j,\beta)}\, 
 x_{j,\alpha_t}^{(j,\beta)} u^{-1}.
$$
\end{lem}

\begin{proof} We proceed by induction on $l+1-h$. If $l+1-h=1$, then $h=l$ and $\alpha_l < \beta$.
It follows from Proposition~\ref{Pivot1} that 
\begin{eqnarray*}
\delta^{(j,\beta)}\, u & = & \det \left( 
\begin{array}{ccc}
x_{i_1,\alpha_1}^{(j,\beta)} & \dots & x_{i_1,\alpha_l}^{(j,\beta)} \\[1ex]
\vdots & & \vdots \\[1ex]
x_{i_l,\alpha_1}^{(j,\beta)} & \dots & x_{i_l,\alpha_l}^{(j,\beta)}
\end{array} \right) \times u  \\[2ex]
 & = & \det \left( 
\begin{array}{cccc}
x_{i_1,\alpha_1}^{(j,\beta)^+} & \dots & x_{i_1,\alpha_l}^{(j,\beta)^+} 
& x_{i_1,\beta}^{(j,\beta)^+} \\[1ex]
\vdots & & \vdots &  \vdots \\[1ex]
x_{i_l,\alpha_1}^{(j,\beta)^+} & \dots & x_{i_l,\alpha_l}^{(j,\beta)^+} 
& x_{i_l,\beta}^{(j,\beta)^+} \\[1ex]
x_{j,\alpha_1}^{(j,\beta)^+} & \dots & x_{j,\alpha_l}^{(j,\beta)^+} & u  \\
\end{array} \right).  \\
\end{eqnarray*}
Expanding this determinant along its last row leads to
\begin{eqnarray*}
\delta^{(j,\beta)}\, u & = & u\, \delta^{(j,\beta)^+} + 
\sum_{t=1}^l (-1)^{t+l+1}  x_{j,\alpha_t}^{(j,\beta)^+}
  \delta_{\alpha_t \rightarrow \beta}^{(j,\beta)^+}.
\end{eqnarray*}
By construction, 
$x_{j,\alpha_t}^{(j,\beta)^+}=x_{j,\alpha_t}^{(j,\beta)}$, 
and it follows from Proposition~\ref{Form0Start} that 
$\delta_{\alpha_t \rightarrow \beta}^{(j,\beta)^+}
=\delta_{\alpha_t \rightarrow \beta}^{(j,\beta)}$. 
Hence,   
 \begin{eqnarray*}
 \delta^{(j,\beta)^+} = \delta^{(j,\beta)}
  - \sum_{t=1}^l (-1)^{t+l+1} u^{-1} x_{j,\alpha_t}^{(j,\beta)}
  \delta_{\alpha_t \rightarrow \beta}^{(j,\beta)},
\end{eqnarray*}
as desired.

Now let $l+1-h>1$, and assume the result holds for smaller values of $l+1-h$. Expand the minor $\delta^{(j,\beta)^+} $ along its last column, to get
$$\delta^{(j,\beta)^+} = \sum_{k=1}^l (-1)^{k+l} x_{i_k,\alpha_l}^{(j,\beta)^+} \delta_{\widehat{i_k}, \widehat{\alpha_l}}^{(j,\beta)^+}.$$
The value corresponding to $l+1-h$ for the minors $\delta_{\widehat{i_k}, \widehat{\alpha_l}}^{(j,\beta)^+}$ is $l-h$, and so the induction hypothesis applies. We obtain
$$\delta_{\widehat{i_k}, \widehat{\alpha_l}}^{(j,\beta)^+}= \delta_{\widehat{i_k}, \widehat{\alpha_l}}^{(j,\beta)}+
\sum_{t=1}^h (-1)^{t+h}\delta_{\substack{\widehat{i_k}, \widehat{\alpha_l}\\ \alpha_t \rightarrow \beta}}^{(j,\beta)}\, 
 x_{j,\alpha_t}^{(j,\beta)} u^{-1}$$
for $k\in\onel$. As  
$x_{i_k,\alpha_l}^{(j,\beta)^+}
= x_{i_k,\alpha_l}^{(j,\beta)}$ by construction,  we obtain 
\begin{align*}
\delta^{(j,\beta)^+} &= \sum_{k=1}^l (-1)^{k+l} x_{i_k,\alpha_l}^{(j,\beta)} \delta_{\widehat{i_k}, \widehat{\alpha_l}}^{(j,\beta)} +
\sum_{t=1}^h \sum_{k=1}^l (-1)^{k+l} x_{i_k,\alpha_l}^{(j,\beta)}  (-1)^{t+h}\delta_{\substack{\widehat{i_k}, \widehat{\alpha_l}\\ \alpha_t \rightarrow \beta}}^{(j,\beta)}\, 
 x_{j,\alpha_t}^{(j,\beta)} u^{-1} \\
 &= \delta^{(j,\beta)}
  + \sum_{t=1}^h  (-1)^{t+h}\delta_{\alpha_t \rightarrow \beta}^{(j,\beta)}\, 
 x_{j,\alpha_t}^{(j,\beta)} u^{-1},
\end{align*}
by two final Laplace expansions. This concludes the induction step.
\end{proof}

\begin{theo}
\label{deletingderivationstnn}
Let $\Xbar=(\xbar_{i,\alpha}) \in \mc_{m,p}(\bbR)$ be a tnn matrix. 
We denote by $X^{(j,\beta)}$ the matrix obtained from $\Xbar$ at step $(j,\beta)$ of the deleting derivations 
algorithm. 
\begin{enumerate}
\item All the entries of $X^{(j,\beta)}$ are nonnegative.
\item $X^{(j,\beta)}$ is a Cauchon matrix.
\item The matrix obtained from $X^{(j,\beta)}$ by deleting the rows $j+1, \dots , m$ and the columns $ \beta , \beta +1, \dots, p$ is tnn.
\item The matrix obtained from $X^{(j,\beta)}$ by deleting the rows $ j , j +1, \dots, m$ is tnn.
\item Let $\delta=[i_1,\dots,i_l |\alpha_1, \dots,\alpha_l](\Xbar)$ be a
minor of $\Xbar$ with $(i_l,\alpha_l) < (j,\beta)$. If $ \delta ^{(j,\beta)^+}=0$, then $\delta^{(j,\beta)}=0$.
\end{enumerate}
\end{theo}

\begin{proof}
We prove this theorem by a decreasing induction on $(j,\beta)$. 

If $(j,\beta)=(m,p+1)$, then 1, 3, 4 hold by hypothesis, 2 holds by Lemma \ref{tnnCauchon}, and 5 is vacuous. 

Assume now that $(j,\beta)\in \Ecirc$ and the result is true for $X^{(j,\beta)^+}$. 

$\bullet$ Let us first prove 1. Let $x_{\ia}^{(j,\beta)}$ be an entry of $X^{(j,\beta)}$. We distinguish between two cases. 
First, if $x_{\ia}^{(j,\beta)}=x_{\ia}^{(j,\beta)^+}$, then it follows from the induction hypothesis that $x_{\ia}^{(j,\beta)}$ is nonnegative, as 
desired. Next, if $x_{\ia}^{(j,\beta)}\neq x_{\ia}^{(j,\beta)^+}$, then $i<j$ and $\alpha<\beta$. Moreover, it follows from the construction of the algorithm and the induction hypothesis 
that $x_{j,\beta}^{(j,\beta)^+} > 0$, and $x_{\ia}^{(j,\beta)} = \det \left( \begin{array}{ll} x_{\ia}^{(j,\beta)^+} & x_{i, \beta}^{(j,\beta)^+} \\
x_{j, \alpha}^{(j,\beta)^+} & x_{j, \beta}^{(j,\beta)^+} \end{array} \right) (x_{j,\beta}^{(j,\beta)^+})^{-1}$. Note that if $\beta<p$, then $(j,\beta)^+= (j,\beta+1)$, while if $\beta=p$, then $(j,\beta)^+= (j+1,1)$. Hence, by the induction hypothesis (3 or 4), the previous determinant is nonnegative, so that $x_{\ia}^{(j,\beta)} \geq 0$, as desired.

$\bullet$ Let us now prove 3. Let $\delta=[i_1,\dots,i_l | \alpha_1, \dots,\alpha_l](\Xbar)$ be a minor of $\Xbar$ with 
$i_l \leq j$ and $\alpha_l < \beta$. We need to prove that   $\delta^{(j,\beta)}$ is nonnegative. Set $u:=x_{j,\beta}^{(j,\beta)}=  x_{j,\beta}^{(j,\beta)^+}$.

First, if $i_l=j$ or $u=0$, then it follows from Proposition \ref{Form0Start} that $\delta^{(j,\beta)}=\delta^{(j,\beta)^+}$, and so we deduce from the induction hypothesis (3 or 4) that $\delta^{(j,\beta)}=\delta^{(j,\beta)^+} \geq 0$, as desired.

Now assume that $i_l<j$ and $u\neq 0$. Then it follows from the induction hypothesis (3 or 4) that 
$u > 0$. Moreover, we deduce from Proposition \ref{Pivot1} that 
$\delta^{(j,\beta)} = \det \left( x_{i,\alpha}^{(j,\beta)^+} \right)_{\substack{i=i_1,\dots,i_l,j \\ \alpha=\alpha_1, \dots,\alpha_l,\beta}} u^{-1}$. 
By the induction hypothesis, the above determinant is nonnegative, so that $\delta^{(j,\beta)} \geq 0$, as claimed.\\

$\bullet$ Let us now prove 4. We will prove by induction on $l$ that all the minors of the form $\delta^{(j,\beta)} =\det \left( x_{i,\alpha}^{(j,\beta)} \right)_{\substack{i=i_1,\dots,i_l \\ \alpha=\alpha_1, \dots,\alpha_l}} $ with $i_l < j$ are nonnegative. 

The case $l=1$ is a consequence of 1. So, we assume $l \geq 2$. Set $u:=x_{j,\beta}^{(j,\beta)}=x_{j,\beta}^{(j,\beta)^+}$.
If $\alpha_l < \beta$, then it follows from point 3 that $\delta^{(j,\beta)} \geq 0$. 
Next, if $u=0$, or if $\beta\in \{\alpha_1,\dots,\alpha_l\}$, or if $\beta< \alpha_1$, then it follows from Proposition \ref{Form0Start} that $\delta^{(j,\beta)}=\delta^{(j,\beta)^+}$, and so  it follows from the induction hypothesis (4) that $\delta^{(j,\beta)}=\delta^{(j,\beta)^+} \geq 0$, as desired. 

So, it just remains to consider the case where $u \ne0$ and 
there exists $h \in \gc 1,  l-1\dc $ such that $\alpha_h < \beta < \alpha_{h+1}$. 

In order to simplify the notation, we set $[I | \Lambda]=[I | \Lambda](X^{(j,\beta)})$ and $[I | \Lambda]^+=[I | \Lambda](X^{(j,\beta)^+})$ for the remainder of the proof of 4, for any index sets $I$ and $\Lambda$.

For all $k \in \gc 1, l\dc$, we set $I_k:=\{i_1 , \dots, \widehat{i_k}, \dots , i_l\}$. We also set $I:=\{i_1 , \dots , i_l\}$ 
and $\Lambda_t:=\{\alpha_1 , \dots, \widehat{\alpha_t}, \dots , \alpha_{l-1}\}$ for all $t \in \gc 1 , h \dc$. 

For all $k \in \gc 1 , l\dc$, it follows from \cite[(2.10)]{Ando} 
and Muir's law of extensible minors \cite[p 179]{Mui} that 
\begin{eqnarray}
[I_k | \Lambda_t \cup \{\beta\}] \; [I | \Lambda_t\cup\{\alpha_t , \alpha_l\}] & = & 
[I_k | \Lambda_t\cup\{\alpha_l\}] \; [I | \Lambda_t \cup\{\alpha_t , \beta\}] \nonumber \\ 
& & +[I_k | \Lambda_t\cup\{\alpha_t\}] \; [I | \Lambda_t\cup\{\beta , \alpha_l\}].   \label{plucker}
\end{eqnarray}
Recall that $\delta^{(j,\beta)}= [I | \Lambda_t\cup\{\alpha_t , \alpha_l \}]$ for all $t$.

It follows from the secondary induction hypothesis (on the size of the minors) that $[I_k | \Lambda_t\cup\{\alpha_l\}] \geq 0$ 
and $[I_k | \Lambda_t\cup\{\alpha_t\}] \geq 0$. Moreover it follows from Proposition \ref{Form0Start} that 
$[I | \Lambda_t \cup\{\alpha_t , \beta\}]=[I | \Lambda_t \cup\{\alpha_t , \beta\}]^+$ and 
$[I | \Lambda_t\cup\{\beta , \alpha_l\}]=[I | \Lambda_t\cup\{\beta , \alpha_l\}]^+$, and so we deduce from the primary induction 
hypothesis that these minors are nonnegative. All these inequalities together show that the right-hand side of equation \eqref{plucker} is nonnegative, that is, 
for all $k \in \gc 1, l\dc$ and $t\in \gc 1,h \dc$, we have
\begin{equation}  \label{secondpoint4eqn}
[I_k | \Lambda_t \cup \{\beta\}] \; [I | \Lambda_t\cup\{\alpha_t , \alpha_l \}] \geq 0.
\end{equation}

From the secondary induction hypothesis, we know that $[I_k | \Lambda_t \cup \{\beta\}] \geq 0$ for all $k,t$. 
We need to prove that $\delta^{(j,\beta)}=[I | \Lambda_t\cup\{\alpha_t , \alpha_l \}]$ is nonnegative. 

If there exist $k$ and $t$ such that $[I_k | \Lambda_t \cup \{\beta\}] > 0$, then it follows from the inequality \eqref{secondpoint4eqn} that 
$[I | \Lambda_t\cup\{\alpha_t , \alpha_l \}] \geq 0$, as desired. 

Finally, we assume that for all $k$ and $t$ we have $[I_k | \Lambda_t \cup \{\beta\}] = 0$. In this case, it follows from a Laplace expansion that 
$[I | \Lambda_t \cup \{\alpha_l,\beta\}] = 0$. In other words, we have 
$\delta_{\alpha_t \rightarrow \beta}^{(j,\beta)}=0$ for all $t \in \gc 1, h \dc$. Hence, we deduce from Lemma \ref{form0lemmaBIS} that $\delta^{(j,\beta)}=\delta^{(j,\beta)^+}$. As $\delta^{(j,\beta)^+} \geq 0$ by the induction hypothesis, we get $\delta^{(j,\beta)} \geq 0$, as desired. This completes the induction step for the proof of 4.\\

$\bullet$ Let us now prove 2. Assume that $x_{\ia}^{(j,\beta)}=0$ for some $(\ia)$. We must prove that $x^{(j,\beta)}_{k,\alpha} =0$ for all $k\le i$ or $x^{(j,\beta)}_{i,\lambda} =0$ for all $\lambda\le\alpha$. We distinguish between several cases.

$\bullet\bullet$ Assume that $i < j$. Then by 4, the matrix obtained from $X^{(j,\beta)}$ by deleting the rows $i+1,\dots,m$ is tnn. This matrix is Cauchon by Lemma \ref{tnnCauchon}, and our desired conclusion follows.

$\bullet\bullet$ Assume that $i \geq j$. Hence, by construction, we have $x_{\ia}^{(j,\beta)^+}=x_{\ia}^{(j,\beta)}=0$. So we deduce from the 
induction hypothesis that $x_{k,\alpha}^{(j,\beta)^+}=0$ for all $k \leq i$ or $x_{i, \lambda}^{(j,\beta)^+}=0$ for all $\lambda \leq \alpha$. 

Assume first that $x_{i, \lambda}^{(j,\beta)^+}=0$ for all $\lambda \leq \alpha$. As $i \geq j$, we get $x_{i, \lambda}^{(j,\beta)}=x_{i, \lambda}^{(j,\beta)^+}=0$ for all $\lambda \leq \alpha$, as desired.

Assume next that $x_{k,\alpha}^{(j,\beta)^+}=0$ for all $k \leq i$. Then for all $j \leq k \leq i$, we get $x_{k,\alpha}^{(j,\beta)}=x_{k,\alpha}^{(j,\beta)^+}=0$. So it just remains to prove that $x_{k,\alpha}^{(j,\beta)}=0$ for all $k < j$. 
Let $k < j$. 
We distinguish between two cases. First, if $x_{j,\beta}^{(j,\beta)}=x_{j,\beta}^{(j,\beta)^+}=0$, then $x_{k,\alpha}^{(j,\beta)}=x_{k,\alpha}^{(j,\beta)^+}=0$, as expected. 
Otherwise, $u:=x_{j,\beta}^{(j,\beta)}=x_{j,\beta}^{(j,\beta)^+} \neq 0$, and 
$$x_{k,\alpha}^{(j,\beta)}=x_{k,\alpha}^{(j,\beta)^+} - x_{k,\beta}^{(j,\beta)^+} u^{-1} x_{j,\alpha}^{(j,\beta)^+}.$$
As $x_{k,\alpha}^{(j,\beta)^+}=0$ and $x_{j,\alpha}^{(j,\beta)^+}=0$, we get $x_{k,\alpha}^{(j,\beta)}=0$, as desired.\\

$\bullet$ Finally, let us prove 5 by induction on $l$. For the case $l=1$, assume we have $x^{(j,\beta)^+}_{i,\alpha} =0$ with $(i,\alpha)< (j,\beta)$. If  $x_{\ia}^{(j,\beta)}=x_{\ia}^{(j,\beta)^+}$, then clearly $x_{\ia}^{(j,\beta)}=0$. Otherwise, we have $i<j$ and $\alpha<\beta$, while $u :=x_{j,\beta}^{(j,\beta)^+} \neq 0$ and 
\begin{equation}  \label{proof5eqn}
x_{i,\alpha}^{(j,\beta)}= - x_{i,\beta}^{(j,\beta)^+} u^{-1} x_{j,\alpha}^{(j,\beta)^+}.
\end{equation}
By points 3 and 4, the numbers $x^{(j,\beta)}_{i,\alpha}$, $x_{i,\beta}^{(j,\beta)^+}$, $x_{j,\alpha}^{(j,\beta)^+}$ are all nonnegative, and $u>0$. Hence, we see by equation \eqref{proof5eqn} that $x_{i,\alpha}^{(j,\beta)}=0$, as desired.

Assume $l > 1$.  Let $\delta=[i_1,\dots,i_l | \alpha_1, \dots,\alpha_l](\Xbar)$ be a
minor of $\Xbar$ with $(i_l,\alpha_l) < (j,\beta)$. We assume that $ \delta ^{(j,\beta)^+}=0$. 
If $ \delta ^{(j,\beta)^+}=\delta ^{(j,\beta)}$, then clearly $\delta^{(j,\beta)}=0$. 
Hence, we deduce from Proposition \ref{Form0Start} that we can assume that $u:=x_{j,\beta}^{(j,\beta)}=x_{j,\beta}^{(j,\beta)^+} \neq 0$, 
$i_l < j$ and there exists $h \in \gc 1, l\dc$ such that $\alpha_h < \beta < \alpha_{h+1}$. 

We distinguish between two cases to prove that $\delta^{(j,\beta)}=0$. 
First we assume that $\alpha_l < \beta$. Let $A$ be the matrix obtained from $X^{(j,\beta)^+}$ by retaining the rows $i_1, \dots, i_l, j$ and the columns 
$\alpha_1 , \dots , \alpha_l, \beta$. It follows from 3 or 4 that $A$ is tnn. Moreover, $A$ has a principal minor, $ \delta ^{(j,\beta)^+}$, which is equal to zero. Hence, we deduce from \cite[Corollary 3.8]{Ando} that $\det (A)=0$, that is, 
$$   \det \left( x_{i,\alpha}^{(j,\beta)^+}
\right)_{\substack{i=i_1,\dots,i_l,j \\ \alpha=\alpha_1, \dots,\alpha_l,\beta}}=0.$$ 
Then, as we have assumed that $u \neq 0$, it follows from Corollary \ref{cor:part1} that 
$$[i_1,\dots,i_l | \alpha_1, \dots,\alpha_l](X^{(j,\beta)})=0,$$
that is,  $\delta^{(j,\beta)}=0$ as desired.

Next, we assume that there exists $h \in \gc 1, l-1\dc$ such that $\alpha_h < \beta < \alpha_{h+1}$. 
With notations similar to those of equation (\ref{plucker}), we have
\begin{eqnarray*}
0=[I_k | \Lambda_t \cup \{\beta\}]^+ \; [I | \Lambda_t\cup\{\alpha_t , \alpha_l \}]^+ & = & 
[I_k | \Lambda_t\cup\{\alpha_l\}]^+ \; [I | \Lambda_t \cup\{\alpha_t , \beta\}]^+ \\ 
& & +[I_k | \Lambda_t\cup\{\alpha_t\}]^+ \; [I | \Lambda\cup\{\beta , \alpha_l\}]^+
\end{eqnarray*}
for all $k \in \gc 1,  l\dc$ and $t\in \gc 1,h \dc$. As all these minors are nonnegative by 4, we get that
$$[I_k | \Lambda_t\cup\{\alpha_l\}]^+ \; [I | \Lambda_t \cup\{\alpha_t , \beta\}]^+= [I_k | \Lambda_t\cup\{\alpha_t\}]^+ \; [I | \Lambda\cup\{\beta , \alpha_l\}]^+ =0$$
for all $k$, $t$. Now we deduce from the secondary induction (on the size of the  minors) and Proposition \ref{Form0Start} that
$$[I_k | \Lambda_t\cup\{\alpha_l\}] \; [I | \Lambda_t \cup\{\alpha_t , \beta\}]= [I_k | \Lambda_t\cup\{\alpha_t\}] \; [I | \Lambda\cup\{\beta , \alpha_l\}] =0$$
for all $k$, $t$. Hence, by equation \eqref{plucker},
$$[I_k | \Lambda_t\cup\{\beta\}]\; \delta^{(j,\beta)}= [I_k | \Lambda_t\cup\{\beta\}]\; [I | \Lambda_t\cup\{\alpha_t , \alpha_l \}]=0$$
for all $k$, $t$.

If $[I_k | \Lambda_t\cup\{\beta\}] \ne 0$ for some $k$, $t$, then $\delta^{(j,\beta)}= 0$. If $[I_k | \Lambda_t\cup\{\beta\}] =0$ for all $k$, $t$, then it follows from a Laplace expansion that $[I | \Lambda_t \cup\{\alpha_l , \beta\}] =0$ for all $t$. In this case (as in the proof of 4), it follows from Lemma  \ref{form0lemmaBIS} that $\delta^{(j,\beta)}= \delta^{(j,\beta)^+}= 0$. This completes the induction step for 5.
\end{proof}

At the end of the deleting derivations algorithm, we get the following result, which provides a converse to Theorem \ref{Nbarnonneg}.

\begin{cor}
Let $\Xbar=(\xbar_{i,\alpha}) \in \mc_{m,p}(\bbR)$ be a tnn matrix. 
We set $X:=X^{(1,2)}$, the matrix obtained at step $(1,2)$ of the deleting derivations 
algorithm. Then $X$ is a nonnegative $\hc$-invariant Cauchon matrix.
\end{cor}

\begin{cor}
\label{bound}
The number of nonempty tnn cells in $\mc_{m,p}^{\ge0}(\bbR)$ is less than or equal to the number of $m\times p$ Cauchon diagrams.
\end{cor}

\begin{proof}
If $\Xbar \in \mmptnn$, then it follows from the previous corollary that 
the matrix $X= (x_{i,\alpha})$ obtained at step $(1,2)$ of the deleting derivations algorithm is a nonnegative $\hc$-invariant Cauchon matrix. Let $C:=\{(i,\alpha) \mid x_{i,\alpha}=0\}$. As $X$ is Cauchon, $C$ is a Cauchon diagram. So we have a mapping $\pi: \Xbar \mapsto C$ from $\mmptnn$ to the set $\mathcal{C}_{m,p}$ of $m\times p$ Cauchon diagrams.  Now, let $\Xbar$ and $\Ybar$ be two $m\times p$ tnn matrices. If $\pi(\Xbar)=\pi(\Ybar)$, then it follows from Corollary \ref{cor:independence} that $\Xbar$ and $\Ybar$ belong to the same tnn cell. 
So, each nonempty tnn cell in $\mmptnn$ is a union of fibres of $\pi$. The corollary follows.
\end{proof}

\section*{Acknowledgements}

The results in this paper were announced during the mini-workshop
``Non-neg\-a\-tiv\-i\-ty is a quantum phenomenon'' that took place at the
Mathematisches For\-schungs\-in\-sti\-tut Oberwolfach, 1--7 March 2009,
\cite{MFO}; we thank the director and staff of the MFO for providing
the ideal environment for this stimulating meeting. We also
thank Konni Rietsch, Laurent Rigal, Lauren Williams and Milen Yakimov for discussions
and comments concerning this paper both at the workshop and at other
times.


\bibliographystyle{amsplain}
\bibliography{biblio}

\vskip 1cm

K.R. Goodearl:

Department of Mathematics,

University of California,

Santa Barbara, CA 93106, USA

Email: {\tt goodearl@math.ucsb.edu} 

\medskip

S. Launois: 

Institute of Mathematics, Statistics and Actuarial Science,

University of Kent

Canterbury, Kent, CT2 7NF, UK

Email: {\tt S.Launois@kent.ac.uk}

\medskip

T.H. Lenagan: 

Maxwell Institute for Mathematical Sciences,

School of Mathematics, University of Edinburgh,

James Clerk Maxwell Building, King's Buildings, Mayfield Road,

Edinburgh EH9 3JZ, Scotland, UK

E-mail: {\tt tom@maths.ed.ac.uk}
\end{document}